\documentclass[11pt,twoside,a4paper,reqno]{amsart}
\usepackage[utf8]{inputenc}
\usepackage{amsmath}
\usepackage{amssymb}
\usepackage{amsthm}
\usepackage[english]{babel}
\usepackage{tikz-cd}
\usepackage{url}
\usepackage{enumerate}
\usepackage[normalem]{ulem}
\usepackage{parskip}
\usepackage{color,xcolor}
\usepackage{hyperref}
\usepackage{bm}
\newtheorem{theorem}{Theorem}
\newtheorem{corollary}[theorem]{Corollary}
\newtheorem{definition}[theorem]{Definition}

\newtheorem{lemma}[theorem]{Lemma}
\newtheorem{remark}[theorem]{Remark}
\newtheorem{proposition}[theorem]{Proposition}

\newcommand{\supp}{\operatorname{supp}}
\newcommand{\Ord}{\operatorname{Ord}}
\newcommand{\Res}{\operatorname{Res}}
\newcommand{\Var}{\operatorname{Var}}

\newcommand{\as}[1]{\text{as } #1 \rightarrow\infty}
\renewcommand{\d}{\, \text{d}}

\newcommand{\s}[2]{(#1_#2)_{#2 \geq 1}}

\usepackage{xcolor}

\newcommand{\M}{\mathcal{M}}
\newcommand{\R}{\mathbb{R}}

\newcommand{\N}{\mathbb{N}}
\renewcommand{\P}{\mathbb{P}}
\newcommand{\E}{\mathbb{E}}

\newcommand{\bea}{\begin{eqnarray}}
\newcommand{\eea}{\end{eqnarray}}
\newcommand{\Bea}{\begin{eqnarray*}}
\newcommand{\Eea}{\end{eqnarray*}}
\def\one{{\mathbf{1}}}
\DeclareMathOperator*{\argmin}{arg\,min}
\DeclareRobustCommand{\stirling}{\genfrac\{\}{0pt}{}}

\usepackage{todonotes}

\begin{document}

\title[Trimmed strong laws and distributional limits]
{Trimmed strong laws and distributional limits for exponentially mixing systems}

\author{M. Auer}
\address[M.~Auer]{School of Mathematics and Physics \\ University of Queensland, Australia.}
\email{\href{mailto:m.auer@uq.edu.au}{m.auer@uq.edu.au}}
\author{S. Liu}
\address[S.~Liu]{Beijing Institute of Mathematical Sciences and Applications\\ Beijing 101408, China.}
\email{\href{mailto:liusixu@bimsa.cn}{liusixu@bimsa.cn}}

\thanks{The authors would like to thank Dmitry Dolgopyat for helpful discussions and the University of Maryland, where this project was initiated, for its hospitality. During this project, M. Auer was partially supported by the Austrian Science Fund (FWF), P 33943-N and by the Australian Research Council (FT240100592). S. Liu was partially supported by National Natural
Science Foundation of China No.~12501237 and Natural Science Foundation of Beijing No.~1244041.}

\date{}
\subjclass{Primary: 37A50, 60F15, 60G70; Secondary: 37C40.}
\keywords{ergodic theory, trimmed sum, law of large numbers, almost sure convergence}

\begingroup
\leftskip4em
\rightskip\leftskip
\begin{small}

\begin{abstract}
The Birkhoff Ergodic Theorem establishes pointwise convergence for integrable observables, but for $f\notin L^1$, no normalization yields almost sure convergence. This paper investigates trimmed ergodic sums, where the largest observations are removed, for observables with polynomial tails satisfying $\P(f>t)\,t^{1/\alpha}\to c>0$ as $t\to\infty$, in exponentially mixing dynamical systems. We prove trimmed strong laws of large numbers when $\alpha\geq 1$, extending known results from the i.i.d.\ case. We further establish distributional limit theorems for both lightly and intermediately trimmed sums in the regime $\alpha>1/2$, showing convergence to an explicitly described non-standard law and to a normal distribution, respectively. The proofs rely on approximating trimmed sums by truncated ergodic sums and exploiting the system's exponential mixing properties.
\end{abstract}

\end{small}
\maketitle
\par
\endgroup
\part{Results}

\section{Introduction}
A cornerstone of modern ergodic theory is the Birkhoff Ergodic Theorem, which asserts that for an ergodic transformation $T$ and an integrable observable $f,$ the time averages converge almost surely to the spatial average. More precisely,
\begin{equation*}
\lim_{N\to\infty} \frac{S_N(f)(x)}{N} = \E f \quad \text{for almost every } x,
\end{equation*}
where
$$
S_N(f)(x) = \sum_{n=0}^{N-1} f(T^n x).
$$
A natural question is what happens for a non-negative observable $f$ that is not integrable. One expects $S_N(f)$ to grow faster than linearly, that is, faster than a constant multiple of $N$. Nevertheless, one might hope that, after renormalization by a suitable sequence $d_N \to \infty$, pointwise convergence still holds. However, Aaronson~\cite{aaronsoninf} showed that such convergence to a nontrivial limit is impossible: for any non-negative $f\notin L^1$ and any positive normalizing sequence $d_N$,
\begin{equation}\label{aaronsonnoncon}
\limsup_{N\to\infty} \frac{S_N(f)(x)}{d_N} = \infty
\quad \text{or} \quad
\liminf_{N\to\infty} \frac{S_N(f)(x)}{d_N} = 0
\quad \text{for almost every } x.
\end{equation}

A heuristic explanation comes from order statistics: in the non-integrable case, the largest few terms typically dominate $S_N(f)$. Even a small shift in the position of an extreme value can therefore drastically change the sum, accounting for the oscillations in \eqref{aaronsonnoncon}. This motivates the study of \emph{trimmed sums}, where the largest terms are removed. For $N \geq 1$ and $0 \leq k(N) \leq N-1$, we define
\begin{equation}\label{TrimSum}
S_N^{k(N)}(f)(x) = \sum_{n=0}^{N-1} f(T^n x) - \sum_{i=1}^{k(N)} f(x_i^N),
\end{equation}
where $\{x_1^N, \dots, x_N^N\}$ is a reordering of the orbit points $x, T(x), \dots, T^{N-1}(x)$ such that
\begin{equation*}
f(x_1^N) \geq f(x_2^N) \geq \dots \geq f(x_N^N).
\end{equation*}
In the case $k(N)=0$, the sum in \eqref{TrimSum} is simply the usual Birkhoff sum $S_N(f)$, which we refer to as the \emph{untrimmed sum}. Trimmed sums originate in classical probability theory, with contributions from Feller~\cite{F68}, Kesten~\cite{kestenweak,kesten1992ratios}, Mori~\cite{morilighttrim,mori2}, and others.

To illustrate the concept, consider the function $f(x)=1/x$ on the unit interval with a singularity at the origin. The largest values of $f$ occur at the orbit points $x_1^N,\dots,x_{k(N)}^N$ that lie nearest to the origin, and the trimmed sum is obtained by summing over the remaining points. However, this example also reveals a key conceptual nuance: if $f$ is not monotonic, even when its only singularity remains at the origin, the points $x_1^N,\dots,x_{k(N)}^N$ need not be those closest to the singularity. This subtlety will play an important role in the subsequent analysis, and we revisit it in Section~\ref{nearsec}.

We begin by establishing a \emph{trimmed strong law of large numbers (SLLN)}: we seek sequences $d_N$ and $k(N)$ such that
\begin{equation}\label{trimsumcon}
\lim_{N\to\infty} \frac{S_N^{k(N)}(f)(x)}{d_N}=1
\quad \text{for almost every } x.
\end{equation}
The case where $k(N)=K>0$ is constant is referred to as \emph{light trimming}, while $k(N) \to \infty$ with $k(N) = o(N)$ is called \emph{intermediate trimming}. Trimming a positive fraction of terms, i.e.\ $k(N)/N \to c>0$ as $N \to \infty$, is uninteresting from our perspective, since it merely amounts to truncating $f$, reducing the problem to the classical SLLN.

In the most general sense, a trimmed SLLN always holds: for any non-negative $f$, there exist sequences $d_N$ and $k(N)$ such that \eqref{trimsumcon} is satisfied~(see \cite{AS25}). While conceptually important, this result is of limited practical use, since it provides no control over the choice of $d_N$ or, more crucially, $k(N)$. We therefore focus on functions with \emph{polynomial singularities}, informally satisfying
\begin{equation}\label{tlno}
\P(f>t)\,t^{\frac{1}{\alpha}}\to 1,\quad t\to\infty\quad \text{for some } \alpha >0.
\end{equation}
A broader definition appears in Definition~\ref{PowSing}. The functions $f$ are characterized by a singularity order $\beta>0,$ with $\alpha = \beta/d,$ where $d$ is the manifold dimension. For example, on $[0,1]$, consider $f(x) = x^{-\alpha}.$

We briefly recall the corresponding i.i.d.\ theory, which serves as the probabilistic model for our results for exponentially mixing systems. A detailed account of related results can be found in \cite{KessSchiSLLN}.
Let $\{X_N\}_{N \geq 1}$ be a sequence of nonnegative i.i.d.\ random variables satisfying
\begin{equation*}
\P(X_1>t)\,t^{\frac{1}{\alpha}}\to 1,\quad t\to\infty\quad \text{for some }\alpha>0.
\end{equation*}
For $k \geq 1$, let $X_N^{(k)}$ denote the $k$-th largest value among $X_1, \ldots, X_N$. For $0 \le k(N) \le N-1$, set
$$
S_N^{k(N)} = \sum_{n=1}^N X_n - \sum_{i=1}^{k(N)} X_N^{(i)}.
$$
The classical results for such trimmed sums are as follows:
\begin{itemize}
\item If $\alpha = 1$, then
\begin{equation}\label{iid1xcon}
\lim_{N\to\infty}\frac{S_N^K}{N\log N}=1
\quad\text{almost surely}.
\end{equation}
This lightly trimmed strong law with a fixed number $K>0$ of trimmed terms was established by Mori~\cite{mori2}, with related results on trimmed sums and order statistics obtained in Mori~\cite{morilighttrim} and Kesten--Maller~\cite{kesten1992ratios,km2}.
\item If $\alpha>1$, then
\begin{equation}\label{iidbetacon}
\lim_{N\to\infty} \frac{S_N^{k(N)}}{N^{\alpha} k(N)^{1-\alpha}}
= \frac{1}{\alpha - 1}
\quad\text{almost surely}.
\end{equation}
This strong law holds provided that $k(N)/\log\log N \to \infty$,  as follows from the law of the iterated logarithm of Haeusler--Mason~\cite{Haeusler_Mason_1987}. In contrast, if $k(N)/\log\log N \to c$ for some $c\in(0,\infty)$, then \eqref{iidbetacon} fails, by the non-standard law of the iterated logarithm proved by Haeusler~\cite{HaeuslernonstandardLIL}.
\end{itemize}
Beyond regular variation, Cs\"{o}rg\H{o}--Simons~\cite{CS96} obtained corresponding strong law results under intermediate trimming for tails satisfying sandwich-type regular variation, including St. Petersburg-type distributions as a notable special case.

It is widely expected that \eqref{iid1xcon} and \eqref{iidbetacon} extend to sufficiently well-behaved dynamical systems. Indeed, the following dynamical counterparts are known under the renormalization for which \eqref{tlno} holds:

\begin{itemize}
\item If $\alpha = 1$, then a dynamical analogue of \eqref{iid1xcon} has been established for $\psi$-mixing sequences \cite{AARONSONpsi}, and in the Gauss map setting, by Diamond--Vaaler \cite{DIAMOND1986} for the continued-fraction observable $f(x)=1/x$. For the same observable over the doubling map, Schindler~\cite{Sch18} showed that light trimming fails, whereas the same law holds after removing about $C\log\log\log N$ largest terms for $C>1/\log 2$. Light trimming fails because of the clustering of large values near the singularity at $0$, which is a fixed point of the doubling map, so removing finitely many largest terms still leaves substantial residual contributions, a mechanism already reflected in Haynes~\cite[Theorem~4]{Haynes14}.

\item If $\alpha > 1$, then \cite{KessSchiSLLN} showed a dynamical analogue of \eqref{iidbetacon} for systems with a spectral gap on a suitable Banach space and observables whose truncations have at most linear growth and whose level-set indicators remain uniformly bounded in the Banach space norm (see also Remark \ref{specrem}). This includes piecewise expanding interval maps with observables whose truncations and level-set indicators have uniformly controlled bounded variation. The result was later extended to subshifts of finite type via a quasi-H\"older framework \cite{KS20}.
\end{itemize}

While these results are well established for i.i.d.\ sequences and certain specific dynamical systems, their validity for general exponentially mixing systems remains an open question. In this paper, we address this gap by establishing trimmed strong laws of large numbers for such systems. For $\alpha=1$, we prove that a dynamical analogue of \eqref{iid1xcon} holds in full. For $\alpha>1$, we prove {a dynamical analogue of \eqref{iidbetacon}} under the trimming condition $k(N)/(\log N)^{\epsilon} \to \infty$ for some $\epsilon > 0$, which is close to the optimal $k(N)/\log\log N \to \infty$ condition known for i.i.d.\ sequences.

Furthermore, we advance the theory by proving the corresponding distributional limit theorems. Once a trimmed SLLN is established, the natural next question concerns the size of second-order deviations. In this context, an untrimmed analogue arises from the distributional limit theorems of Gnedenko--Kolmogorov~\cite{gnedenko1968limit} for observables not in $L^2$. They show that, in the i.i.d.\ setting, if $f$ has regularly varying tails, which is similar to our assumption of polynomial singularities, then there are sequences $a_N$ and $b_N$ such that
\begin{equation*}
\frac{S_N(f) - a_N}{b_N} \Rightarrow L \quad \as{N},
\end{equation*}
for a stable distribution $L$, where $\Rightarrow$ denotes distributional convergence.

Early work on the effect of extreme terms on heavy-tailed sums includes Darling~\cite{Dar52} and Arov--Bobrov~\cite{AB60}. For lightly trimmed sums, Hall~\cite{Hal78} established distributional limits in the stable-domain setting. In particular, for fixed $K\geq 1$, there exist sequences $a_N$ and $b_N$ such that
\begin{equation}\label{lightDLT}
\frac{S_N^K-a_N}{b_N} \Rightarrow L \quad \text{as } N \to \infty,
\end{equation}
where $L$ is a non-degenerate limiting distribution. As shown by Kesten~\cite{kestenweak}, the normalizing sequences $a_N$ and $b_N$ do not depend on $K$, whereas the limiting distribution $L$ does. Mori~\cite{Mori84} further developed the distributional theory of lightly trimmed sums, deriving in particular an explicit formula for the limiting distribution $L$.

The intermediately trimmed case, where $k(N)\rightarrow\infty$ but $k(N)=o(N)$, is of special interest because the limiting distribution is always normal. Cs\"{o}rg\H{o}--Horv\'{a}th--Mason~\cite{CHMinter} show that, for $k(N)\rightarrow\infty$ with $k(N)=o(N)$, there exist sequences $a_N$ and $b_N$ such that
\begin{equation}\label{interDLT}
\frac{S_N^{k(N)} - a_N}{b_N} \Rightarrow \mathcal{N}(0,1) \quad \as{N},
\end{equation}
where $\mathcal{N}(0,1)$ is the standard normal distribution with mean $0$ and variance $1$. Subsequently, Cs\"{o}rg\H{o}--Haeusler--Mason~\cite{CHMdesc} provided an explicit characterization of the norming sequences $a_N$ and $b_N$ in this case. Related distributional limits for intermediately trimmed sums normalized by extreme order statistics were obtained by Teugels~\cite{Teu81}.

While most classical results on distributional convergence are formulated for i.i.d.\ sequences, dynamical systems with strong mixing properties have also been studied in this context. For example, Bazarova--Berkes--Horv\'{a}th \cite{BBH16} established asymptotically Gaussian behavior for trimmed sums of partial quotients under the Gauss map. Our results extend this line of work to exponentially mixing systems of all orders.

In Theorems \ref{lighttrimthm} and \ref{intertrimthm}, we prove dynamical analogues of \eqref{lightDLT} and \eqref{interDLT} for systems that are exponentially mixing of all orders. Establishing these limit theorems in a dependent setting requires handling the interplay between extreme values (governing the trimming) and the system's correlation structure. We show that the lightly trimmed sum converges to a non-standard distribution described in Section \ref{SumPP}, while the intermediately trimmed sum, when suitably normalized, converges to a centered normal distribution. Collectively, these results provide a nearly complete dynamical counterpart to the classical trimming theory for i.i.d.\ sequences with polynomial tails, extending it to a broad class of non-independent systems.

The remainder of the paper is organized as follows. Section~\ref{PS} states the main theorems, and Section~\ref{nearsec} establishes the asymptotic equivalence between removing the largest observations and excluding points near a critical value. Moment estimates for truncated ergodic sums are given in Section~\ref{Tr}. Sections~\ref{LT} and \ref{IT} prove strong laws of large numbers under light and intermediate trimming, respectively. The tools for the non-standard limit law under light trimming are developed in Sections~\ref{PoiLgT} and \ref{SumPP}, with the proof completed in Section~\ref{LgTrDL}. Section~\ref{metricsec} summarizes facts on distributional convergence, and Section~\ref{intersec} establishes distributional limits under intermediate trimming.

Throughout this paper, we adhere to the following notational conventions:
\begin{itemize}
\item[$\bullet$] For real-valued functions $\alpha$ and $\beta$ defined on $\mathbb{N}$ or $\mathbb{R}$, we write $\beta = O(\alpha)$ if there exists a constant $C>0$ such that $|\beta(x)| \leq C |\alpha(x)|$ for all $x$, and $\beta=o(\alpha)$ if $|\beta(x)|/|\alpha(x)| \to 0$ as $x \to \infty$. If $\lim_{x\to\infty} \alpha(x)/\beta(x) = 1$, we write $\alpha \sim \beta$.

\item[$\bullet$] For $[0,\infty)$-valued functions $\alpha$ and $\beta$, we often write $\beta \ll \alpha$ instead of $\beta=O(\alpha)$, $\beta \gg \alpha$ instead of $\alpha=O(\beta)$, and $\beta \asymp \alpha$ if both $\beta \ll \alpha$ and $\alpha \ll \beta$.

\item[$\bullet$] For a measurable function $f$, its expectation is denoted $\mathbb{E} f = \int f \, d\mu$ if $f \in L^1(\mu)$, and its variance by $\mathrm{Var} f = \mathbb{E}((f-\mathbb{E}f)^2)$ if $f \in L^2(\mu)$.

\item[$\bullet$] For $x^* \in M$ and $r \ge 0$, let $B_r(x^*)$ denote the ball of radius $r$ centered at $x^*$. For simplicity, set $B_r(x^*) = \emptyset$ for $r < 0$. The notation $\mathbf{1}_A$ denotes the indicator function of a measurable set $A$.

\item[$\bullet$] For $d \geq 1$, let $B_d$ denote the volume of the $d$-dimensional unit ball.
\end{itemize}

\section{Polynomial Singularities}\label{PS}
Let $M$ be a compact smooth $d$-dimensional Riemannian manifold. Fix $\kappa\geq 1$. Let $T:M\to M$ be a $C^{\kappa}$ map preserving a probability measure $\mu$ that is absolutely continuous with respect to the Riemannian volume $\mathrm{vol},$ with a continuous density $\rho = \frac{\d\mu}{\d\mathrm{vol}}.$

We assume that the system is exponentially mixing for $C^\kappa$ observables, as defined below. When needed, we further assume that it is exponentially mixing of all orders for $C^\kappa$ observables.

\begin{definition}[Exponential mixing of order $m$]
The system $(M,\mathcal{B}(M),\mu,T)$ is said to be exponentially mixing of order $m$ if there are constants $C_m\geq 0$ and $\gamma>0$ such that, for all functions $f_0, \dots, f_{m-1} \in C^{\kappa}(M)$, and integers $0 \leq k_0 \leq \dots \leq k_{m-1}$, we have
\begin{equation}\label{ballretasm:mem}
\left|\E\left(\prod_{j=0}^{m-1} f_j \circ T^{k_j}\right)-\prod_{j=0}^{m-1} \E f_j \right|
\leq C_m e^{-\gamma \min_{0\leq j\leq m-2}(k_{j+1}-k_j)}\prod_{j=0}^{m-1} \|f_j\|_{C^{\kappa}}.
\end{equation}
In particular, when $m = 2$, the system is said to be exponentially mixing. If \eqref{ballretasm:mem} holds for all $m \geq 2,$ the system is called exponentially mixing of all orders.
\end{definition}

Since $m$ is fixed  or ranges over a bounded set in each application below, the dependence of $C_m$ on $m$ is immaterial, and we simply write $C$ in place of $C_m$ in \eqref{ballretasm:mem}.

\begin{definition}\label{PowSing}
Let $x^* \in M$, and let $f\in C^{\kappa}(M\setminus \{x^*\})$ be a non-negative measurable function. We say that $f$ has a \textit{power singularity of order} $\mathrm{Ord}_{x^*}(f)=\beta$ at $x^*$ if the function defined by
\begin{equation*}
g(x)=f(x) d(x,x^*)^{\beta} \quad x\not = x^*
\end{equation*}
admits a $C^{\kappa}$ extension in a neighborhood of $x^*$ and satisfies $g(x^*)>0$, where $d(\cdot, \cdot)$ denotes the geodesic distance on $M$. In this case, we call $g(x^*) = \mathrm{Res}_{x^*}(f)$ the residue of $f$ at $x^*$.
\end{definition}

\begin{remark}
The function $f(x)=x^{-\alpha}$ on $[0,1]$ does not strictly satisfy Definition~\ref{PowSing}, since the singularity at $x^*=0$ is one-sided. On $\mathbb{T}^1$, it would be two-sided. Nevertheless, the limit theorems below still apply after replacing $B_1$ by $B_1/2=1$. Indeed, $\mathrm{vol}(f>t)=t^{-1/\alpha},$
whereas for a proper power singularity $h$ of order $\alpha$,
$$
\mu(h>t) \sim B_d \rho(x^*) \mathrm{Res}_{x^*}(h)^{\frac{1}{\alpha}} t^{-\frac{1}{\alpha}}.
$$
Since $B_1=2$ corresponds to the two-sided case on $\mathbb{T}^1$, the one-sided setting contributes half this constant. This accounts for the factor $1/2$ difference between the constants in the i.i.d.\ case~\eqref{iid1xcon},~\eqref{iidbetacon} and those in Theorems~\ref{lightSLLN} and~\ref{interSLLN}.
\end{remark}

Under this condition, for $r>0$, there exists $\epsilon(r)>0$ such that
\begin{equation}\label{Estf}
\Bigl(1-\epsilon(r)\Bigr)\mathrm{Res}_{x^*}(f)d(x,x^*)^{-\beta}<f(x)<
\Bigl(1+\epsilon(r)\Bigr)\mathrm{Res}_{x^*}(f)d(x,x^*)^{-\beta}
\end{equation}
whenever $d(x,x^*)\leq r,$ and $\epsilon(r)\rightarrow0$ as $r\rightarrow0$. Without loss of generality, we assume that $\epsilon$ is $C^{\kappa}$ smooth and monotonically decreasing as $r$ approaches $0$. In the proofs we will always take $r = r_N$ with $r_N \to 0$ as $N \to \infty$, for example $r_N =\left(\frac{(\log \log N)^6}{N \log N}\right)^{\frac{1}{d}}$ in Section~\ref{LT} and $r_N\sim\left(\frac{k(N)}{B_d \rho(x^*) N}\right)^{\frac{1}{d}}$ in Section \ref{IT}.

Furthermore let $C_0 > 0$ be a constant such that
\begin{equation}\label{DeCf}
f(x) \leq C_0 d(x, x^*)^{-\beta} \quad \forall x\in M.
\end{equation}
We will also often denote $\alpha=\frac{\beta}{d}$.

\begin{definition}[Slowly recurrent point]\label{SlowRec}
A point $x^*\in M$ is called slowly recurrent if, for every $A, C>0$, there exists $r_0>0$ such that for all $0<r<r_0$ and positive integers $n\leq C |\log r|$, it holds that
\begin{equation*}
\mu \Bigl( B_r(x^*) \cap T^{-n}\bigl(B_r(x^*)\bigr) \Bigr) \leq \mu\bigl(B_r(x^*)\bigr) |\log r|^{-A}.
\end{equation*}
\end{definition}
This notion of slow recurrence is closely related to the mixing properties of the system. Specifically, if the system $(M, \mathcal{B}(M),\mu,T)$ is exponentially mixing, then $\mu$-almost every point is slowly recurrent (see \cite[Lemma 4.14]{DFL22}).

Throughout the rest of the paper, we assume that $x^*$ is the unique singularity of $f$, that it is of the power type described in Definition \ref{PowSing}, and that $\rho(x^*)>0$. Furthermore, we assume that $x^*$ is slowly recurrent as defined in Definition \ref{SlowRec}.

We will establish trimmed limit theorems for functions that exhibit non-integrable power singularities.

\begin{theorem}\label{lightSLLN}
Assume that the system $(M, \mathcal{B}(M), \mu, T)$ is exponentially mixing. Let $f \in C^{\kappa}(M \setminus \{x^*\})$ be a function with a power singularity of order $\operatorname{Ord}_{x^*}(f)=d$ at a slowly recurrent point $x^*\in M$. Then, for every fixed $K \geq 1$,
\begin{equation*}
\lim_{N \to \infty} \frac{S_N^K(f)(x)}{N \log N} = \operatorname{Res}_{x^*}(f)B_d \rho(x^*) \quad \text{for almost every } x.
\end{equation*}
\end{theorem}

\begin{remark}
For the doubling map on $\mathbb T^1$, Theorem~\ref{lightSLLN} cannot be applied directly to the observable
$$f_c(x)=\frac{1}{|x-c|},$$
because $f_c$ is not $C^\kappa$ at the identified endpoint $0=1$, where $\kappa\geq 1$ is the regularity exponent fixed at the beginning of this section. It does, however, apply to $\hat f_c=f_c+\psi_c$, where $\psi_c$ is a bounded function, vanishing near $c$, chosen so that $\hat f_c$ is non-negative and $C^\kappa$  away from the singularity at $c$. Since adding a bounded function does not affect the lightly trimmed strong law, we obtain a lightly trimmed strong law for $f_c$ and Lebesgue-almost every $c\in[0,1]$, because almost every point is slowly recurrent. In contrast, for the observable $f(x)=1/x$, no lightly trimmed strong law holds \cite{Haynes14}, although an intermediately trimmed strong law does hold~\cite{Sch18}. Hence, while light trimming fails to yield a strong law for $1/x$, it does yield one for $1/|x-c|$ for Lebesgue-almost every $c\in[0,1]$.

This example also shows that the regularity of the observable away from the singularity is not restrictive: provided that the required $C^\kappa$ structure holds near the singularity, bounded irregularities away from the singularity can be removed by adding a bounded function supported away from it.
\end{remark}

A lightly trimmed strong law implies an untrimmed weak law. In the i.i.d.\ setting, this follows immediately from Kesten's results \cite{kestenweak}. In the ergodic setting, the implication follows provided that $f$ takes values comparable to the leading order of the trimmed sum only with probability $o(1/N)$; see \cite[Remark 7]{AS25}. The corresponding result in our setting is the following.

\begin{corollary}[Convergence in probability]\label{SNCP}
Under the assumptions of Theorem \ref{lightSLLN}, we have
$$
\lim_{N \to \infty} \mu\left( \left| \frac{S_N(f)}{N \log N} - \operatorname{Res}_{x^*}(f)B_d \rho(x^*)\right| > \varepsilon \right) = 0, \quad \forall \varepsilon > 0.
$$
\end{corollary}

\begin{theorem}\label{interSLLN}
Assume that the system $(M, \mathcal{B}(M), \mu, T)$ is exponentially mixing of all orders. Let $f \in C^{\kappa}(M \setminus \{x^*\})$ be a function with a power singularity of order $\operatorname{Ord}_{x^*}(f)=\beta>d$ at a slowly recurrent point $x^*\in M$. Suppose that $k:\mathbb{N}\to\mathbb{N}$ is nondecreasing, $k(N)=o(N)$, and there exists $\epsilon>0$ such that $\lim_{N\rightarrow \infty} \frac{k(N)}{(\log N)^{\epsilon}} = \infty.$ Then it holds that
\begin{equation*}
\lim_{N\rightarrow \infty} \frac{S_N^{k(N)}(f)(x)}{N^{\alpha} k(N)^{1-\alpha}} = \frac{\operatorname{Res}_{x^*}(f)}{\alpha-1} B_d^{\alpha} \rho(x^*)^{\alpha}\quad \text{for almost every } x,
\end{equation*}
where $\alpha=\frac{\beta}{d}$.
\end{theorem}

\begin{remark}\label{specrem}
The assumption of exponential mixing of all orders is used here as an abstract quantitative mixing condition. It can be established by several mechanisms. In standard expanding and hyperbolic settings, it often follows from spectral estimates for transfer operators on appropriate function spaces, together with suitable multiplication estimates. Examples include uniformly expanding maps, piecewise expanding maps, and volume-preserving Anosov diffeomorphisms~\cite{Bow08,Via97}. Alternatively, \cite[Theorem~A.3]{DFL22} derives exponential mixing of all orders from quantitative equidistribution along unstable leaves and applies, among other examples, to time-one maps of contact Anosov flows, mostly contracting systems, partially hyperbolic translations on homogeneous spaces, and partially hyperbolic automorphisms of nilmanifolds. Products and certain skew products are also treated in~\cite{DFL22}. Thus Theorem~\ref{interSLLN} uses only the multiple-correlation estimates stated above, independently of the mechanism by which they are obtained. It remains open whether ordinary exponential mixing implies exponential mixing of all orders~\cite[Question~1.7]{DKRH24}.

This distinction is relevant because, even when a spectral gap is available, existing intermediately trimmed SLLNs apply only to restricted classes of observables. In \cite{KessSchiSLLN}, the relevant condition, called property $\mathfrak{D}$, includes in particular the growth condition
$$\|f \cdot \one_{f \leq n}\| = O(n).$$
The issue is not the possible loss of smoothness under truncation, which can be handled by smooth approximation, but the growth of the norm after truncation. In derivative-based norms, such as $C^\kappa$ norms or the smooth anisotropic spaces of Gou\"{e}zel--Liverani \cite{GL06}, truncation near singularities may yield superlinear growth. For instance, $f(x)=1/x$ truncated at height $n$ has $C^1$ norm of order $n^2$. Thus the above linear bound is natural in low-regularity spaces such as the space of functions of bounded variation, but generally fails in the smooth function spaces used for uniformly hyperbolic systems. Consequently, the observables covered by \cite{KessSchiSLLN} are typically not available in these settings.
\end{remark}

\begin{remark}
For i.i.d.\ observations, as discussed in the introduction, a stronger statement holds: a trimmed SLLN applies whenever $k(N)=o(N)$ and
\begin{equation*}
\lim_{N\to\infty} \frac{k(N)}{\log \log N} = \infty.
\end{equation*}
The same condition should be attainable for Theorem~\ref{interSLLN} if one could establish Bernstein-type moderate deviation estimates for the shrinking-target counts $S_N\bigl(\one_{B_{r_N}(x^*)}\bigr)$ and for the truncated sums $S_N\bigl(f \cdot (1-\one_{B_{r_N}(x^*)})\bigr)$.
However, under our current assumption of exponential mixing of all orders, such estimates appear to lie beyond the reach of the moment method employed here. Moreover, it is not clear that they can be deduced from the mixing assumption alone, since quantitative mixing does not in general imply Bernstein-type inequalities of i.i.d.\ strength. Indeed, Adamczak~\cite[Section~3.3]{Ada08} constructs a Markov chain for which no uniform Bernstein inequality with an i.i.d.-optimal large-deviation term can hold over all bounded centered observables. One can verify that this chain is exponentially $\beta$-mixing. Since exponential $\beta$-mixing implies exponential mixing of every finite order for bounded coordinate observables, this suggests that exponential multiple mixing alone is insufficient to guarantee concentration inequalities of i.i.d.\ strength. We therefore restrict ourselves to the present formulation of the theorem and leave this strengthening for future work.
\end{remark}

\begin{remark}
Analogously, one can derive intermediate trimming for the case $\operatorname{Ord}_{x^*}(f) = d$. The corresponding statement is as follows: if $k(N) \rightarrow\infty$ and $k(N)=o(N)$ then
\begin{equation}\label{interSLLNd}
\lim_{N \to \infty} \frac{S_N^{k(N)}(f)(x)}{N \log \left(\frac{N}{k(N)} \right)} =\operatorname{Res}_{x^*}(f)B_d \rho(x^*) \quad \text{for almost every } x.
\end{equation}
However, given that Theorem \ref{lightSLLN} already establishes the strong law under light trimming, in what can be regarded as its strongest form with $K=1,$ the result \eqref{interSLLNd} is not the most interesting one. We therefore omit the proof and leave it as an exercise for the interested reader.

It should be noted, however, that \eqref{interSLLNd} does not follow from the seemingly stronger light trimming result in Theorem~\ref{lightSLLN}. Indeed, as shown in \cite{AS25}, there exist functions for which trimming is not stable under removal of additional terms: an SLLN may hold for some trimming sequence $k(N)$, but fail if more terms are removed, say $k'(N) > k(N)$.
\end{remark}

\begin{theorem}\label{lighttrimthm}
Assume that the system $(M, \mathcal{B}(M), \mu, T)$ is exponentially mixing of all orders. Let $f \in C^{\kappa}(M \setminus \{x^*\})$ be a function with a power singularity of order $\operatorname{Ord}_{x^*}(f)=\beta>\frac{d}{2}$ at a slowly recurrent point $x^*\in M$. Then, for every fixed $K\geq 0$, there exist constants $a_N$ and a non-degenerate distribution $Y$ such that
\begin{equation*}
\frac{S_N^K(f) - a_N}{N^{\alpha}} \Rightarrow Y \;\;\;\as{N},
\end{equation*}
where $\alpha=\frac{\beta}{d}$. Furthermore, it holds that
\begin{align*}
\lim_{N\to\infty} \frac{a_N}{N}
&= \int_M f\,\mathrm{d}\mu,
&& \text{if}\;\; \tfrac12<\alpha<1,\\[0.5ex]
\lim_{N\to\infty} \frac{a_N}{N\log N}
&= \Res_{x^*}(f)\, B_d\, \rho(x^*),
&& \text{if}\;\; \alpha=1,\\[0.5ex]
a_N
&= 0,
&& \text{if}\;\; \alpha>1.
\end{align*}
\end{theorem}

\begin{remark}
As noted in the introduction, the normalizing sequences $a_N$ and $N^\alpha$ do not depend on the trimming number $K$, similarly to the i.i.d.\ case studied by Kesten~\cite{kestenweak}. The limit $Y$, however, does depend on $K$. More specifically, $Y$ is the light-trimming limit generated by the limiting Poisson process of rare visits to the singularity $x^*$, as described precisely in Proposition~\ref{PPPconprop}.

This is the same limiting mechanism that appears in the classical i.i.d.\ light-trimming theory in the stable-domain setting studied by Hall~\cite{Hal78}. In particular, the limiting distribution $Y$ coincides with that obtained in the i.i.d.\ case with the same marginal distribution. In the present setting, the Poisson process description yields the corresponding light-trimming limit, with the scale determined by the local tail asymptotics of $f$ near $x^*$.
\end{remark}

\begin{remark}
The case $K=0$ corresponds to the untrimmed setting, where stable limit laws are expected. Although this is likely well known and can be derived by combining the Poisson limit theorem from \cite{DFL22} with the approach of Tyran-Kami\'{n}ska \cite{Tyr10}, we include it in the statement for completeness.
\end{remark}

\begin{theorem}\label{intertrimthm}
Assume that the system $(M, \mathcal{B}(M), \mu, T)$ is exponentially mixing of all orders. Let $f \in C^{\kappa}(M \setminus \{x^*\})$ be a function with a power singularity of order $\operatorname{Ord}_{x^*}(f)=\beta>\frac{d}{2}$ at a slowly recurrent point $x^*\in M$.
Suppose $k(N)$ is a function with $k(N)\rightarrow\infty$ as $N\rightarrow\infty$, and $k(N)=o(N).$
Then there exist constants $a_N$ such that
\begin{equation*}
\frac{S_N^{k(N)}(f) - a_N}{\operatorname{Res}_{x^*}(f) B_d^{\alpha} \rho(x^*)^{\alpha} N^{\alpha} k(N)^{\frac{1}{2} - \alpha}} \Rightarrow \mathcal{N}(0,\sigma_{\alpha}^2) \;\;\;\as{N},
\end{equation*}
where $\alpha=\frac{\beta}{d}$ and
$$
\sigma_{\alpha}^2= \frac{2\alpha}{2\alpha - 1}.
$$
Furthermore, the normalizing constants $a_N$ satisfy
\begin{align*}
\lim_{N\to\infty} \frac{a_N}{N}
&= \int_M f\,\mathrm{d}\mu,
&& \text{if}\;\;\tfrac12<\alpha<1,\\[0.5ex]
\lim_{N\to\infty} \frac{a_N}{N\log\!\bigl(\tfrac{N}{k(N)}\bigr)}
&= \Res_{x^*}(f)\, B_d\, \rho(x^*),
&& \text{if}\;\; \alpha=1,\\[0.5ex]
\lim_{N\to\infty} \frac{a_N}{N^{\alpha} k(N)^{1-\alpha}}
&= \frac{\Res_{x^*}(f)}{\alpha-1}\,
  B_d^{\alpha}\, \rho(x^*)^{\alpha},
&& \text{if}\;\; \alpha>1.
\end{align*}
\end{theorem}

\part{Proofs}

\section{Large observations and closest points}\label{nearsec}

Let us revisit an issue mentioned in the introduction: the conceptual
relationship between trimming based on extremal function values and
trimming based on geometric proximity to a singularity.
For the canonical radial example
$f(x)=1/\|x\|,$ $x=(x_1,x_2)\in B_1(0)\setminus\{0\}\subset\mathbb R^2,$ ordering by function value is
equivalent to ordering by distance from the singularity. This equivalence
fails in general. Indeed, for
$f_\delta(x)=(1+\delta x_1)/\|x\|$ with $0<\delta<1$, the normalized
singularity profile $g_\delta(x)=f_\delta(x)\|x\|=1+\delta x_1$ extends
smoothly across the origin, yet there exist points arbitrarily close to the
origin whose ordering by distance is opposite to their ordering by function
value. For instance,
for $p_r=(-r,0)$ and $q_r=(r+\delta r^2,0)$, with $r>0$ sufficiently
small, one has $d(p_r,0)<d(q_r,0)$ but $f_\delta(p_r)<f_\delta(q_r).$

Despite this potential discrepancy, the asymptotic behavior often
remains similar. We prove that the trimmed sum is
asymptotically equivalent to the sum obtained by removing points
closest to the singularity, allowing us to work with this geometric
notion of trimming in the proofs that follow.

Consider $k$ as either a constant $K$ or a function $k(N)$.
Let $\hat{S}_N^k(f)(x)$ denote the sum obtained by removing the $k$ closest visits to $x^*$, formally defined as
\begin{equation}\label{hatSdef}
\hat{S}_N^k(f)(x)=S_N(f)(x) - \sum_{l=1}^k f(\hat{x}_l),
\end{equation}
where $\hat{x}_l$ is the $l$th closest point to $x^*$ among $\{x,\ldots,T^{N-1}(x)\}$. By definition of $S_N^k,$ it clearly holds that
\begin{equation*}
S_N^k(f) \leq \hat{S}_N^k(f).
\end{equation*}
On the other hand, let $x_1,\dots,x_k$ be the points that maximize $f$ among $\{x,\ldots,T^{N-1}(x)\},$ ordered by increasing distance to $x^*$ so that
\begin{align*}
&d(x_1,x^*)\leq d(x_2,x^*) \leq \cdots \leq d(x_k, x^*),\\
& f(x_i) \geq f(T^j(x)) \quad \text{for } 1\leq i\leq k,\,\,0\leq j\leq N-1,\,\, T^j(x)\not\in \left\{x_1,\dots,x_k\right\}.
\end{align*}
Then we have
\begin{equation*}
\hat{S}_N^k(f)(x)-S_N^k(f)(x)=\sum_{i=1}^kf(x_i)-\sum_{i=1}^k f(\hat{x}_i).
\end{equation*}
If $\{x_1,\ldots,x_k\}=\{\hat{x}_1,\ldots,\hat{x}_k\}$, then the sums coincide. Otherwise, let $1\leq m\leq k$ be the number of points $x_i$ not in $\{\hat{x}_1,\ldots,\hat{x}_k\}$.
Since $x$ is almost surely not eventually periodic, no two points in the orbit of such an $x$ coincide. Therefore, there are also exactly $m$ points among the $\hat{x}_i$ not contained in $\{x_1,\ldots,x_k\}$.
Relabel these points so that $\{y_1,\ldots,y_m\}=\{x_1,\ldots,x_k\}\setminus\{\hat{x}_1,\ldots,\hat{x}_k\},$ and likewise for $\{\hat{y}_1,\ldots,\hat{y}_m\}$, keeping the original order, although the order itself is not essential for the
subsequent argument.

For any $i,j\in\{1,\dots,m\}$, by construction we have $f(y_i)\geq f(\hat{y}_j)$ since $y_i$ was chosen among the largest values. On the other hand, by \eqref{Estf} we have
\begin{align*}
f(y_i) &\leq \biggl(1+\epsilon\Bigl(d(x^*, y_i)\Bigr)\biggr) \mathrm{Res}_{x^*}(f)d(x^*, y_i)^{-\beta} \\
& \leq\biggl(1+\epsilon\Bigl(d(x^*,x_k)\Bigr)\biggr) \mathrm{Res}_{x^*}(f)d(x^*, \hat{y}_j)^{-\beta} \leq \frac{1+\epsilon\Bigl(d(x^*,x_k)\Bigr)}{1-\epsilon\Bigl(d(x^*,\hat{x}_k)\Bigr)} f(\hat{y}_j).
\end{align*}
At the same time $\sum_{i=1}^m f(\hat{y}_i)\leq S_N^k(f)(x)$ and therefore
\begin{align*}
\hat{S}_N^k(f)(x)-S_N^k(f)(x)&=\sum_{i=1}^m (f(y_i)-f(\hat{y}_i))\\
& \leq \left( \frac{1+\epsilon\Bigl(d(x^*,x_k)\Bigr)}{1-\epsilon\Bigl(d(x^*,\hat{x}_k)\Bigr)}-1\right) \sum_{i=1}^m f(\hat{y}_i)\\
&\leq \left(\frac{1+\epsilon\Bigl(d(x^*,x_k)\Bigr)}{1-\epsilon\Bigl(d(x^*,\hat{x}_k)\Bigr)}-1\right) S_N^k(f)(x),
\end{align*}
and we conclude
\begin{equation}\label{hatSineq}
\hat{S}_N^k(f)(x) \leq \frac{1+\epsilon\Bigl(d(x^*,x_k)\Bigl)}{1-\epsilon\Bigl(d(x^*,\hat{x}_k)\Bigl)} S_N^k(f)(x).
\end{equation}

\begin{lemma}\label{hatxklem}
Let $k(N)=o(N),$ and suppose that $f$ has a power singularity at $x^*$. Then, for almost every $x,$ it holds that
\begin{equation*}
d\Bigl(x^*,\hat{x}_{k(N)}\Bigr) \rightarrow 0 \quad \as{N}.
\end{equation*}
\end{lemma}
\begin{proof}
Let $G$ be the set of points $x\in M$ satisfying
\begin{equation*}
\lim_{N\rightarrow\infty} \frac{S_N\Bigl(\one_{B_{2^{-l}}(x^*)}\Bigr)(x)}{N} = \mu\Bigl(B_{2^{-l}}(x^*)\Bigr) \quad \forall l\in\N.
\end{equation*}
By the Ergodic Theorem, $\mu(G)=1$. For $x\in G$ and $\varepsilon>0$, choose $l$ sufficiently large so that $2^{-l} < \varepsilon$, then for all sufficiently large $N$ it follows that
\begin{equation*}
S_N\Bigl(\one_{B_{2^{-l}}(x^*)}\Bigr)(x) > k(N).
\end{equation*}
This implies that all $k(N)$ closest points $\hat{x}_1,\dots,\hat{x}_{k(N)}$ lie in $B_{2^{-l}}(x^*)$, so $d\bigl(\hat{x}_{k(N)},x^*\bigr) \leq 2^{-l}$.
\end{proof}

\begin{lemma}\label{xklem}
Let $k(N)=o(N)$ and suppose $f$ has a power singularity at $x^*$. Then for almost every $x,$
\begin{equation}\label{mdsay}
d\bigl(x^*,x_{k(N)}\bigr) \rightarrow 0 \quad \as{N}.
\end{equation}
\end{lemma}
\begin{proof}
Applying the same argument as in Lemma \ref{hatxklem} to the sets $\{f>R\}$, where $R>0$, instead of balls, yields that almost surely $f(x_{k(N)})\rightarrow\infty$ as $N\rightarrow\infty$. Keeping in mind that $f(y) \leq C_0 d(x^*,y)^{-\beta}$ for all points $y$, the claim follows.
\end{proof}

\begin{lemma}\label{xkhatxklem}
Let $k(N)=o(N)$ and suppose that $f$ has a power singularity at $x^*$. Then
\begin{equation*}
d\bigl(x^*,\hat{x}_{k(N)}\bigr) \sim d\bigl(x^*,x_{k(N)}\bigr)
\end{equation*}
for almost every $x.$
\end{lemma}
\begin{proof}
By definition $d\bigl(x^*,\hat{x}_{k(N)}\bigr) \leq d\bigl(x^*,x_{k(N)}\bigr)$, so it suffices to show
\begin{equation}\label{xkhatxklimsup}
\limsup_{N\rightarrow\infty} \frac{d\bigl(x^*,x_{k(N)}\bigr)}{d\bigl(x^*,\hat{x}_{k(N)}\bigr)} \leq 1.
\end{equation}
If $d\bigl(x^*,\hat{x}_{k(N)}\bigr)\ge d\bigl(x^*,x_{k(N)}\bigr),$
there is nothing to prove. Otherwise, $x_{k(N)}\notin\{\hat{x}_1,\ldots,\hat{x}_{k(N)}\}$, and hence
there exists some $1\leq i\leq k(N)$ such that $\hat{x}_i\notin\{x_1,\ldots,x_{k(N)}\}.$ By the definition of $x_1,\ldots,x_{k(N)}$, we have $f(\hat{x}_i)\leq f(x_{k(N)}).$ Moreover, $d(x^*,\hat{x}_i)\leq d\bigl(x^*,\hat{x}_{k(N)}\bigr).$
Thus, for all sufficiently large $N$,
\begin{align*}
\frac{\Res_{x^*}(f)\left(1-\epsilon\Bigl(d\bigl(x^*,\hat{x}_{k(N)}\bigr)\Bigr)\right)}{ d\bigl(x^*,\hat{x}_{k(N)}\bigr)^\beta}&\leq
f(\hat{x}_i) \leq f(x_{k(N)})\\
&\leq \frac{\Res_{x^*}(f)\left(1+\epsilon\Bigl(d\bigl(x^*,x_{k(N)}\bigr)\Bigr)\right)}{ d\bigl(x^*,x_{k(N)}\bigr)^\beta }.
\end{align*}
Since $d(x^*,\hat{x}_{k(N)})\to0$ and $d(x^*,x_{k(N)})\to0$ as $N\to\infty$, we have
\begin{equation*}
\epsilon\Bigl(d\bigl(x^*,\hat{x}_{k(N)}\bigr)\Bigr)\to0
\quad\text{and}\quad
\epsilon\Bigl(d\bigl(x^*,x_{k(N)}\bigr)\Bigr)\to0,
\end{equation*}
and hence \eqref{xkhatxklimsup} follows.
\end{proof}

\begin{lemma}\label{strongsklem}
Let $k(N)=o(N)$ and suppose $f$ has a power singularity at $x^*$. Then it holds that
\begin{equation*}
\lim_{N\rightarrow \infty} \frac{\hat{S}_N^{k(N)}(f)(x)}{S_N^{k(N)}(f)(x)} = 1 \quad \text{for almost every } x.
\end{equation*}
\end{lemma}
\begin{proof}
Since $\hat{S}_N^{k(N)}(f) \geq S_N^{k(N)}(f)$ by definition, it suffices to show that
\begin{equation*}
\limsup_{N\rightarrow \infty} \frac{\hat{S}_N^{k(N)}(f)(x)}{S_N^{k(N)}(f)(x)} \leq 1 \quad \text{for almost every } x.
\end{equation*}
Given that $d\bigl(x^*,\hat{x}_{k(N)}\bigr),\, d\bigl(x^*,x_{k(N)}\bigr) \rightarrow 0$ as $N\to\infty,$ the desired inequality follows directly from \eqref{hatSineq}.
\end{proof}

\section{Moments of Truncated Observables}\label{Tr}
In this section, we collect the estimates on truncated observables that will be used in the proofs of the main theorems. Throughout, $f$ is assumed to possess a power singularity of order $\beta$ at a slowly recurrent point $x^*$. The basic truncations are obtained by removing a small ball around the singularity: for $\epsilon>0$, we call
$$
f\cdot\left(1-\one_{B_\epsilon(x^*)}\right)
$$
a \emph{truncation} of $f$. The guiding idea is that, once the largest observations are removed, the trimmed ergodic sum can be compared with an ordinary ergodic sum of a suitable truncation of $f$. We therefore require two quantitative ingredients: control of short returns to small balls around $x^*$, and moment estimates for ergodic sums of functions truncated near $x^*$.

We begin with a standard consequence of exponential mixing, which will be used to control correlations of small balls.

\begin{lemma}[Mixing Estimate for Balls]\label{MixEmB}
If the system $(M, \mathcal{B}(M), \mu, T)$ is exponentially mixing as defined in \eqref{ballretasm:mem}, then for any balls $B_1, B_2 \subset M$ and any $k \in \mathbb{N}$, we have
$$
\left|\mu(B_1 \cap T^{-k} B_2)-\mu(B_1) \mu(B_2)\right| \ll \max\Bigl(\mu(B_1), \mu(B_2) \Bigr)^{\frac{d-1}{d} \cdot \frac{2\kappa}{2\kappa+1}} e^{-\frac{\gamma k}{2\kappa+1}}.$$
\end{lemma}
Lemma~\ref{MixEmB} is a standard consequence of exponential mixing and its proof is therefore omitted.

For the later arguments it is convenient to formulate the truncation estimates slightly more generally. Instead of only removing balls $B_\epsilon(x^*)$, we will allow the removed sets to be small perturbations of balls. This flexibility is useful because the relevant truncation levels will depend on $N$, and in some places the natural sets are not exactly balls but are asymptotically equivalent to balls around $x^*$.

We shall say that a sequence of sets $A_N\subset M$ is \emph{regularly adapted} if there exists a sequence $r_N\to0$ and constants $a,a'\ge0$ such that the following two conditions hold.

\begin{itemize}
\item[(R1)] There exists a sequence $\delta_N=o(r_N)$ such that
\begin{equation}\label{regasm1}
B_{r_N}(x^*) \subset A_N \subset B_{r_N+\delta_N}(x^*)\quad \forall N\ge1.
\end{equation}
\item[(R2)] For every sequence $\epsilon_N=o(r_N)$ there are approximating functions $h_N\in C^\kappa(M)$ such that $\one_{A_N}\le h_N\le 1$ and
\begin{equation}\label{regasm2}
\|\one_{A_N}-h_N\|_{L^1} \ll \epsilon_N r_N^a, \quad \text{and} \quad \|h_N\|_{C^{\kappa}} \ll \epsilon_N^{-\kappa} r_N^{-a'}
\end{equation}
\end{itemize}

\begin{remark}
The parameter $r_N$ should be viewed as the effective radius of $A_N$. This scale is well defined asymptotically: if two sequences $r_N$ and $r_N'$ both satisfy \eqref{regasm1} for the same sets $A_N$, then $r_N\sim r_N'$. Hence all asymptotic estimates below are independent of the particular admissible choice of $r_N$.
\end{remark}

The next two elementary observations record the consequences of regular adaptation that will be used throughout the rest of the section.

\begin{lemma}\label{regadaptmeasure}
Let $A_N\subset M$ be a sequence of sets that is regularly adapted. Then
$$
\mu(A_N)\sim B_d\rho(x^*)r_N^d.
$$
\end{lemma}

\begin{proof}
By \eqref{regasm1}, $B_{r_N}(x^*)\subset A_N\subset B_{r_N+\delta_N}(x^*),$ and $\delta_N=o(r_N).$
Since $\mu(B_r(x^*))\sim B_d\rho(x^*)r^d$ and $r_N+\delta_N\sim r_N$, the conclusion follows by squeezing.
\end{proof}

\begin{lemma}\label{slowrecanlem}
Let $A_N\subset M$ be a sequence of sets that is regularly adapted. Then, for all constants $A,\,C>0$, there exists $N_0>1$ such that for all $N>N_0$ and all positive integers $1\le n\le C|\log r_N|$, one has
\begin{equation*}
\mu \left(A_N \cap T^{-n}(A_N) \right) \leq r_N^d |\log r_N|^{-A}.
\end{equation*}
\end{lemma}

\begin{proof}
Let $\widetilde r_N = r_N + \delta_N$. By \eqref{regasm1}, $A_N \subset B_{\widetilde r_N}(x^*)$ and $\widetilde r_N \sim r_N$. Hence $|\log \widetilde r_N| \sim |\log r_N|$, and for large $N$, $1\le n \le C|\log r_N|$ implies $1\le n \le 2C|\log \widetilde r_N|$. Applying the slow recurrence with exponent $A+2$ and time constant $2C$ yields
$$
\mu\bigl(B_{\widetilde r_N}(x^*) \cap T^{-n}B_{\widetilde r_N}(x^*)\bigr)
\le \mu\bigl(B_{\widetilde r_N}(x^*)\bigr)\,|\log \widetilde r_N|^{-A-2}.
$$
Since $\widetilde r_N \sim r_N$ and $\mu(B_{\widetilde r_N}(x^*)) \sim B_d\,\rho(x^*)\,\widetilde r_N^{\,d}$, the right-hand side is bounded by $r_N^d|\log r_N|^{-A}$ for all sufficiently large $N$. The result follows from the inclusion $A_N \cap T^{-n}(A_N) \subset B_{\widetilde r_N}(x^*) \cap T^{-n}B_{\widetilde r_N}(x^*)$.
\end{proof}

For a sequence of sets $A_N$ that is regularly adapted, we refer to $f\cdot(1-\one_{A_N})$ as a \emph{generalized truncation} of $f$. We next prove a short-time covariance estimate for such truncations.

\begin{lemma}[Covariance for the generalized truncation]\label{truncvarlem}
Let $\beta>\frac{d}{2}$, let $A_N \subset M$ be regularly adapted, and set $f_N=f \cdot \left(1-\one_{A_N}\right).$ Then for any constants $A,\,C>0$, there exists $N_0>1$ such that for all $N>N_0$ and positive integers $1\leq n\leq C |\log r_N|$,
\begin{equation}\label{truncvarlemconc}
\E\left(f_N \cdot f_N\circ T^n\right) \leq r_N^{- 2\beta + d} |\log r_N|^{-A}.
\end{equation}
\end{lemma}

\begin{proof}
We first record a local moment bound. Since $f(x) \le C_0\,d(x,x^*)^{- \beta}$ and $\rho$ is continuous, for any $s>0$ with $s\beta>d$ there exists a constant $K_s$ independent of $r$, such that for all sufficiently small $r>0$,
\begin{equation}\label{lcmb}
\int_{M \setminus B_r(x^*)} f^s \d\mu \le K_s \, r^{d - s\beta}.
\end{equation}
In addition, for small enough $r>0$, it holds that
\begin{equation*}
\mu (B_r(x^*)) \leq 2 \rho(x^*) B_d r^d.
\end{equation*}
Choose $q>1$ such that $q^2\beta>d$ and $\beta-d+\frac{d}{q^2}>0$. Denote by $p$ the conjugate exponent of $q$, and choose $H>0$ so large that $H\left(\beta-d+\frac{d}{q^2}\right)\ge A+2,$ and $H(2\beta-d)\ge A+2$.

Set $\hat{r}_N=r_N|\log r_N|^H$ and define $\hat{f}_N=f_N \cdot \left(1-\one_{B_{\hat{r}_N}(x^*)}\right).$
For $1\leq n <C |\log r_N|,$ we decompose the term $\E\left(f_N \cdot f_N\circ T^n\right)$ as
$$
\E\left(f_N \cdot f_N\circ T^n\right)=\mathfrak{A} + \mathfrak{B} + \mathfrak{C} +\mathfrak{D},
$$
where
$$
\begin{aligned}
\mathfrak{A} &:= \int_{T^{-n}(B_{\hat{r}_N}(x^*))} \hat{f}_N \cdot f_N \circ T^n \d\mu, \\
\mathfrak{B} &:= \int_{B_{\hat{r}_N}(x^*)} f_N \cdot \hat{f}_N \circ T^n \d\mu, \\
\mathfrak{C} &:= \E \left(\hat{f}_N \cdot \hat{f}_N \circ T^n\right), \\
\mathfrak{D} &:= \int_{B_{\hat{r}_N}(x^*) \cap T^{-n}(B_{\hat{r}_N}(x^*))} f_N \cdot f_N \circ T^n\d\mu.
\end{aligned}
$$

Bounding $\mathfrak{A}, \mathfrak{B}$:

Applying H\"{o}lder's inequality twice with exponents $p$ and $q$ yields
\begin{align*}
\mathfrak{A}&= \int_{T^{-n}(B_{\hat{r}_N}(x^*))} \hat{f}_N \cdot f_N \circ T^n \d\mu\leq\mu(B_{\hat{r}_N}(x^*))^{\frac{1}{p}} \left(\E\left(\hat{f}_N \cdot f_N \circ T^n \right) \right)^{\frac{1}{q}}\\
&\leq\left(2B_d\rho(x^*)\hat{r}_N^d\right)^{\frac{1}{p}} \left(\E (\hat{f}_N^{pq}) \right)^{\frac{1}{pq}} \left(\E (f_N^{q^2}) \right)^{\frac{1}{q^2}}.
\end{align*}
Note that always $pq \geq 4$, hence $p q \beta > d$. Since $\hat{f}_N$ vanishes on $B_{\hat r_N}(x^*)$ and $f_N$ vanishes on $A_N\supset B_{r_N}(x^*)$, we can apply \eqref{lcmb} to obtain
$$\E(\hat{f}_N^{pq})\le K_{pq} \hat{r}_N^{d-pq\beta},\quad\E(f_N^{q^2})\le K_{q^2}r_N^{d-q^2\beta}.
$$
For sufficiently large $N$ it follows that
\begin{align*}
\mathfrak{A}&\leq\left(2B_d\rho(x^*)\right)^{\frac{1}{p}}K_{pq}^{\frac{1}{pq}}
K_{q^2}^{\frac{1}{q^2}}\hat{r}_N^{-\beta+d-\frac{d}{q^2}} r_N^{-\beta+\frac{d}{q^2}}\\
&\leq\left(2B_d\rho(x^*)\right)^{\frac{1}{p}}K_{pq}^{\frac{1}{pq}}
K_{q^2}^{\frac{1}{q^2}}r_N^{-2\beta+d}|\log r_N|^{-A-2}.
\end{align*}
The same estimate applies to $\mathfrak{B}.$

Bounding $\mathfrak{C}$: We apply the Cauchy--Schwarz inequality and obtain
\begin{equation*}
\mathfrak{C} \le \E (\hat{f}_N^2)\le K_2\hat{r}_N^{-2\beta+d}\le K_2 r_N^{-2\beta+d}|\log r_N|^{-A-2},
\end{equation*}
for $N$ sufficiently large.

Bounding $\mathfrak{D}$: Let $A'>0$ be fixed, to be chosen later. By the slow recurrence property (see Definition \ref{SlowRec}), for $N$ sufficiently large, for $1\le n<2C|\log\hat{r}_N|$ it holds that
$$
\mu\left(B_{\hat{r}_N}(x^*) \cap T^{-n}\Bigl(B_{\hat{r}_N}(x^*)\Bigr)\right) \leq \mu(B_{\hat{r}_N}(x^*)) | \log \hat{r}_N |^{-A'} \leq 2B_d\rho(x^*)\hat{r}_N^d | \log \hat{r}_N|^{-A'}.
$$

Applying H\"{o}lder's inequality and \eqref{lcmb} we obtain
\begin{align*}
\mathfrak{D} &\leq\mu\left(B_{\hat{r}_N}(x^*)\cap T^{-n}\Bigl(B_{\hat{r}_N}(x^*)\Bigr)\right)^{\frac{1}{p}} \left(\E \left(f_N^q \cdot f_N^{q} \circ T^n \right)\right)^{\frac{1}{q}}\\
&\le\left(2B_d\rho(x^*)\hat{r}_N^d | \log \hat{r}_N|^{-A'}\right)^{\frac1p}\left(\E\Bigl(f_N^{pq}\Bigr) \right)^{\frac{1}{pq}} \left(\E \Bigl(f_N^{q^2}\Bigr) \right)^{\frac{1}{q^2}}\\
&\le\left(2B_d\rho(x^*)\right)^{\frac{1}{p}}K_{pq}^{\frac{1}{pq}}
K_{q^2}^{\frac{1}{q^2}}r_N^{-2\beta+d}|\log r_N|^{\frac{Hd}{p}}|\log\hat{r}_N|^{-\frac{A'}{p}}.
\end{align*}

Note that $|\log \hat{r}_N|\ge \frac{1}{2} |\log r_N|$ for sufficiently large $N$, and choosing $A'>Hd+p(A+2)$ gives
$$\mathfrak{D}\leq 2^{\frac{A'}{p}}\left(2B_d\rho(x^*)\right)^{\frac{1}{p}}K_{pq}^{\frac{1}{pq}}
K_{q^2}^{\frac{1}{q^2}}r_N^{-2\beta+d}|\log r_N|^{-A-2}.$$

Combining the bounds for $\mathfrak{A}$, $\mathfrak{B}$, $\mathfrak{C}$, and $\mathfrak{D}$, we obtain
$$\E\left(f_N \cdot f_N\circ T^n\right)\le\tilde{K}r_N^{-2\beta+d}|\log r_N|^{-A-2},
$$
where
$$\tilde{K}=\bigl(2+2^{\frac{A'}{p}}\bigr)\left(2B_d\rho(x^*)\right)^{\frac{1}{p}}K_{pq}^{\frac{1}{pq}}
K_{q^2}^{\frac{1}{q^2}}+K_2.
$$
Choose $N_0$ sufficiently large so that $\tilde{K}|\log r_N|^{-2}\le 1$ for all $N>N_0$. Then
$$\E\left(f_N \cdot f_N\circ T^n\right)\le r_N^{-2\beta+d}|\log r_N|^{-A},
$$
which completes the proof.
\end{proof}

\begin{proposition}\label{momentprop}
Fix $m_1,m_2\in \mathbb N$, and set $m=m_1+m_2.$ Assume that $m\geq2$ and that the system is exponentially mixing of order $m$.

Suppose that $\Ord_{x^*}(f)=\beta>\frac{d}{2}$. Let $A_N\subset M$ be a sequence of sets that is regularly adapted and satisfies $N r_N^{d} \rightarrow \infty$. Define $f_N=f \cdot \left(1-\one_{A_N}\right).$ Then
\begin{equation}\label{momentclaim}
\lim_{N\rightarrow\infty} \frac{\int_M \left(S_N(\one_{A_N}) - N \mu(A_N) \right)^{m_1} \left(S_N(f_N) - N \E( f_N)\right)^{m_2}\d\mu}{\Res_{x^*}(f)^{m_2}\bigl(B_d \rho(x^*)\bigr)^{\frac{m}{2}} N^{\frac{m}{2}} r_N^{\frac{md}{2}-m_2\beta }}=\omega(m_1,m_2),
\end{equation}
where $\alpha=\frac{\beta}{d}$,
\begin{equation*}
\sigma^2= \frac{1}{2\alpha-1},
\end{equation*}
and \begin{equation*}
\omega(m_1,m_2)= \begin{cases}
(m_1 - 1)!! (m_2-1)!! \sigma^{m_2} & \text{if both } m_1 \text{ and } m_2 \text{ are even},\\
0 & \text{otherwise}.
\end{cases}
\end{equation*}
\end{proposition}

The cases $m=0$ and $m=1$ are trivial. Indeed, if $m=0$, then both sides of \eqref{momentclaim} are equal to $1$. If $m=1$, then either $(m_1,m_2)=(1,0)$ or $(0,1)$, and the corresponding moment in \eqref{momentclaim} vanishes by centering. Thus, as an immediate consequence of Proposition~\ref{momentprop} and the method of moments, we obtain the following joint distributional limit. This consequence will be used in Section~\ref{intersec}, while the explicit moment estimate of Proposition~\ref{momentprop} will also be needed later.

\begin{corollary}\label{CLTgenrem}
Assume that the system is exponentially mixing of all orders and that the remaining assumptions of Proposition~\ref{momentprop} hold. Then
$$\left(\frac{S_N(\one(A_N)) - N \mu(A_N)}{\sqrt{B_d\rho(x^*)N r_N^d}},\,\frac{S_N(f_N) - N \E(f_N)}
{\Res_{x^*}(f)\sqrt{B_d\rho(x^*)N r_N^d}\, r_N^{-\beta}}\right) \Rightarrow (\mathcal{N}(0,1),\mathcal{N}(0,\sigma^2))
$$
as $N\to\infty,$ where the right-hand side denotes a vector of independent normal random variables, and
$\sigma^2=\frac{1}{2\alpha-1}.$
\end{corollary}

\begin{proof}[Proof of Proposition \ref{momentprop}]
Denote $g_1=\one_{A_N}$ and $g_2=f_N$. By replacing $f$ with $f/\Res_{x^*}(f)$, we may assume throughout the proof that $\Res_{x^*}(f)=1$. The general case is recovered by rescaling at the end.

Step 1: Smooth approximation.

Let $\epsilon_N>0$ be chosen later. Using \eqref{regasm2}, there is a $C^{\kappa}$ smooth function $h_1$ such that $\one_{A_N}\le h_1\le 1$, with
\begin{equation*}
\|g_1- h_1\|_{L^1} \ll \epsilon_N r_N^a \quad \text{and} \quad \|h_1\|_{C^{\kappa}} \ll \epsilon_N^{-\kappa} r_N^{-a'}.
\end{equation*}
Define $h_2= f \cdot (1-h_1)$. Then
\begin{equation*}
\|g_2 - h_2 \|_{L^1} \ll \epsilon_N r_N^{a-\beta} \quad \text{and} \quad \|h_2\|_{C^{\kappa}} \ll \epsilon_N^{-\kappa} r_N^{-a'-\beta-\kappa}.
\end{equation*}

Define $\bar{g}_j=g_j-\int_M g_j \d\mu$ and $\bar{h}_j=h_j - \int_M h_j \d\mu$ for $j=1, 2.$ We next justify that replacing $g_j$ with $h_j$ in the integral of $\prod_{j=1}^2 \Bigl(S_N(\bar{g}_j)\Bigr)^{m_j}$ introduces a negligible error.

Since centering preserves the $L^1$-order, we have
\begin{equation}\label{cL1e}
\|\bar g_1-\bar h_1\|_{L^1}\ll \epsilon_N r_N^a
\quad \text{and} \quad
\|\bar g_2-\bar h_2\|_{L^1}\ll \epsilon_N r_N^{a-\beta}.
\end{equation}

We now estimate the error caused by replacing $g_j$ with $h_j$. Using
$$|x^m-y^m|\le m|x-y|\max(|x|,|y|)^{m-1}, $$
we obtain
\begin{equation*}
\begin{aligned}
&\left\|\prod_{j=1}^2 \Bigl(S_N(\bar g_j)\Bigr)^{m_j}-\prod_{j=1}^2 \Bigl(S_N(\bar h_j)\Bigr)^{m_j}\right\|_{L^1}\\
&\le m_1\left\|S_N(\bar g_1)-S_N(\bar h_1)\right\|_{L^1}\max\left(\left\|S_N(\bar g_1)\right\|_{L^\infty},
\left\|S_N(\bar h_1)\right\|_{L^\infty}\right)^{m_1-1}\left\|S_N(\bar g_2)\right\|_{L^\infty}^{m_2} \\
&\quad+m_2\left\|S_N(\bar g_2)-S_N(\bar h_2)\right\|_{L^1}\left\|S_N(\bar h_1)\right\|_{L^\infty}^{m_1}
\max\left(\left\|S_N(\bar g_2)\right\|_{L^\infty},\left\|S_N(\bar h_2)\right\|_{L^\infty}\right)^{m_2-1}.
\end{aligned}
\end{equation*}
Using \eqref{cL1e},
$$\left\|S_N(\bar g_1)-S_N(\bar h_1)\right\|_{L^1}\ll N\epsilon_N r_N^a,\quad
\left\|S_N(\bar g_2)-S_N(\bar h_2)\right\|_{L^1}\ll N\epsilon_N r_N^{a-\beta}.$$
Moreover,
$$\left\|S_N(\bar g_1)\right\|_{L^\infty},\left\|S_N(\bar h_1)\right\|_{L^\infty}\ll N,
\quad\text{and}\quad
\left\|S_N(\bar g_2)\right\|_{L^\infty},\left\|S_N(\bar h_2)\right\|_{L^\infty}\ll Nr_N^{-\beta}.$$
Therefore
$$\left\|\prod_{j=1}^2 \Bigl(S_N(\bar g_j)\Bigr)^{m_j}-\prod_{j=1}^2 \Bigl(S_N(\bar h_j)\Bigr)^{m_j}\right\|_{L^1}\ll \epsilon_N N^m r_N^{a-m_2\beta}.
$$
After division by the normalization in \eqref{momentclaim}, this is bounded by
$O\bigl(\epsilon_N N^{\frac{m}{2}}r_N^{a-\frac{md}{2}}\bigr).$
Since $Nr_N^d\to\infty$, choosing $\epsilon_N=N^{-B}$ with $B>0$ sufficiently large makes this term $o(1)$. Thus the error caused by replacing  $g_j$ by $h_j$ in \eqref{momentclaim} is negligible.

Therefore, the claim \eqref{momentclaim} will follow once we show
\begin{equation}\label{hbarerr}
\lim_{N\rightarrow\infty} \frac{\int_M \left(S_N(\bar{h}_1) \right)^{m_1} \left(S_N(\bar{h}_2)\right)^{m_2}\d\mu}{\left(B_d \rho(x^*)\right)^{\frac{m}{2}} N^{\frac{m}{2}} r_N^{\frac{md}{2}-m_2\beta }}=\omega(m_1,m_2).
\end{equation}

Step 2: Cluster decomposition.

We expand the product
\begin{align*}
\prod_{j=1}^2 (S_N(\bar{h}_j))^{m_j} &= \sum_{\textbf{n}\in \{0,\ldots,N-1\}^m}
\left(\prod_{i=1}^{m_1} (\bar{h}_1 \circ T^{n_i})\right) \left(\prod_{i=m_1+1}^m (\bar{h}_2 \circ T^{n_i})\right)\\
&=\sum_{\textbf{n}\in \{0,\ldots,N-1\}^m} \prod_{i=1}^{m} (\bar{h}_{j(i)} \circ T^{n_i}),
\end{align*}
where
\begin{equation*}
j(i)=\begin{cases}
1 & \text{if } 1\leq i \leq m_1,\\
2 & \text{if } m_1 + 1 \leq i \leq m.
\end{cases}
\end{equation*}

Let $ C > 0 $ denote a sufficiently large constant. For each $\textbf{n}= (n_1, \dots, n_m)\in \{0,\ldots,N-1\}^m$, define clusters as follows. There exists a minimal integer $1 \leq k=k(\textbf{n}) \leq m^2 + 1$ such that for all $i, i'\in \{1,\dots,m\},$
\begin{equation}\label{nigp}
|n_i - n_{i'}| \leq C^k \log N \quad \text{or} \quad |n_i - n_{i'}| > C^{k+1} \log N.
\end{equation}
For $l\geq 1$ we call $\{i_1,\ldots,i_l\}\subset \{1,\ldots,m\}$ a cluster of $\textbf{n}$ if it is a maximal collection with
\begin{equation*}
|n_i - n_{i'}| < C^k \log N \quad \forall i,i'\in \{i_1,\ldots,i_l\}.
\end{equation*}
In such cases, define $ l $ as the cluster length. Then there exists an integer $1\leq \eta = \eta(\textbf{n}) \leq m$ such that $ \textbf{n}$ consists of the clusters $\{i^s_1, \dots, i^s_{l_s}\}$ for $ s = 1, \dots, \eta $, with corresponding lengths $ l_s $.

Equivalently, let $ k $ be the smallest integer for which the set $\{n_i\}$ can be partitioned into $\eta$ clusters of sizes $l_s$, each being a maximal subset where pairwise differences are at most $C^k \log N$, and distinct clusters are separated by more than $C^{k+1} \log N$.

By exponential mixing, we have
\begin{equation}\label{cluexp}
\begin{aligned}
\left|\int \left(\prod_{i=1}^m (\bar{h}_{j(i)} \circ T^{n_i})\right) \d\mu \right.&\left.-\prod_{s=1}^{\eta}\int\left(\prod_{r=1}^{l_s} (\bar{h}_{j(i^s_r)} \circ T^{n_{i^s_r}})\right) \d\mu\right|\\
&\ll e^{-\gamma C^{k+1} \log N}\prod_{s=1}^\eta\left\|\prod_{r=1}^{l_s} (\bar{h}_{j(i^s_r)} \circ T^{n_{i^s_r}})\right\|_{C^{\kappa}}.
\end{aligned}
\end{equation}
Each $\bar{h}_j$ satisfies $\|\bar{h}_j\|_{C^{\kappa}} \ll N^{B'}$ for a large constant $B'>0$. Since the time differences within each cluster are at most $C^k \log N$, we obtain
\begin{align*}
\left\|\prod_{r=1}^{l_s} (\bar{h}_{j(i^s_r)} \circ T^{n_{i^s_r}})\right\|_{C^{\kappa}} &\ll\prod_{r=1}^{l_s} \| \bar{h}_{j(i^s_r)} \|_{C^{\kappa}} L^{C^k \log N}&\ll\left(N^{B'} L^{C^k \log N}\right)^{l_s},
\end{align*}
where $L=\sup_{\|f\|_{C^{\kappa}}=1} \|f\circ T\|_{C^{\kappa}}$ is the operator norm of $f\mapsto f\circ T$ on $C^{\kappa}$.

Consequently, the right-hand side of \eqref{cluexp} can be estimated as
\begin{equation}\label{proderr}
e^{-\gamma C^{k+1}\log N}N^{B'm}L^{mC^k\log N}\le N^{-100m},
\end{equation}
which holds for sufficiently large $C$ independent of $N.$ After summing over all $\mathbf n$, this error is negligible.

Considering \eqref{hbarerr} and \eqref{proderr}, the claim \eqref{momentclaim} will follow once we show
\begin{equation*}
\lim_{N\rightarrow \infty} \frac{\sum_{\textbf{n}\in \{0,\ldots,N-1\}^m} \prod_{s=1}^{\eta(\textbf{n})} \int_M \prod_{r=1}^{l_s} (\bar{h}_{j(i^s_r)} \circ T^{n_{i^s_r}}) \d\mu}{\left(B_d \rho(x^*)\right)^{\frac{m}{2}} N^{\frac{m}{2}} r_N^{\frac{md}{2}-m_2\beta }} = \omega(m_1,m_2).
\end{equation*}

Replacing once more the smooth approximations $\bar{h}_{j(i^s_r)}$ by $\bar{g}_{j(i^s_r)}$ on the left hand side, similarly to \eqref{hbarerr}, only introduces another negligible error term of order $o(1)$. Therefore, the proof will be finished once we show
\begin{equation}\label{momentsumclaim}
\lim_{N\rightarrow \infty} \frac{\sum_{\textbf{n}\in \{0,\ldots,N-1\}^m} \prod_{s=1}^{\eta(\textbf{n})} \int_M \prod_{r=1}^{l_s} (\bar{g}_{j(i^s_r)} \circ T^{n_{i^s_r}})\d\mu}{\left(B_d \rho(x^*)\right)^{\frac{m}{2}} N^{\frac{m}{2}} r_N^{\frac{md}{2}-m_2\beta }} = \omega(m_1,m_2).
\end{equation}

Step 3: Negligible cluster configurations.

Define $\Delta_\eta = \left\{\mathbf{n} : \mathbf{n} \text{ consists of } \eta \text{ clusters}\right\}.$
It holds that
\begin{equation}\label{clusternumber}
\#\Delta_{\eta} \ll \binom{m}{\eta}\eta !N^{\eta}\eta^{m-\eta} \left(2C^k\log N\right)^{m-\eta}\ll N^{\eta} (\log N)^{m}.
\end{equation}

Case 3.1: $\eta>\frac{m}{2}.$ At least one cluster is a singleton.
Since
$$\int\bar{g}_{j(i)} \circ T^{n_i} \d\mu = 0,$$
then it holds that
\begin{equation*}
\prod_{s=1}^{\eta} \int_M \prod_{r=1}^{l_s} (\bar{g}_{j(i^s_r)} \circ T^{n_{i^s_r}}) \d\mu=0.
\end{equation*}

Case 3.2: $\eta<\frac{m}{2}.$ If at least one cluster is a singleton, it can be treated as in Case 3.1.
If no cluster is a singleton, then all $l_s \geq 2$. Given that
\begin{align*}
&\|\bar{g}_1\|_{L^2} \ll r_N^{\frac{d}{2}},\quad \|\bar{g}_1\|_{L^{\infty}} \ll 1, \quad \text{ and}\\
&\|\bar{g}_2\|_{L^2} \ll r_N^{-\beta+\frac{d}{2}},\quad \|\bar{g}_2\|_{L^{\infty}} \ll r_N^{-\beta},
\end{align*}

we have
\begin{equation}\label{clusterint}
\left|\prod_{s=1}^{\eta}\int\prod_{r=1}^{l_s}(\bar{g}_{j(i^s_r)} \circ T^{n_{i^s_r}})\d\mu\right|\ll
r_N^{-m_2 \beta + \eta d}.
\end{equation}
Thus
\begin{equation*}
\begin{aligned}
\sum_{\eta<m/2} \sum_{\mathbf{n} \in \Delta_\eta}
\left|\prod_{s=1}^{\eta}\int\prod_{r=1}^{l_s}(\bar{g}_{j(i^s_r)} \circ T^{n_{i^s_r}})\d\mu\right|
&\ll\sum_{\eta<m/2}N^{\eta} (\log N)^m r_N^{-m_2 \beta + \eta d}\\
&\ll\sum_{\eta<m/2} \left(N r_N^d\right)^\eta r_N^{- m_2 \beta}(\log N)^m.
\end{aligned}
\end{equation*}
Subcase 3.2~(i): $N r_N^d>N^{1/100}$. This bound is of order $o\left(N^{\frac{m}{2}} r_N^{\frac{md}{2} - m_2 \beta}\right).$

Subcase 3.2~(ii): $N r_N^d\leq N^{1/100}$. For each index $\textbf{n}\in \Delta_{\eta}$ we distinguish two cases:

Subcase 3.2~(ii-a): all $n_i$ within the same cluster are equal for each cluster.
Then there exist $O(N^\eta)$ such indices \textbf{n}. Taking into account the estimate in \eqref{clusterint}, the total contribution from these indices is of order $O\left(N^{\eta} r_N^{\eta d - m_2 \beta}\right).$

Subcase 3.2~(ii-b): at least one cluster contains two distinct indices. Then Lemma \ref{slowrecanlem} provides an additional saving. Using the estimates established in Case~4.1~(B) below,
\eqref{clusterint} can be replaced by
\begin{align*}
\left| \prod_{s=1}^{\eta} \int \prod_{r=1}^{l_s} (\bar{g}_{j(i^s_r)} \circ T^{n_{i^s_r}} )\d\mu \right| &\ll
r_N^{\eta d - m_2 \beta} |\log r_N|^{-100m} \\
&\ll
r_N^{\eta d - m_2 \beta} (\log N)^{-100m}.
\end{align*}
Note that here the assumption $N r_N^d\leq N^{\frac{1}{100}}$ is used crucially to show $\log N \asymp |\log r_N|$.

Since the number of indices $\mathbf{n}$ in this subcase is at most $\#\Delta_{\eta}\ll N^{\eta} (\log N)^m$, the total contribution from this case is of order
\begin{equation*}
O\left(N^{\eta} r_N^{\eta d - m_2 \beta} (\log N)^{-99m}\right).
\end{equation*}

Combining subcases 3.2(ii-a) and 3.2(ii-b), it holds that
\begin{align*}
\sum_{\eta < \frac{m}{2}} \sum_{\textbf{n} \in \Delta_{\eta}} \left| \prod_{s=1}^{\eta} \int \prod_{r=1}^{l_s} (\bar{g}_{j(i^s_r)} \circ T^{n_{i^s_r}}) \d\mu\right| &\ll \sum_{\eta < \frac{m}{2}} N^{\eta} r_N^{\eta d - m_2 \beta}.
\end{align*}
Using the assumption $N r_N^d \rightarrow\infty$, it follows that $N^{\eta} r_N^{\eta d - m_2 \beta} =o\left( N^{\frac{m}{2}} r_N^{\frac{md}{2} - m_2 \beta}\right)$ for each $\eta<\frac{m}{2}$. Therefore the above bound is of order $o\left( N^{\frac{m}{2}} r_N^{\frac{md}{2} - m_2 \beta}\right)$.

In conclusion, all terms in \eqref{momentsumclaim} for $\eta \neq \frac{m}{2}$ are asymptotically negligible. Note that, in particular, if $m$ is odd, then \eqref{momentsumclaim} is satisfied. So for the rest of the proof, assume $m$ is even.

Step 4: The critical case $\eta=\frac{m}{2}$.

Unless all the clusters are of length $2$, $\textbf{n}$ must have a singleton. Define
$$
\Gamma=\left\{\mathbf{n}\in \Delta_{\frac{m}{2}} : \text{each cluster has size } 2\right\}.
$$
Then we consider
$\sum_{\mathbf{n}\in\Gamma}\prod_{s=1}^{\frac{m}{2}}\int \big(\bar{g}_{j(i^s_1)} \cdot \bar{g}_{j(i^s_2)} \circ T^{|n_{i^s_2}-n_{i^s_1}|}\big)\d\mu.$

Since all terms with singletons are negligible -- as in Case 3.1 -- it is enough to show
\begin{equation}\label{momentgammaclaim}
\lim_{N\rightarrow \infty} \frac{\sum_{\textbf{n}\in \Gamma} \prod_{s=1}^{\frac{m}{2}} \int_M \prod_{r=1}^{2} (\bar{g}_{j(i^s_r)} \circ T^{n_{i^s_r}}) \d\mu}{\left(B_d \rho(x^*)\right)^{\frac{m}{2}} N^{\frac{m}{2}} r_N^{\frac{md}{2}-m_2\beta }} = \omega(m_1,m_2).
\end{equation}

Step 4.1: Estimates for a single pair.

For $n_i,n_{i'}$ in the same cluster, we aim to estimate $\int_M (\bar{g}_{j(i)} \circ T^{n_i})(\bar{g}_{j(i')} \circ T^{n_{i'}})\d\mu$. Simplifying notation, we consider
\begin{equation*}
\int_M \bar{g}_j \cdot \bar{g}_{j'} \circ T^k\d\mu,
\end{equation*}
for $j,j' \in \{1,2\}$ and $k=|n_i - n_{i'}| \leq C^{m^2+1} \log N,$ see \eqref{nigp}. Below we will distinguish several cases.

We first highlight a technical point. Since $g_1 = \one_{A_N}$, the slow recurrence estimate, specifically Lemma~\ref{slowrecanlem}, is applicable only up to times $k=C' |\log r_N|$ for a constant $C'>0$, which might be small if $r_N$ decays substantially slower than $1/N$.

Case 4.1~(A): $k=0$. We consider the cases where $j=j'=1,$ $j=j'=2$, and $j\neq j'$ respectively. For $j=j'=1$ we use Lemma~\ref{regadaptmeasure} to obtain
\begin{equation}\label{mainest1}
\int \bar{g}_1^2\d\mu\sim \mu(A_N)\sim B_d \rho(x^*) r_N^d.
\end{equation}
For $j=j'=2$, a straightforward calculation in Euclidean space reveals
\begin{equation}\label{mainest2}
\begin{aligned}
\int_M \bar{g}_2^2 \d\mu &\sim\int_{r_N \leq d(x,x^*) < 1} f^2(x)\d\mu(x)\\
&\sim \rho(x^*)\int_{\textbf{y}\in \R^d, r_N < \|\textbf{y}\| < 1} \frac{1}{\|\textbf{y}\|^{2\beta}} \d \textbf{y}\\
&\sim B_d \rho(x^*) r_N^{-2\beta+d} \int_1^{\infty} u^{-2\alpha} \d u \\
&\sim \sigma^2 B_d \rho(x^*) r_N^{-2\beta+d},
\end{aligned}
\end{equation}
where $\sigma^2=\frac{1}{2\alpha-1}$.

A straightforward computation in local coordinates near $x^*$, together with the boundedness of $f$ away from $x^*$, gives
\begin{equation*}
\|g_2\|_{L^1}\ll \begin{cases}
1 & \text{if } \beta<d\\
|\log r_N| & \text{if } \beta=d\\
r_N^{d-\beta} & \text{if } \beta>d.
\end{cases}
\end{equation*}
For $j\neq j'$, since the supports of $g_1$ and $g_2$ are disjoint, we therefore have
\begin{align*}
\left|\int_M \bar g_1\bar g_2\d\mu\right|=\left|\int_M g_1\d\mu\right|\left|\int_M g_2\d\mu\right|
&\ll r_N^{d}\left(1+|\log r_N|+r_N^{d-\beta}\right)\\
&=r_N^{d-\beta} \left(r_N^{\beta}(1+|\log r_N|) + r_N^d\right).
\end{align*}

Case 4.1~(B): $0<k<C' \left|\log r_N \right|,$ where $C'>0$ is a sufficiently large constant. For $j=j'=1$, applying Lemma~\ref{slowrecanlem} to $A_N \cap T^{-k}(A_N)$ with $A=100m$ yields
\begin{equation*}
\left|\int_M \bar{g}_1 \cdot \bar{g}_1 \circ T^{k} \d\mu\right|\ll \mu\left(A_N \cap T^{-k}(A_N)\right)
\ll r_N^d \left|\log r_N \right|^{-100m}.
\end{equation*}

For $j=j'=2$, Lemma \ref{truncvarlem} with $A=100m$ yields
\begin{equation}
\left| \int_M \bar{g}_2 \cdot \bar{g}_2 \circ T^k \d\mu \right| \ll r_N^{-2\beta+d} |\log r_N|^{-100m}.
\end{equation}

For the cross term $j\neq j'$, we consider the case $j=1$ and $j'=2$;
the reverse situation follows analogously.

Let $\hat{r}_N=r_N|\log r_N|^H$, where $H>\frac{100m}{\beta}$ is fixed. Since $A_N$ is regularly adapted, for sufficiently large $N$ we have $A_N\subset B_{\hat{r}_N}(x^*)$. Hence,
\begin{equation*}
\int_M g_1\cdot g_2\circ T^k\d\mu
=
\int_{A_N\cap T^{-k}(M\setminus A_N)}f\circ T^k\d\mu
=
\mathfrak A+\mathfrak B,
\end{equation*}
where
\begin{align*}
\mathfrak A&:=
\int_{A_N\cap T^{-k}(B_{\hat{r}_N}(x^*)\setminus A_N)}
f\circ T^k\d\mu,\\
\mathfrak B&:=
\int_{A_N\cap T^{-k}(M\setminus B_{\hat{r}_N}(x^*))}
f\circ T^k\d\mu .
\end{align*}

For $\mathfrak A$, note that for sufficiently large $N$, $|\log r_N|\leq 2|\log\hat{r}_N|$. Since $k<C'|\log r_N|$, we therefore have $k<2C'|\log\hat r_N|.$
Using slow recurrence with exponent $A'=Hd+100m$ and the inclusion $B_{r_N}(x^*)\subset A_N\subset B_{\hat r_N}(x^*),$
we obtain
\begin{align*}
\mathfrak A
&\leq
\sup_{y\notin B_{r_N}(x^*)}|f(y)|
\mu(B_{\hat{r}_N}(x^*)\cap T^{-k}B_{\hat{r}_N}(x^*))\\
&\ll
r_N^{-\beta}\hat{r}_N^d|\log\hat{r}_N|^{-A'}\\
&\ll
r_N^{d-\beta}|\log r_N|^{-100m}.
\end{align*}

For $\mathfrak B$, using Lemma~\ref{regadaptmeasure},
$$\mathfrak B\leq\sup_{y\notin B_{\hat{r}_N}(x^*)}|f(y)|\,\mu(A_N)\ll
\hat{r}_N^{-\beta}r_N^d\ll r_N^{d-\beta}|\log r_N|^{-100m}.
$$

Combining the estimates for $\mathfrak A$ and $\mathfrak B$ gives
\begin{equation*}
\int_M g_1\cdot g_2\circ T^k\d\mu
\ll
r_N^{d-\beta}|\log r_N|^{-100m}.
\end{equation*}
Finally, centering does not affect the estimate, so for $j\neq j'$,
\begin{equation*}
\left|\int_M\bar g_j\cdot\bar g_{j'}\circ T^k\d\mu\right|
\ll
r_N^{d-\beta}|\log r_N|^{-100m}.
\end{equation*}

Case 4.1~(C): $k> C' \left|\log r_N\right|.$ We emphasize that this case only needs to be considered when $|\log r_N|=o(\log N)$. Indeed, all computations here are still inside a cluster, so the relevant time differences satisfy $k=O(\log N)$. If $|\log r_N|$ is comparable to $\log N$, then taking $C'$ large enough makes the condition $k>C'|\log r_N|$ incompatible with the intra-cluster bound $k=O(\log N)$, and this subcase becomes empty. Using \eqref{regasm2} with some $\varepsilon\in (0,r_N)$ to be determined later, there exists a $C^{\kappa}$ smooth function $\theta:M\rightarrow[0,1]$ such that
\begin{equation*}
\|g_1-\theta\|_{L^1} \ll\varepsilon r_N^a \quad \text{and} \quad \|\theta\|_{C^{\kappa}} \ll\varepsilon^{-\kappa} r_N^{-a'}.
\end{equation*}
Denote $\bar{\theta}=\theta - \int_M \theta\d\mu$. By applying exponential mixing, we obtain
\begin{align*}
\left| \int_M \bar{g}_{1} \cdot \bar{g}_{1} \circ T^{k} \d\mu \right| &\ll \left| \int_M \bar{\theta} \cdot \bar{\theta} \circ T^{k} \d\mu \right| + \varepsilon r_N^{a} \\
& \ll e^{-\gamma k} \varepsilon^{-2\kappa} r_N^{-2a'} +\varepsilon r_N^a.
\end{align*}
Choosing $\varepsilon=e^{-\frac{\gamma k}{2\kappa+1}}$ guarantees $\varepsilon<r_N$ whenever $C'>(2\kappa+1)/\gamma$, since $k>C'\left|\log r_N\right|$. We then obtain
\begin{equation*}
\left| \int_M \bar{g}_{1} \cdot \bar{g}_{1} \circ T^{k} \d\mu \right| \ll e^{-\frac{\gamma k}{2\kappa+1}} r_N^{-2a'}.
\end{equation*}
If $k=C'|\log r_N| + i$ for $i\in\mathbb{N},$ then we have
\begin{equation*}
\left| \int_M \bar{g}_{1} \cdot \bar{g}_{1} \circ T^{k} \d\mu \right| \ll r_N^{\frac{\gamma C'}{2\kappa+1}-2a'}e^{-\frac{\gamma i}{2\kappa+1}}\ll r_N^{100d} e^{-\frac{\gamma i}{2\kappa+1}}
\end{equation*}
for large $C'$.

The same argument, with the corresponding smooth approximations of $g_2$ and of the cross terms, gives
\begin{equation*}
\left| \int_M \bar{g}_{2} \cdot \bar{g}_{2} \circ T^{k} \d\mu \right| \ll r_N^{100d} e^{-\frac{\gamma i}{2\kappa+1}}
\end{equation*}
and
\begin{equation*}
\left| \int_M \bar{g}_{j} \cdot \bar{g}_{j'} \circ T^{k} \d\mu \right| \ll r_N^{100d} e^{-\frac{\gamma i}{2\kappa+1}},
\end{equation*}
for $j\neq j'$.

Step 4.2: Negligibility outside the diagonal pair configurations.

Let
\begin{equation*}
\Gamma'=\left\{\textbf{n} \in \Gamma \;|\; n_{i^s_1}=n_{i^s_2} \text{ and } j(i^s_1)=j(i^s_2) \;\; \forall s=1,\ldots,\frac{m}{2}\right\},
\end{equation*}
then, using the estimates from $j\neq j'$ in case 4.1~(A), and cases 4.1~(B) and 4.1~(C) above, we obtain
\begin{align*}
&\sum_{\textbf{n} \in \Gamma\setminus \Gamma'} \prod_{s=1}^{\frac{m}{2}} \left|\int_M \bar{g}_{j(i^s_1)} \cdot \bar{g}_{j(i^s_2)} \circ T^{n_{i^s_2}-n_{i^s_1}} \d\mu\right|\\
&\ll N^{\frac{m}{2}} r_N^{\frac{md}{2}-m_2 \beta}\left(r_N^\beta(1+|\log r_N|)+r_N^d \right)
+ N^{\frac{m}{2}} \left|\log r_N \right|^{\frac{m}{2}} r_N^{\frac{md}{2} - m_2 \beta} \left|\log r_N \right|^{-100m}\\
&\ll N^{\frac{m}{2}} r_N^{\frac{md}{2}-m_2 \beta} \left|\log r_N \right|^{-99m}.
\end{align*}
Note that here we used the fact that all cases, where one of the $|n_{i^s_1}-n_{i^s_2}|>C'|\log r_N|$, are negligible by summing the estimate for $i$ from Case 4.1~(C).

Step 4.3: The diagonal pair contribution.

Using \eqref{mainest1} and \eqref{mainest2}, for $\mathbf{n}\in \Gamma'$ it holds that
$$
\prod_{s=1}^{\frac{m}{2}} \int_M \bar{g}_{j(i^s_1)}^2 \d\mu\sim\left(B_d \rho(x^*)\right)^{\frac{m}{2}} r_N^{\frac{md}{2}-m_2\beta } \sigma^{m_2}.
$$
If either $m_1$ or $m_2$ is odd, then $\Gamma'=\varnothing$, since indices of type $g_1$ and $g_2$ must be paired separately. Otherwise, there are $(m_1-1)!!$ and $(m_2-1)!!$ possible pairings within each type. For each fixed pairing, this leaves $m/2$ free time variables, hence $N^{\frac{m}{2}}$ choices. The cluster-separation condition excludes only configurations where two such times lie within distance $O(\log N)$, whose number is $O(N^{\frac{m}{2}-1}\log N)$. Hence
$$\#\Gamma'\sim N^{\frac{m}{2}}(m_1-1)!!(m_2-1)!!.
$$
It follows that
$$
\sum_{\mathbf{n} \in \Gamma'} \prod_{s=1}^{\frac{m}{2}} \int_M \bar{g}_{j(i^s_1)}^2 \d\mu \sim \left(B_d \rho(x^*)\right)^{\frac{m}{2}} N^{\frac{m}{2}} r_N^{\frac{md}{2}-m_2\beta } \omega(m_1,m_2),
$$
which completes the proof of \eqref{momentgammaclaim}.
\end{proof}

\section{Lightly trimmed SLLN}\label{LT}
In this section, we assume that the system $(M, \mathcal{B}(M), \mu, T)$ is exponentially mixing, and that $\Ord_{x^*}(f) = d$. We establish a lightly trimmed SLLN, which proves Theorem \ref{lightSLLN}. In light of Lemma \ref{strongsklem}, it suffices to show a SLLN for $\hat{S}_N^K$.

Define
$$
r_N=\left(\frac{(\log \log N)^6}{N \log N}\right)^{\frac{1}{d}}\quad \text{and}\quad f_N=f \cdot \left(1-\one_{B_{r_N}(x^*)}\right).
$$
The subsequent lemma bounds the probability that the trimmed sum differs from the original sum by more than a given threshold.

\begin{lemma}[Trimmed Sum Approximation Error]\label{TSAR}
Under the hypothesis above, we have
\begin{equation}\label{SumkDf}
\mu\left(\left| \hat{S}_N^K(f) - S_N(f_N) \right|>\frac{KC_0N\log N}{(\log \log N)^6}\right) \ll (\log N)^{-\frac{3}{2}},
\end{equation}
where the constant $C_0$ is defined in \eqref{DeCf}.
\end{lemma}
\begin{proof}
We observe that
$$\left| \hat{S}_N^K(f) - S_N(f_N) \right| = \left| \sum_{n=0}^{N-1} f(T^n x) \one_{B_{r_N}(x^*)}(T^n x) - \sum_{i=1}^K f(\hat{x}_i) \right|.$$

If $S_N(\one_{B_{r_N}(x^*)})(x)=m\leq K,$ then $\hat{x}_1,\ldots,\hat{x}_m \in B_{r_N}(x^*)$, where $\hat{x}_j$ is the $j$th closest orbit point to $x^*$ as defined in Section~\ref{nearsec}, and consequently,
\begin{equation*}
\left| \hat{S}_N^K(f)(x) - S_N(f_N)(x) \right| \leq \sum_{i=m+1}^K f(\hat{x}_i) \leq \frac{KC_0N\log N}{(\log \log N)^6}.
\end{equation*}
It follows that
$$\left\{\left|\hat{S}_N^K(f) - S_N(f_N)\right|>\frac{KC_0N\log N}{(\log \log N)^6}\right\}\subset \left\{S_N(\one_{B_{r_N}(x^*)}) \geq K+1\right\}.$$

Applying Markov's inequality, we obtain for sufficiently large $N,$
\begin{equation}\label{TriDf}
\begin{aligned}
\mu&\left( \left| \hat{S}_N^K(f) - S_N(f_N) \right| >\frac{KC_0N\log N}{(\log \log N)^6}\right)\\
&\leq\mu\left( S_N(\one_{B_{r_N}(x^*)}) \geq K+1 \right) \\
&\leq \mu\left(S_N(\one_{B_{r_N}(x^*)})\left(S_N(\one_{B_{r_N}(x^*)})-1\right)\geq (K+1)K\right)\\
&\leq \frac{2}{(K+1)K}\E\left(\binom{S_N(\one_{B_{r_N}(x^*)})}{2}\right).
\end{aligned}
\end{equation}
We compute the  second factorial moment
$$
\mathbb{E}\left( \binom{S_N(\mathbf{1}_{B_{r_N}(x^*)})}{2} \right) = \sum_{0 \leq j < k \leq N-1} \mu\left( T^{-j} B_{r_N}(x^*) \cap T^{-k} B_{r_N}(x^*) \right)
$$
and consider two cases based on the distance between indices $j$ and $k.$

Case 1 (Long Range): If $k-j > C \log N$ for some large constant $C>0$, using Lemma \ref{MixEmB} and the fact that $\mu(B_{r_N}(x^*)) \asymp (\log \log N)^6 / (N \log N)$, we have
$$
\mu\left( T^{-j} B_{r_N}(x^*) \cap T^{-k} B_{r_N}(x^*) \right) = \mu(B_{r_N}(x^*))^2 + O\left( \mu(B_{r_N}(x^*))^{\frac{d-1}{d} \cdot \frac{2\kappa}{2\kappa+1}} N^{-\frac{C \gamma}{2\kappa+1}} \right).
$$
Thus, the contribution from long-range pairs is $O\left( (\log \log N)^{12} / (N^2(\log N)^2) \right).$
Since the number of such pairs is at most $N^2$, their total contribution is
$$O\left( \frac{(\log \log N)^{12}}{(\log N)^2} \right).$$

Case 2 (Short Range): By the slow recurrence property with $A = 3$, for large $N$, we obtain
$$
\mu\left( T^{-j} B_{r_N}(x^*) \cap T^{-k} B_{r_N}(x^*) \right) \ll \mu(B_{r_N}(x^*)) (\log N)^{-2} \ll \frac{(\log\log N)^6}{N (\log N)^3}.
$$
The number of such pairs is at most $C N \log N$, so their total contribution is
\begin{equation*}
O\left( \frac{(\log\log N)^6}{(\log N)^2} \right).
\end{equation*}

Combining both cases, we have
$$
\mathbb{E}\left( \binom{S_N(\mathbf{1}_{B_{r_N}(x^*)})}{2} \right) \ll \frac{(\log \log N)^{12}}{(\log N)^2} + \frac{(\log\log N)^6}{(\log N)^2} \ll (\log N)^{-\frac{3}{2}}.
$$

Finally, we complete the proof using the inequality \eqref{TriDf}.
\end{proof}

\begin{lemma}[Ergodic Sum of the Truncated Function]\label{ESTF}
There exists a sequence $\epsilon_N \to 0$ such that
$$
\mu\left( \left| S_N(f_N)-\operatorname{Res}_{x^*}(f)B_d \rho(x^*) N \log N \right| > \epsilon_N N \log N \right) \ll (\log N)^{-1} (\log \log N)^{-4}.
$$
\end{lemma}

\begin{proof}
First note that
$$\E f_N \sim \operatorname{Res}_{x^*}(f)B_d \rho(x^*) \log N.$$
Let $\varepsilon_N \to 0$ be chosen such that
\begin{equation}\label{Etdf}
\left|\mathbb{E}f_N- \operatorname{Res}_{x^*}(f)B_d \rho(x^*) \log N\right|<\varepsilon_N\log N.
\end{equation}
Letting $\epsilon_N=2\varepsilon_N + (\log\log N)^{-1}$, Chebyshev's inequality yields
\begin{align*}
& \mu\left( \left| S_N(f_N)- \operatorname{Res}_{x^*}(f)B_d \rho(x^*) N \log N \right| > \epsilon_N N \log N \right)\\
\leq & \mu \left( \Bigl| S_N(f_N) - \E S_N(f_N) \Bigr| > \frac{\epsilon_N N \log N}{2} \right)
\ll \frac{\Var\Bigl(S_N(f_N)\Bigr)}{(\epsilon_N N \log N)^2}.
\end{align*}

Applying \eqref{lcmb} with $s=2$ and
$\beta=d$, we obtain
$$\E f_N^2=\int_{M\setminus B_{r_N}(x^*)}f^2\d\mu\ll r_N^{-d}.
$$
For $1\le k\le C'|\log r_N|$, Lemma~\ref{truncvarlem} with
$A=3$ gives
$$\E(\bar{f}_N\cdot\bar{f}_N\circ T^k)\le \E(f_N \cdot f_N\circ T^k)\ll r_N^{-d}|\log r_N|^{-3},
$$
where $\bar{f}_N=f_N-\E f_N.$

For $k=C'|\log r_N|+i$, $i\ge1$, it follows from Case~4.1~(C) in the proof of Proposition~\ref{momentprop} that for $C'$ sufficiently large,
$$\E(\bar{f}_N\cdot\bar{f}_N\circ T^k)\ll r_N^{100d} e^{-\frac{\gamma i}{2\kappa+1}}.
$$
Consequently,
\begin{align*}
\operatorname{Var}(S_N(f_N))&=N\Var(f_N)
+2\sum_{k=1}^{N-1}(N-k)\E(\bar{f}_N\cdot\bar{f}_N\circ T^k)\\
&\ll Nr_N^{-d}+N\sum_{1\le k\le C'|\log r_N|}r_N^{-d}|\log r_N|^{-3}+N\sum_{i\ge1}
r_N^{100d}e^{-\frac{\gamma i}{2\kappa+1}}\\
&\ll Nr_N^{-d}\ll\frac{N^2 \log N}{(\log \log N)^6},
\end{align*}
and we conclude
\begin{align*}
\mu&\left(\left|S_N(f_N)- \Res_{x^*}(f)B_d \rho(x^*) N \log N\right|>\epsilon_N N \log N \right)\\
&\ll \frac{N^2 \log N}{(\log \log N)^6(\epsilon_N N \log N)^2}\\
&\ll (\log N)^{-1} (\log \log N)^{-4}.
\end{align*}
\end{proof}

\begin{proof}[Proof of Theorem \ref{lightSLLN}]
By Lemmas \ref{TSAR} and \ref{ESTF}, for $$\omega_N=\epsilon_N+\frac{KC_0}{(\log\log N)^6}\to 0$$ we obtain
$$
\mu\left(\biggl| \frac{\hat{S}_N^K(f)}{N \log N}-\operatorname{Res}_{x^*}(f)B_d \rho(x^*) \biggr| >\omega_N \right) \ll (\log N)^{-1} (\log \log N)^{-4}.
$$
Now consider the subsequence defined by
$$N_l = \left\lceil \exp\left(\frac{l}{(\log l)^2}\right) \right\rceil,$$
from the above we obtain
\begin{align*}
\sum_{l=2}^\infty & \mu\left(\biggl|\frac{\hat{S}_{N_l}^K(f)}{N_l \log N_l}-\operatorname{Res}_{x^*}(f)B_d \rho(x^*) \biggr|>\omega_{N_l} \right)\\
& \ll \sum_{l=2}^\infty (\log N_l)^{-1}(\log\log N_l)^{-4}\\
&\ll\sum_{l=2}^{\infty} l^{-1}(\log l)^{-2}\\
&< \infty.
\end{align*}
Therefore, by the first Borel--Cantelli lemma, for $\mu$-almost every $x,$
$$
\lim_{l\to \infty} \frac{\hat{S}_{N_l}^K(f)(x)}{N_l \log N_l} = \Res_{x^*}(f)B_d \rho(x^*).
$$

For arbitrary $N \in [N_l, N_{l+1}],$ the monotonicity of the trimmed sum $\hat{S}_N^K(f)$ implies
$$
\frac{\hat{S}_{N_l}^K(f)}{N_{l+1} \log N_{l+1}} \leq \frac{\hat{S}_N^K(f)}{N \log N} \leq \frac{\hat{S}_{N_{l+1}}^K(f)}{N_l \log N_l}.
$$
Taking the limit as $l \to \infty$, and observing that $\lim_{l \to \infty}N_{l+1}/N_l = 1$, we conclude that
$$
\lim_{N \to \infty} \frac{\hat{S}_N^K(f)(x)}{N \log N} = \operatorname{Res}_{x^*}(f)B_d \rho(x^*)
$$
for $\mu$-almost every $x$. Finally, Lemma \ref{strongsklem} implies that also
$$
\lim_{N \to \infty} \frac{S_N^K(f)(x)}{N \log N} = \operatorname{Res}_{x^*}(f)B_d \rho(x^*).
$$
\end{proof}

\begin{proof}[Proof of Corollary~\ref{SNCP}]
Theorem~\ref{lightSLLN} establishes that the normalized trimmed sum $S_N^K(f)/(N \log N)$ converges almost surely to $\operatorname{Res}_{x^*}(f)B_d \rho(x^*) $. To pass from trimmed to untrimmed sums, we observe that the contribution from the $K$ largest terms is negligible in probability:

\begin{align*}
&\mu\left(\sum_{i=1}^{K} f(x_i^N)> KN \log \log N \right)\leq \mu\left(\max_{0\le n<N}f(T^nx)>N\log\log N
\right)\\ &\leq\sum_{n=0}^{N-1}\mu\left(f\circ T^n>N\log\log N\right)=
N\mu\left(f>N\log\log N\right)\ll (\log \log N)^{-1} \to 0.
\end{align*}
Therefore, the difference between $S_N(f)/(N \log N)$ and $S_N^K(f)/(N \log N)$ vanishes in probability, and the full sum shares the same limit in probability, that is
$$
\lim_{N \to \infty} \mu\left( \biggl| \frac{S_N(f)}{N \log N}-\operatorname{Res}_{x^*}(f)B_d \rho(x^*)\biggr|> \varepsilon \right) = 0, \quad \forall \varepsilon > 0.
$$
\end{proof}

\section{Intermediately trimmed SLLN}\label{IT}
Suppose that the system $(M, \mathcal{B}(M), \mu, T)$ is exponentially mixing of all orders, and that $\Ord_{x^*}(f) = \beta > d$. We establish an intermediately trimmed SLLN for ergodic sums, thereby proving Theorem \ref{interSLLN}. According to Lemma~\ref{strongsklem}, it suffices to prove the SLLN for $\hat{S}_N^{k(N)}$.

The following lemmas are moment estimates that follow directly from Proposition \ref{momentprop}.

\begin{lemma}\label{momentlem}
Let $r_N\to0$ satisfy $N r_N^{d} \rightarrow \infty$ and put
$$f_N=f \cdot \left(1-\one_{B_{r_N}(x^*)}\right),\quad\bar{f}_N=f_N-\E f_N.
$$
Then, for every fixed $m\ge 1$,
$$
\E\left( \left|S_N(\bar{f}_N)\right|^m \right) \ll N^{\frac{m}{2}} r_N^{\frac{md}{2}-m\beta}.
$$
\end{lemma}

\begin{proof}
For even $m$, the conclusion follows from Proposition~\ref{momentprop} applied with the regularly adapted sequence $A_N=B_{r_N}(x^*),$ and the choices $m_1=0,$ $m_2=m.$
If $m$ is odd, set $q=m+1$. Then
$$\E\left(\left|S_N(\bar f_N)\right|^m\right)\le
\left(\E\left(\left|S_N(\bar f_N)\right|^q\right)\right)^{\frac{m}{q}}.
$$
Applying the even case to $q$ gives
$$\E\left(\left|S_N(\bar f_N)\right|^m\right)\ll\left(N^{\frac{q}{2}} r_N^{\frac{qd}{2}-q\beta}\right)^{\frac{m}{q}}=N^{\frac{m}{2}} r_N^{\frac{md}{2}-m\beta},$$
which proves the lemma.
\end{proof}

\begin{lemma}\label{momentlemind}
Let $r_N\to0$ with $N r_N^{d} \rightarrow \infty$ and put
$$\bar{\chi}_N=\one_{B_{r_N}(x^*)}-\mu\bigl(B_{r_N}(x^*)\bigr).$$
Then, for every fixed $m\ge1$,
\begin{equation*}
\E\left(\bigl|S_N(\bar{\chi}_N)\bigr|^m\right) \ll N^{\frac{m}{2}} r_N^{\frac{md}{2}}.
\end{equation*}
\end{lemma}

\begin{proof}
For even $m$, the result follows by applying Proposition~\ref{momentprop} with the regularly adapted sequence $A_N=B_{r_N}(x^*)$ and parameters $m_1=m,m_2=0$. The odd case then follows from H\"{o}lder's inequality, exactly as in the proof of Lemma~\ref{momentlem}.
\end{proof}

For the remainder of this section, let $k(N)\rightarrow \infty$ be a sequence with $k(N)=o(N)$, and define $r_N>0$ by
\begin{equation*}
\mu\bigl(B_{r_N}(x^*)\bigr)=\frac{k(N)}{N}.
\end{equation*}

\begin{lemma}[Trimmed Sum Approximation Error]\label{interTSAR}
Let $f_N=f \cdot \left(1 - \one_{B_{r_N}(x^*)} \right).$ Then, for every $A>0$ and small $\epsilon>0,$ there exists a constant $C>0$ such that
\begin{equation*}
\mu\left(\left| \hat{S}_N^{k(N)}(f) - S_N(f_N) \right|>C N^{\alpha} k(N)^{1-\alpha-\epsilon}\right) \ll k(N)^{-A}.
\end{equation*}
\end{lemma}

\begin{proof}
By the continuity of the density function of $\mu,$ we have
$$r_N^d\sim \frac{k(N)}{B_d \rho(x^*) N}.$$ Define the event
$$
E_N=\left\{\left|S_N\bigl(\one_{B_{r_N}(x^*)}\bigr)-k(N)\right|<k(N)^{1-\epsilon} \right\}\cap\left\{ S_N\bigl(\one_{B_{r_N/2}(x^*)}\bigr)<k(N) \right\}.
$$
On $E_N$, the difference between the trimmed sum and the truncated sum satisfies
$$
\left|\hat{S}_N^{k(N)}(f)-S_N(f_N)\right|<C_02^{\beta} r_N^{-\beta} k(N)^{1-\epsilon}< C N^{\alpha} k(N)^{1-\alpha - \epsilon},
$$
where $C_0 > 0$ is the constant from \eqref{DeCf} and $C=C_0(2 B_d\rho(x^*)+10)^{\beta}.$
It follows that
$$\left\{\left|\hat{S}_N^{k(N)}(f)-S_N(f_N) \right|>C N^{\alpha} k(N)^{1-\alpha - \epsilon}\right\}\subset E_N^c.$$
The probability of the complement $E_N^c$ is bounded by
\begin{equation}\label{EsCE}
\begin{aligned}
\mu(E_N^c)&\leq \mu\left(\left|S_N\left(\one_{B_{r_N}(x^*)}\right) - k(N)\right| \geq k(N)^{1-\epsilon} \right)\\
& + \mu\left( \left|S_N\left(\one_{B_{r_N/2}(x^*)}\right) - N \mu\bigl(B_{r_N/2}(x^*)\bigr) \right| \geq k(N) - N \mu\bigl(B_{r_N/2}(x^*)\bigr) \right)\\
&\leq \mu\left(\left|S_N\left(\one_{B_{r_N}(x^*)}\right)-k(N)\right| \geq k(N)^{1-\epsilon} \right)\\
&+ \mu\left(\left|S_N\left(\one_{B_{r_N/2}(x^*)}\right) - N\mu\bigl(B_{r_N/2}(x^*)\bigr)\right| \geq \tfrac{k(N)}{3} \right),
\end{aligned}
\end{equation}
where in the last step we used the fact that $N \mu(B_{r_N/2}(x^*))\sim 2^{-d} k(N)$.

By Lemma \ref{momentlemind}, for any integer $m \geq 2$,
\begin{align*}
\E\left( \left|S_N\left(\one_{B_{r_N}(x^*)}\right)-k(N)\right|^m \right)& \ll k(N)^{\frac{m}{2}} \quad \text{and} \\
\E \left(\left|S_N\left(\one_{B_{r_N/2}(x^*)}\right)-N\mu\bigl(B_{r_N/2}(x^*)\bigr)\right|^m\right)& \ll k(N)^{\frac{m}{2}}.
\end{align*}
Therefore, applying Markov's inequality to \eqref{EsCE}, we deduce
$$
\mu(E_N^c) \ll k(N)^{-\frac{m}{2}+m\epsilon} + k(N)^{-\frac{m}{2}} \ll k(N)^{-A},
$$
provided that $m>2A/(1-2\epsilon).$ This completes the proof.
\end{proof}

\begin{lemma}[Ergodic Sum of the Truncated Function]\label{interESTF}
Under the hypotheses of Lemma \ref{interTSAR}, for any $\epsilon>0$ and $A>0,$
\begin{equation*}
\mu\left(\left|S_N(f_N)-\E (S_N(f_N)) \right| > N^{\alpha} k(N)^{1-\alpha-\epsilon}\right) \ll k(N)^{-A}.
\end{equation*}
\end{lemma}
\begin{proof}
For $m\geq 2$, Lemma \ref{momentlem} yields
\begin{equation*}
\E(|S_N(\bar{f}_N)|^m) \ll N^{\frac{m}{2}} r_N^{- m\beta + \frac{md}{2}} \ll N^{m\alpha} k(N)^{\frac{m}{2}-m\alpha},
\end{equation*}
where $\bar{f}_N = f_N - \E f_N$. Using Markov's inequality, it follows that
\begin{align*}
\mu\left( \left| S_N(f_N) -\E S_N(f_N) \right| > N^{\alpha} k(N)^{1-\alpha-\epsilon} \right)
& \ll \frac{\E\left(\left|S_N(\bar{f}_N)\right|^m\right)}{N^{m\alpha} k(N)^{m-m\alpha-m\epsilon}}\\
& \ll k(N)^{-\frac{m}{2} + m \epsilon}.
\end{align*}
The desired result follows by selecting an integer $m>2A/(1-2\epsilon).$
\end{proof}

\begin{proof}[Proof of Theorem \ref{interSLLN}]
Using Lemmas \ref{interTSAR} and \ref{interESTF}, for $A>0$, we have
\begin{equation}\label{interineq}
\mu \left( \left| \hat{S}_N^{k(N)}(f)-\E S_N(f_N) \right| >2C N^{\alpha} k(N)^{1-\alpha-\epsilon} \right) \ll k(N)^{-A},
\end{equation}
where the constant $C>1$ is the one specified in Lemma~\ref{interTSAR}.

Define the subsequence $\tilde{N}_l = \lceil 2^{l^{\epsilon}} \rceil$ and construct $N_{l,j}$ recursively by setting $N_{l,0} = \tilde{N}_l$ and $N_{l,j+1} = \min(\tilde{N}_{l+1},N^*),$ where
\begin{equation*}
N^* = \max\left(N_{l,j} + 1, \;\max\Bigl(N\geq N_{l,j}: \;\frac{k(N)}{k(N_{l,j})} < 2^{l^{-\epsilon}}\Bigr)\right).
\end{equation*}
This ensures that the sequence $N_{l,j}$ increases monotonically from $\tilde{N}_l$ to $\tilde{N}_{l+1},$ and that the growth of $k(N)$ remains controlled at each step. Specifically, if $N_{l,j+2}<\tilde{N}_{l+1}$ then
\begin{equation}\label{kNgrowth}
\frac{k(N_{l,j+2})}{k(N_{l,j})} \geq 2^{l^{-\epsilon}}.
\end{equation}
Furthermore, since $k(\tilde{N}_{l+1}) \leq \tilde{N}_{l+1} \leq 2^{(l+1)^{\epsilon}} + 1$ and the cumulative growth factor over $2 l^{3\epsilon}$ steps is at least $(2^{l^{-\epsilon}})^{l^{3\epsilon}} = 2^{l^{2\epsilon}},$ the sequence must reach $ \tilde{N}_{l+1}$ within at most $2 l^{3\epsilon}$ steps. Therefore,
\begin{equation*}
N_{l,2 l^{3\epsilon}} = \tilde{N}_{l+1}.
\end{equation*}

Let $\{N_i\}$ be the ordered sequence of distinct values from $N_{l,j}$. Then for sufficiently small $\epsilon>0$ and $A>2 \epsilon^{-2},$
\begin{equation}\label{kNsum}
\sum_{i=1}^{\infty} k(N_i)^{-A} \ll \sum_{l=1}^{\infty} l^{3\epsilon} (\log\tilde{N}_l)^{-A\epsilon} \ll \sum_{l=1}^{\infty} l^{-A \epsilon^2 + 3\epsilon} <\infty,
\end{equation}
because $\lim_{N\rightarrow \infty} \frac{k(N)}{(\log N)^{\epsilon}} = \infty$. By the first Borel--Cantelli lemma, and \eqref{interineq}, we have
\begin{equation}\label{kNE}
\lim_{i\rightarrow \infty} \frac{\hat{S}_{N_i}^{k(N_i)}(f)(x)}{\E S_{N_i}(f_{N_i})}=1 \quad \text{ for almost every } x.
\end{equation}
Since
\begin{equation*}
\E S_N(f_N) \sim c N^\alpha k(N)^{1-\alpha},
\end{equation*}
with $c=\frac{1}{\alpha-1} \operatorname{Res}_{x^*}(f)B_d^{\alpha} \rho(x^*)^{\alpha}$, it follows that
\begin{equation}\label{kNcon}
\lim_{i\rightarrow \infty} \frac{\hat{S}_{N_i}^{k(N_i)}(f)(x)}{N_i^{\alpha} k(N_i)^{1-\alpha}} = c \quad \text{ for almost every } x.
\end{equation}

Define two auxiliary trimming sequences as
\begin{align*}
k_+(N) &= \left\lceil 2^{l^{-\epsilon}} k(N) \right\rceil \quad \text{if}\quad \tilde{N}_l \leq N < \tilde{N}_{l+1} \quad \text{ and }\\
k_-(N) &= \left\lfloor 2^{-l^{-\epsilon}} k(N) \right\rfloor \quad \text{if}\quad \tilde{N}_l \leq N < \tilde{N}_{l+1}.
\end{align*}
Since $\lim_{N\rightarrow \infty} \frac{k_{\pm}(N)}{k(N)} = 1$, Lemmas \ref{interTSAR} and \ref{interESTF}, as well as \eqref{kNsum}, also hold for $k_\pm$ instead of $k$, thus the first Borel--Cantelli lemma implies
\begin{equation}\label{kNpmcon}
\lim_{i\rightarrow \infty} \frac{\hat{S}_{N_i}^{k_{\pm}(N_i)}(f)(x)}{N_i^{\alpha} k(N_i)^{1-\alpha}} = c \quad \text{ for almost every } x.
\end{equation}

For $N \notin \{N_i\}_{i=1}^\infty$ with $\tilde{N}_l<N_i<N<N_{i+1}<\tilde{N}_{l+1},$ the construction of $\{N_i\}$ and $k_\pm$ implies
\begin{equation}\label{kN+ineq}
2^{-l^{-\epsilon}} k_+(N_i)-1\leq k(N_i)\leq k(N)\leq k(N_{i+1})\leq 2^{l^{-\epsilon}} k(N_i) \leq k_+(N_i),
\end{equation}
and
\begin{equation}\label{kN-ineq}
k_-(N_{i+1}) \leq k(N) \leq 2^{l^{-\epsilon}} k_-(N_{i+1})+2^{l^{-\epsilon}}.
\end{equation}

Let $x$ satisfy \eqref{kNcon} and \eqref{kNpmcon}. By the monotonicity of trimmed sums and applying \eqref{kN+ineq}, \eqref{kN-ineq}, we obtain
$$
\hat{S}_{N_i}^{k_+(N_i)}(f) \leq \hat{S}_N^{k(N)}(f) \leq \hat{S}_{N_{i+1}}^{k_-(N_{i+1})}(f).$$
Rewriting this inequality as
\begin{align*}
\frac{N_i^{\alpha}k_+(N_i)^{1-\alpha}}{N^{\alpha}k(N)^{1-\alpha}}\cdot
\frac{\hat{S}_{N_i}^{k_+(N_i)}(f)}{N_i^{\alpha}k_+(N_i)^{1-\alpha}}&\leq \frac{\hat{S}_N^{k(N)}(f)}{N^{\alpha}k(N)^{1-\alpha}}\\
&\leq
\frac{N_{i+1}^{\alpha}k_-(N_{i+1})^{1-\alpha}}{N^{\alpha}k(N)^{1-\alpha}}\cdot
\frac{\hat{S}_{N_{i+1}}^{k_-(N_{i+1})}(f)}{N_{i+1}^{\alpha}k_-(N_{i+1})^{1-\alpha}},
\end{align*}
and observing that the scaling factors converge to $1$ as $N \to \infty$, it follows from the squeeze theorem that
\begin{equation*}
\lim_{N\rightarrow\infty, N \not \in \{N_i\}} \frac{\hat{S}_N^{k(N)}(f)(x)}{N^{\alpha} k(N)^{1-\alpha}}=c.
\end{equation*}
Combining this with \eqref{kNcon}, this shows that
\begin{equation*}
\lim_{N\rightarrow\infty} \frac{\hat{S}_N^{k(N)}(f)(x)}{N^{\alpha} k(N)^{1-\alpha}}=c.
\end{equation*}
Finally, Lemma \ref{strongsklem} implies that also
\begin{equation*}
\lim_{N\rightarrow\infty} \frac{S_N^{k(N)}(f)(x)}{N^{\alpha} k(N)^{1-\alpha}}=c\quad \text{ for almost every } x.
\end{equation*}
\end{proof}

\section{Returns to small balls}\label{PoiLgT}

The main tool for analyzing distributional limits in the lightly trimmed case is the asymptotic Poisson nature of return times to small balls. This section establishes this property.

The Poisson limit theorem for return times has been investigated extensively in the literature. One of the earliest results in a dynamical systems context is due to Pitskel \cite{PitPois}, who proved it for mixing Markov chains. Subsequently, the result has been established under various mixing assumptions by many authors; for a comprehensive overview, we refer the reader to the survey \cite{Haydn13}, especially Section~5.1 for a discussion of the method of moments. More recently, slow recurrence together with exponential mixing of all orders has been employed in \cite{DFL22}; see also \cite{auerp}.

Specifically, \cite[Proposition~20]{auerp} implies that for $\mu$-almost every $x^* \in M$ and any finite collection of times $0<t_1<\cdots<t_J$, the finite-dimensional distributions converge weakly as $r \to 0$:
\begin{equation}\label{classicalPLTF}
\Bigl(S_{\lceil t_1 / \mu(B_r(x^*)) \rceil} (\one_{B_r(x^*)}), \ldots, S_{\lceil t_J / \mu(B_r(x^*)) \rceil} (\one_{B_r(x^*)})\Bigr) \;\Rightarrow\; \left(P(t_1), \ldots, P(t_J)\right),
\end{equation}
where each $P(t)$ is a Poisson random variable with mean $t$, and the vector has independent increments.

Moreover, a minor modification of the argument shows that the convergence in \eqref{classicalPLTF} also holds for points $x^*$ that are slowly recurrent.
If $\rho(x^*)>0,$ the measure of $B_r(x^*)$ satisfies $\mu(B_r(x^*))\sim B_d \rho(x^*) r^d$, where $B_d$ is the volume of the $d$-dimensional unit ball. Then \eqref{classicalPLTF} becomes
\begin{equation*}
\Bigl(S_{\lceil \frac{t_1}{B_d \rho(x^*)r^d}\rceil} (\one_{B_r(x^*)}),\ldots, S_{\lceil\frac{t_J}{B_d \rho(x^*)r^d} \rceil} (\one_{B_r(x^*)})\Bigr) \Rightarrow \left(P(t_1),\ldots, P(t_J)\right).
\end{equation*}

For the arguments that follow, we require a variant of this result in which the number of iterations is fixed while the radius of the balls varies. Although this version can be derived by the same methods, we include a proof for completeness.

\begin{proposition}\label{PLTprop}
Suppose the system $(M,\mathcal{B}(M),\mu,T)$ is exponentially mixing of all orders. If $x^* \in M$ is slowly recurrent and $\rho(x^*) > 0$, then for any $J\in \mathbb{N}$ and time points $0<t_1<t_2<\cdots<t_J$, the finite-dimensional distributions converge weakly as $N \to \infty$:
\begin{equation}\label{PCFinDim}
\Bigl(S_N(\chi_{N,t_1}), S_N(\chi_{N,t_2}), \ldots, S_N(\chi_{N,t_J}) \Bigr) \Rightarrow \left(P(t_1), P(t_2), \ldots, P(t_J) \right),
\end{equation}
where
$r_N(t)=\Bigl( \frac{t}{B_d \rho(x^*) N} \Bigr)^{\frac{1}{d}},$ $\chi_{N,t}=\one_{B_{r_N(t)}(x^*)},$
and $(P(t))_{t\geq 0}$ denotes  a standard Poisson process.
\end{proposition}

\begin{proof}
By the method of moments, it suffices to show that for all $J\geq 1$, $0=t_0<t_1<\dots<t_J$, and $(m_1,\dots,m_J) \in \mathbb{N}^J$,
\begin{equation}\label{PoiFDmC}
\lim_{N \to \infty} \mathbb{E}\left[ \prod_{j=1}^J \binom{S_{N,j} - S_{N,j-1}}{m_j} \right] = \prod_{j=1}^J \frac{(t_j - t_{j-1})^{m_j}}{m_j!},
\end{equation}
where $S_{N,j} =S_N(\chi_{N,t_j})$.

For $N \geq \max_j m_j$, we express the product in the left-hand side of \eqref{PoiFDmC} as
\begin{equation}\label{slbinomprod}
\prod_{j=1}^J \binom{S_{N,j} - S_{N,j-1}}{m_j} = \sum_{\mathbf{k}} \prod_{j=1}^J \prod_{i=1}^{m_j} \xi_N^{j,i},
\end{equation}
where, with the convention $t_0=0$ and $\chi_{N,t_0}=0$,
\begin{equation*}
\xi_N^{j,i}=\left(\chi_{N,t_j}-\chi_{N,t_{j-1}}\right)\circ T^{k_i^j}=\one_{
B_{r_N(t_j)}(x^*)\setminus B_{r_N(t_{j-1})}(x^*)}\circ T^{k_i^j}.
\end{equation*}
The sum runs over all tuples $\mathbf{k} = (k_1^1, \dots, k_{m_1}^1, \dots, k_1^J, \dots, k_{m_J}^J)$ whose entries satisfy $0 \leq k_1^j < \dots < k_{m_j}^j \leq N-1$ for each $j=1,\dots,J$. Note that
$\xi_N^{j,i}\cdot\xi_N^{j',i'}=0$ whenever $j\not=j'$ and $k_i^j=k_{i'}^{j'},$ and therefore we can restrict the sum to tuples with
\begin{equation*}
k^j_i \neq k^{j'}_{i'} \;\text{ whenever }\; (j,i)\neq (j',i').
\end{equation*}
For a large constant $C,$ let
\begin{equation*}
\Delta_N=\Delta_{N,(t_j),(m_j)}:=\left\{\mathbf{k}=(k_{i,j})_{\substack{j=1,\ldots,J\\i=1,\ldots,m_j}} \;\left|\;
\begin{aligned}
&0\leq k^j_{1} < \dots < k^j_{m_j} \leq N-1\text{ for } j=1,\ldots,J\\
&k^j_i \neq k^{j'}_{i'} \;\text{ whenever }\; (j,i)\neq (j',i')
\end{aligned}
\right\}\right. \\
\end{equation*}
and
\begin{equation*}
\Delta'_N=\Delta'_{N,(t_j),(m_j)}:=\left\{\mathbf{k}\in \Delta_N \;\left|\;\min_{\substack{j=1,\ldots,J, \; i=1,\ldots, m_j\\ j'=1,\ldots,J, \; i'=1,\ldots, m_{j'},\; (j,i)\not=(j',i')}} |k^j_i-k^{j'}_{i'}|\leq C \log N\right\}\right. .
\end{equation*}

Note that $ \prod_{j=1}^J \prod_{i=1}^{m_j} \xi_N^{j,i} = 0$ if $k^j_i = k^{j'}_{i'}$ for some $(j,i)$ and $(j',i')$ with $j \neq j'$. We then decompose the sum in \eqref{slbinomprod} into two components
\begin{equation*}
\prod_{j=1}^J {{S_{N,j} - S_{N,j-1}} \choose m_j} = M_N+R_N,
\end{equation*}
where
\begin{equation}\label{DecSNj}
M_N=M_{N,(t_j),(m_j)}= \sum_{\mathbf{k} \in \Delta_N \setminus \Delta_N'} \prod_{j=1}^J \prod_{i=1}^{m_j} \xi_N^{j,i}, \;\;
R_N=R_{N,(t_j),(m_j)}=\sum_{\mathbf{k} \in \Delta_N'} \prod_{j=1}^J \prod_{i=1}^{m_j} \xi_N^{j,i}.
\end{equation}
We aim to prove that
\begin{equation}\label{mlrlcon}
\mathbb{E}M_N \rightarrow \prod_{j=1}^{J} \frac{(t_j-t_{j-1})^{m_j}}{m_j!} \;\;\text{ and }\;\; \mathbb{E} R_N \rightarrow 0 \;\;\;\as{N}.
\end{equation}
Suppose $\mathcal{M}=\sum_{j=1}^J m_j.$ Let us first establish the asymptotic behavior of $M_N$.
For $\mathbf{k} \in \Delta_N \setminus \Delta_N'$, the indices are well-separated. Analogously to Lemma \ref{MixEmB}, exponential mixing of all orders yields
\begin{align*}
\left|\mu\biggl(\bigcap_{j=1}^J \bigcap_{i=1}^{m_j} T^{-k_i^j} \left(B_{r_N(t_j)}(x^*) \setminus B_{r_N(t_{j-1})}(x^*)\right)\biggr)\right.&\left.-\prod_{j=1}^J \mu \left(B_{r_N(t_j)}(x^*) \setminus B_{r_N(t_{j-1})}(x^*)\right)^{m_j}\right|\\
&=o(N^{-M}).
\end{align*}
Summing over $\mathbf{k}\in \Delta_N\setminus \Delta'_N$, and noting that
\begin{equation*}
\mu \left(B_{r_N(t_j)}(x^*) \setminus B_{r_N(t_{j-1})}(x^*)\right) \sim \frac{t_j-t_{j-1}}{N},
\end{equation*}
and
$$\#\left(\Delta_N\setminus \Delta'_N\right) \sim \prod_{j=1}^J {N \choose {m_j}},$$
we obtain
\begin{equation}\label{EMN}
\mathbb{E}M_N \sim N^{-\sum_{j=1}^J m_j} \sum_{\mathbf{k}\in \Delta_N\setminus \Delta'_N} \prod_{j=1}^J (t_j-t_{j-1})^{m_j} \sim \prod_{j=1}^{J} \frac{(t_j-t_{j-1})^{m_j}}{m_j!}.
\end{equation}

Now we turn our attention to $R_N$. Let $A_N = B_{r_N(t_J)}(x^*)$. Note that $R_N$ is supported on the set
$$
U_N = \bigcup_{n=0}^{N-1} T^{-n} \left( A_N \cap \{ \varphi_{A_N} \leq C \log N\} \right),
$$
where $\varphi_{A_N}(x) = \min\{ k \geq 1 : T^k(x) \in A_N \}$ is the first return time. By the slow recurrence condition at $x^*$,
$$
\mu(U_N) = o\left( (\log N)^{-2\mathcal{M}} \right).
$$
Applying the Cauchy--Schwarz inequality,
$$
\E R_N\leq \mu(\supp (R_N))^{1/2} \| R_N \|_{L^2} \leq \mu(U_N)^{1/2} \| R_N \|_{L^2}.
$$
Thus, to prove $\E R_N \rightarrow 0$, it suffices to show $\|R_N\|_{L^2} \ll (\log N)^{\mathcal{M}}$.
We bound $\|R_N\|_{L^2}$ by noting that
$$R_N^2 \leq\sum_{0\leq k_1,\ldots,k_{2\mathcal{M}} \leq N-1} \prod_{i=1}^{2\mathcal{M}} \one_{A_N} \circ T^{k_i}= S_{N,J}^{2\mathcal{M}},$$
where $S_{N,J} = S_N(\one_{A_N})$ and
\begin{equation*}
S_{N,J}^{2\mathcal{M}}= \sum_{m=1}^{2\mathcal{M}}m! \stirling{2\mathcal{M}}{m} {S_{N,J} \choose m} \leq C_{\mathcal{M}} \sum_{m=1}^{2\mathcal{M}} {S_{N,J} \choose m},
\end{equation*}
where $C_{\mathcal{M}}=\max_{1\le q\le2M} q!\stirling{2\mathcal{M}}{q}$ and $\stirling{2\mathcal{M}}{q}$ are the Stirling numbers of the second kind.

Fix $m \geq 1,$ we adapt the previous notation for $t_J,m$ instead of $(t_j),(m_j).$ Then $\Delta_N$ denotes the set of $m$-tuples $0 \leq k_1 < \dots < k_m \leq N-1,$ and $\Delta_N'$ is the subset of $\Delta_N$ containing tuples where the minimum distance between any two distinct indices is at most $C \log N.$ Analogously to \eqref{DecSNj}, we decompose $\binom{S_{N,J}}{m} = M_{N,m} + R_{N,m}$, where $M_{N,m}$ and $R_{N,m}$ are sums of products of $\one_{A_N} \circ T^{k_i}$ over index tuples $\mathbf{k} \in \Delta_N \setminus \Delta_N'$ and $\mathbf{k} \in \Delta_N'$, respectively.

As in \eqref{EMN}, $\mathbb{E}M_{N,m}$ is bounded for each fixed $m$. It remains to show that
$$\mathbb{E}R_{N,m}\ll (\log N)^{2\mathcal{M}}$$
for all $m = 1, \ldots, 2\mathcal{M}$.

We partition $\Delta_N$ by defining
$$s(\mathbf{k}) = \max\left\{ s \geq 1 \;\middle|\; \exists\, 1 \leq i_1 < \dots < i_s \leq m\text{ such that } \min_{r=1,\dots,s-1} |k_{i_{r+1}} - k_{i_r}| > C \log N \right\}.$$
Then $\Delta_N = \bigcup_{s=1}^m \Delta_N^s$, where $\Delta_N^s = \{ \mathbf{k} : s(\mathbf{k}) = s \}$, and $\Delta_N' = \bigcup_{s=1}^{m-1} \Delta_N^s$.

For $s = 1, \dots, m-1$ and $0\leq l_1 < \dots < l_s \leq N-1$ with $\min_r |l_{r+1}-l_r| > C\log N$, define
$$
A_{l_1,\dots,l_s} = \bigcap_{r=1}^s T^{-l_r} A_N.
$$
By exponential mixing of all orders, $\mu(A_{l_1,\dots,l_s})\ll
\mu(A_N)^s+o(N^{-s}) \ll N^{-s}$. For $x\in A_{l_1,\dots,l_s},$ consider
\begin{equation*}
\mathfrak{K}_{l_1,\dots,l_s}(x)=\left\{\mathbf{k}=(k_1,\dots,k_m)\in \Delta_N^s\;\left|
\begin{aligned}
&\exists\, i_1,\dots,i_s\; \text{ such that } k_{i_r}=l_r, \; \forall r=1,\dots,s \\
&\text{and }\; x\in \bigcap_{i=1}^m T^{-k_i} A_N
\end{aligned} \right\} \right. .
\end{equation*}
Then $|\mathfrak{K}_{l_1,\dots,l_s}(x)| \ll (\log N)^m$, since the remaining indices must lie within $C \log N$ of some $l_r$. Furthermore, for any $\mathbf{k} \in \mathfrak{K}_{l_1,\dots,l_s}(x),$ the support of the corresponding product satisfies
$$
\supp\left(\prod_{i=1}^{m} \one_{A_N}\circ T^{k_i}\right) \subset A_{l_1,\dots,l_s}.
$$
Thus,
\begin{align*}
\sum_{\mathbf{k} \in \Delta_N^s} \E \left[ \prod_{i=1}^m \one_{A_N} \circ T^{k_i} \right] & \leq
\sum_{\substack{0\leq l_1<\ldots<l_s\leq N-1\\ \min_{r=1,\ldots,s-1} |l_{r+1}-l_r|>C\log N}}
\int_{A_{l_1,\ldots,l_s}} |\mathfrak{K}_{l_1,\ldots,l_s}(x)| \d\mu(x)\\
& \ll (\log N)^m \sum_{\substack{0\leq l_1<\ldots<l_s\leq N-1\\ \min_{r=1,\ldots,s-1} |l_{r+1}-l_r|>C\log N}} \mu(A_{l_1,\ldots,l_s})\\
&\ll (\log N)^m.
\end{align*}
Summing over $s = 1, \dots, m-1$ yields $\E R_{N,m} \ll (\log N)^m \leq (\log N)^{2\mathcal{M}},$ thereby completing the proof.
\end{proof}

To prepare for a trimmed version, we introduce an alternative representation of the limit in \eqref{PCFinDim} using Poisson point processes.

\begin{definition}
A random countable set $\Lambda \subset [0, \infty)$ is called a (standard) Poisson point process (PPP) if
\begin{itemize}
\item For any interval $[t_1,t_2) \subset [0, \infty)$, the number of points in $\Lambda \cap[t_1,t_2)$ follows a Poisson distribution with mean $(t_2-t_1)$. That is, for $k\in\mathbb{N},$
\begin{equation*}
\P\Bigl(\# \bigl(\Lambda \cap [t_1,t_2)\bigr) =k\Bigr) = \frac{(t_2-t_1)^k e^{-(t_2-t_1)}}{k!}
\end{equation*}
\item For any collection of disjoint intervals $[s_1, t_1),$ $[s_2, t_2),$ $\dots,$ $[s_n, t_n)\subset [0, \infty)$, the counts
  $$
 \#\left(\Lambda \cap [s_1, t_1)\right), \#\left(\Lambda \cap [s_2, t_2)\right), \dots, \#\left(\Lambda \cap [s_n, t_n)\right)$$
  are independent.
\end{itemize}
\end{definition}

Now the standard Poisson process $\mathcal{P}$ is given by
\begin{equation*}
\mathcal{P}=(P_t)_{t\geq 0} =\left(\Lambda \cap [0,t)\right)_{t\geq 0}.
\end{equation*}
For $K\geq 1$, the trimmed Poisson point process $\Lambda^K$ is defined as $\Lambda$ with the $K$ points closest to the origin removed, that is,
$$
\Lambda^K = \Lambda \setminus \argmin_{\{x_1,\dots,x_K\}\subset \Lambda,\,\,x_1<x_2<\dots<x_K} \sum_{j=1}^K x_j,
$$
with the convention $\Lambda^0=\Lambda$.

\begin{remark}\label{PPiid}
Given a PPP $\Lambda$, there are i.i.d.\ standard exponentially distributed random variables $E_1, E_2, \dots$ such that
\begin{equation*}
\Lambda=\{\xi_k\}_{k\geq 1},
\end{equation*}
where $\xi_k = \sum_{j=1}^k E_j$ denotes the $k$-th arrival time. Then the corresponding trimmed PPP is $\Lambda^K=\{\xi_k\}_{k\geq K+1}$.
\end{remark}

\begin{remark}
The trimmed Poisson point process defined above removes the $K$ points closest to the origin. This differs from the notion of trimmed point processes used elsewhere, for instance in \cite{BIM17}, where trimming refers to discarding the largest jumps of a L\'{e}vy process. The two definitions serve different purposes and are not equivalent.
\end{remark}

The trimmed version of Proposition \ref{PLTprop} is the following.

\begin{proposition}\label{trimPLTpropF}
Let $K\geq 0$. Under the same assumptions as in Proposition \ref{PLTprop}, for any $0<t_1<\cdots<t_J$,
$$
\Bigl(\hat{S}_N^K(\chi_{N,t_1}), \ldots, \hat{S}_N^K(\chi_{N,t_J})\Bigr)
\Rightarrow \left(\#\bigl(\Lambda^K \cap [0,t_1)\bigr), \ldots, \#\bigl(\Lambda^K \cap [0,t_J)\bigr)\right),
$$
where $\Lambda^K$ is a trimmed Poisson point process, and $r_N(t)$, $\chi_{N,t}$ are defined as in Proposition \ref{PLTprop}.
\end{proposition}

\begin{proof}
For $K=0$, the statement reduces to Proposition~\ref{PLTprop}. For $K\ge1$, the identities
$$
\hat{S}_N^K(\chi_{N,t}) = \max\Bigl(S_N(\chi_{N,t}) - K, 0\Bigr)
$$
and
$$
\#\bigl(\Lambda^K \cap [0,t)\bigr) = \max\Bigl(\#\bigl(\Lambda \cap [0,t)\bigr) - K, 0\Bigr),
$$
show that the conclusion follows immediately from Proposition \ref{PLTprop}.
\end{proof}

\section{Summation over Poisson point processes}\label{SumPP}

We investigate the asymptotic behavior of $\sum_{x\in \Lambda^K \cap [0,R)} \frac{1}{x^{\alpha}}$, for $\alpha> \frac{1}{2}$, as $R\rightarrow \infty,$ where $\Lambda^K$ denotes a trimmed Poisson point process.

The case $\alpha > 1$ can be addressed directly.
Using the tail estimate $\mathbb{P}(\#(\Lambda \cap [n,n+1]) > n^{\frac{\alpha-1}{2}}) \ll e^{-n^{\frac{\alpha-1}{2}}}$ for $n\geq1$ and the first Borel–-Cantelli lemma, we deduce that almost surely, for all sufficiently large $n$,
$
\#(\Lambda \cap [n,n+1]) \ll n^{\frac{\alpha-1}{2}}.
$
It follows that
$$\sum_{x\in \Lambda^K\cap[1,\infty)}\frac{1}{x^{\alpha}}\ll
\sum_{n\ge1} n^{-\alpha}\#(\Lambda\cap[n,n+1])\ll\sum_n n^{-\frac{\alpha+1}{2}}<\infty.$$
Moreover, since $\#(\Lambda\cap[0,1))$ is almost surely finite and contains no point at $0$,
$$
\sum_{x\in \Lambda^K\cap[0,1)}\frac1{x^\alpha}<\infty
\quad\text{ almost surely}.
$$
Hence
\begin{equation}\label{PPPcon>1}
\sum_{x\in \Lambda^K \cap [0,R)} \frac{1}{x^{\alpha}} \rightarrow \sum_{x\in \Lambda^K} \frac{1}{x^{\alpha}} \quad \as{R}, \; \text{ almost surely}.
\end{equation}
It remains to study the case $\alpha\in (\frac{1}{2},1]$.

\begin{proposition}\label{PPPconprop}
Let $K \geq 0$ and $\alpha>\frac{1}{2}$. Then there is a
non-degenerate random variable $Y$ such that
$$\sum_{x\in\Lambda^K\cap[0,R)}\frac1{x^\alpha}-c_R\Rightarrow Y
\qquad\text{as }R\to\infty,$$
where
$$
c_R=\begin{cases}
\displaystyle \frac{R^{1-\alpha}-1}{1-\alpha},
& \frac12<\alpha<1,\\[1.2ex]
\log R,
& \alpha=1,\\[0.8ex]
0,
& \alpha>1.
\end{cases}
$$
Moreover, the $p$-th absolute moment of $Y$ for any positive integer $p$ satisfies
\begin{equation*}
\E(|Y|^p)
\begin{cases}
<\infty & \text{if } p<\frac{K+1}{\alpha},\\
=\infty & \text{if } p\geq \frac{K+1}{\alpha}.
\end{cases}
\end{equation*}
\end{proposition}

\begin{proof}
For $\alpha>1$, we have already shown in \eqref{PPPcon>1} that
\begin{equation*}
\sum_{x\in \Lambda^K \cap [0,R)} \frac{1}{x^{\alpha}}
\rightarrow
\sum_{x\in \Lambda^K} \frac{1}{x^{\alpha}}
\quad \as{R},
\; \text{almost surely}.
\end{equation*}
Therefore, the desired convergence in distribution holds with $c_R=0$. The finiteness of moments, however, depends on the behavior near the origin, which will be treated together with the case $\alpha\in(1/2,1]$. Thus, in the main part of the proof, we focus on constructing the centering constants and establishing convergence for $\alpha\in(1/2,1]$.

(i) To establish weak convergence, we first
prove convergence for a suitable sequence of centering constants $c_R>0$ and a distribution $\tilde{Y}$, without yet claiming non-degeneracy, such that
\begin{equation}\label{poissigmacon}
\Sigma_R - c_R \rightarrow \tilde{Y}\quad \as{R}, \text{ in } L^p \text{ for all } 1\leq p < \infty,
\end{equation}
where $\Sigma_R=\sum_{x\in \Lambda^K \cap [1,R)} \frac{1}{x^{\alpha}}$.
In Step (iv), we will show that $c_R$ may be replaced by the explicit centering constants stated in the proposition.

Once \eqref{poissigmacon} has been established, the conclusion of the proposition follows with
\begin{equation*}
Y= \tilde{Y}+\sum_{x\in \Lambda^K \cap [0,1)} \frac{1}{x^{\alpha}}.
\end{equation*}
We separate the contribution from $[0,1)$ because it is independent of $R$. Thus it does not affect the convergence as $R\to\infty$, although it remains part of the limiting random variable and determines the range of $p$ for which $Y\in L^p$.

(ii) For the moment, take $c_R=\E \Sigma_R$. For every even $k\geq 2$ we will show
\begin{equation}\label{PPPtrimlemvar1}
\E\left( \left|\Sigma_{R,R'} - c_{R,R'}\right|^k \right) \rightarrow 0 \;\;\;\as{R>R'},
\end{equation}
where $ \Sigma_{R,R'} =\Sigma_R - \Sigma_{R'}$ and $c_{R,R'}=c_R-c_{R'}$. This convergence will imply that the sequence $\{\Sigma_R-c_R\}$ is Cauchy in $L^k$, and hence converges in $L^k$ to some $\tilde{Y}$. Convergence in $L^p$ for every $p\geq 1$  then follows immediately from  H\"{o}lder's inequality. Once convergence has been established, the limit $\tilde{Y}$ does not depend on $p$ and \eqref{poissigmacon} follows.

Let $\Sigma^*_{R,R'} = \sum_{x\in \Lambda \cap [R',R)} x^{-\alpha}$, for $R>R'>1$ be the untrimmed sum, and $c^*_{R,R'}=\E \Sigma^*_{R,R'}$. If $K=0$, then $\Sigma_{R,R'}=\Sigma^*_{R,R'}.$ If $K\ge1,$ we observe that
\begin{align*}
\P\left(\Sigma_{R,R'} \not = \Sigma^*_{R,R'}\right) \leq \P\left( \#\{\Lambda\cap [0,R')\} <K\right) \ll (R')^K e^{-R'},
\end{align*}
and consequently,
\begin{align*}
&\left|\left\|\Sigma_{R,R'}-c_{R,R'}\right\|_{L^k}-\left\|\Sigma^*_{R,R'}-c^*_{R,R'}\right\|_{L^k}
\right|\le\left\|\bigl(\Sigma_{R,R'}-c_{R,R'}\bigr)-\bigl(\Sigma^*_{R,R'}-c^*_{R,R'}\bigr)
\right\|_{L^k}\\
&\le 2\left\|\Sigma_{R,R'}-\Sigma^*_{R,R'}\right\|_{L^k}\leq 2\sup |\Sigma_{R,R'}-\Sigma^*_{R,R'}|\cdot \P\left(\Sigma_{R,R'} \not=\Sigma^*_{R,R'}\right)^{\frac{1}{k}}\\
&\ll K(R')^{-\alpha+\frac{K}{k}}e^{-\frac{R'}{k}}.
\end{align*}
Therefore, it suffices to prove \eqref{PPPtrimlemvar1} with $\Sigma^*_{R,R'}$ instead of $\Sigma_{R,R'}$.

Furthermore, it is enough to verify \eqref{PPPtrimlemvar1} for integer values $R=n$, $R'=n'$, as the general case follows by standard moment comparison arguments.

(iii)
Let $N= |\Lambda \cap [n',n)|$ denote the number of points in $[n',n)$, where $N\sim \text{Poisson}(n-n')$. Given $N=k$, the ordered arrival times $\xi_1, \xi_2, \dots, \xi_k$ have the same distribution as the order statistics of $k$ independent random variables uniformly distributed on $[n',n)$. Hence $\Sigma_{n,n'}^*=\sum_{i=1}^k \xi_i^{-\alpha}$ and
\begin{equation*}
c^*_{n,n'}=\E \Sigma^*_{n,n'}=\E \left[\E \left[\Sigma^*_{n,n'}\mid N\right]\right]=\int_{n'}^{n} x^{-\alpha} \d x=
\begin{cases}
\dfrac{n^{1-\alpha} - (n')^{1-\alpha}}{1 - \alpha}, & \alpha \neq 1;\\
\log n-\log n', & \alpha = 1.
\end{cases}
\end{equation*}

The moment generating function of $\bar{Y}=\Sigma^*_{n,n'}-c^*_{n,n'}$ is defined as
\begin{equation*}
M_{\bar{Y}}(t) = \mathbb{E}\Bigl[e^{t\bigl(\Sigma^*_{n,n'}-c^*_{n,n'}\bigr)}\Bigr]=e^{-t c^*_{n,n'}}\mathbb{E}\Bigl[e^{t \Sigma^*_{n,n'}}\Bigr].
\end{equation*}
We first compute $\mathbb{E}[e^{t \Sigma^*_{n,n'}}]$ using the law of total expectation
\begin{equation*}
\mathbb{E}\Bigl[e^{t \Sigma^*_{n,n'}}\Bigr] = \mathbb{E}\left[\mathbb{E}\Bigl[e^{t \Sigma^*_{n,n'}}\Big| N\Bigr]\right].
\end{equation*}
As above we obtain
\begin{equation*}
\mathbb{E}\left[e^{t \Sigma^*_{n,n'}} \Big| N=k\right] = \mathbb{E}\biggl[\prod_{i=1}^k e^{t \xi_i^{-\alpha}}\biggr] =
\left(\mathbb{E}e^{t \xi_1^{-\alpha}} \right)^k.
\end{equation*}
Since $\xi_1\sim \text{Uniform}[n', n)$, we have
\begin{equation*}
M(t):=\mathbb{E}[e^{t \xi_1^{-\alpha}}]=\frac{1}{n-n'} \int_{n'}^{n} e^{t u^{-\alpha}} \d u.
\end{equation*}
Thus,
\begin{equation*}
\mathbb{E}\left[e^{t \Sigma_{n,n'}^*}\right]=\mathbb{E}\left[M(t)^N\right].
\end{equation*}
Since $N\sim \text{Poisson}(n-n')$, its probability generating function is $\mathbb{E}[z^N] = e^{(n-n')(z-1)}$. Substituting $z= M(t)$, we obtain
\begin{equation*}
\mathbb{E}\left[e^{t \Sigma^*_{n,n'}}\right]=e^{(n-n')(M(t)-1)}.
\end{equation*}
Therefore,
\begin{equation*}
M_{\bar{Y}}(t)= e^{-t c^*_{n,n'}} e^{(n-n')(M(t)-1)}.
\end{equation*}
Substituting $c^*_{n,n'}=\int_{n'}^{n} x^{-\alpha} \d x,$ we get
\begin{equation*}
M_{\bar{Y}}(t)=\exp\left(\int_{n'}^{n} \left( e^{t u^{-\alpha}}-t u^{-\alpha}-1\right)\d u\right).
\end{equation*}
It follows that the cumulant generating function is given by
\begin{equation*}
K_{\bar{Y}}(t) =\log M_{\bar{Y}}(t)= \int_{n'}^{n} \left( e^{t u^{-\alpha}}-t u^{-\alpha}-1 \right)\d u.
\end{equation*}
The $k$-th cumulant is $\kappa_k = K_{\bar{Y}}^{(k)}(0)$. Noting that $K_{\bar{Y}}(0) = 0$, we further observe that $\mathbb{E}\bar{Y}= K_{\bar{Y}}'(0) = 0$, which is consistent with the definition of centering. For $k \geq 2$, define $g(t, u) = e^{t u^{-\alpha}}-t u^{-\alpha}-1$. Then
\begin{equation*}
\frac{\partial^k}{\partial t^k} g(t, u) = (u^{-\alpha})^k e^{t u^{-\alpha}}.
\end{equation*}
Therefore
\begin{equation*}
K_{\bar{Y}}^{(k)}(t) = \int_{n'}^{n} u^{-k\alpha} e^{t u^{-\alpha}} \d u, \quad \text{and} \quad K_{\bar{Y}}^{(k)}(0) = \int_{n'}^{n} u^{-k\alpha} \d u.
\end{equation*}
This yields
\begin{equation*}
\kappa_k = \begin{cases}
0 & \text{if } k = 1, \\
\frac{n^{1-k\alpha} - (n')^{1-k\alpha}}{1 - k\alpha} &\text{if } k \geq 2,\, k\alpha \neq 1, \\
\log n -\log n', &\text{if } k \geq 2,\, k\alpha = 1.
\end{cases}
\end{equation*}
Higher-order moments of $\left(\Sigma^*_{n,n'}-c^*_{n,n'}\right)$ can be expressed in terms of cumulants $\kappa_n$
as follows:
\begin{equation*}
\E\left[\left(\Sigma^*_{n,n'}-c^*_{n,n'}\right)^k\right]= \sum_{\ell=1}^k B_{k,\ell}(\kappa_1, \kappa_2, \dots, \kappa_{k-\ell+1}),
\end{equation*}
where $B_{k,\ell}(x_1, \dots, x_{k-\ell+1})$ denotes the partial Bell polynomial
\begin{equation*}
B_{k,\ell}(x_1,\ldots,x_{k-\ell+1})
=\sum_{\substack{
r_1+r_2+\cdots+r_{k-\ell+1}=\ell\\
r_1+2r_2+\cdots+(k-\ell+1)r_{k-\ell+1}=k
}}
\frac{k!}{r_1!\,r_2!\cdots r_{k-\ell+1}!}
\prod_{j=1}^{k-\ell+1}\left(\frac{x_j}{j!}\right)^{r_j}.
\end{equation*}
Given that $ \kappa_1=0 $, all terms with $ r_1 \geq 1 $ vanish, and the summation is therefore restricted to the case $ r_1 = 0 $. This leads to
\begin{equation*}
\E\left[\left(\Sigma^*_{n,n'}-c^*_{n,n'}\right)^k\right]= \sum_{\substack{r_2, r_3, \dots \geq 0 \\ 2r_2 + 3r_3 + \cdots = k}} \frac{k!}{r_2!\,r_3! \cdots} \prod_{j \geq 2} \left( \frac{\kappa_j}{j!} \right)^{r_j}.
\end{equation*}
If $k$ is even, the dominant term corresponds to the case where $r_2=\frac{k}{2}$ and $ r_j = 0 $ for all $ j \neq 2 $, which is expressed as
\begin{equation*}
\frac{k!}{(\frac{k}{2})!\,2^{\frac{k}{2}}} \kappa_2^{\frac{k}{2}}\ll \left((n')^{1-2\alpha}-n^{1-2\alpha}\right)^{\frac{k}{2}}\to 0
\end{equation*}
as $n>n'\to\infty$ since $\alpha>\frac{1}{2}.$

(iv) We now identify the centering constants explicitly. In the case  $\frac12<\alpha\le1$, recall that the centering term used above is $c_R=\mathbb E\Sigma_R$. Since $\Lambda^K$ is obtained from $\Lambda$ by removing its first $K$ points $\xi_1,\ldots,\xi_K$, we have
$$\mathbb E\Sigma_R^*-c_R=\mathbb E\biggl(\sum_{j=1}^K\xi_j^{-\alpha}\mathbf 1_{{1\le\xi_j<R}}\biggr)
\to\mathbb E\biggl(\sum_{j=1}^K\xi_j^{-\alpha}
\mathbf 1_{{\xi_j\ge1}}\biggr)<\infty,
$$
where, similarly to the above, $\Sigma_R^*=\sum_{x\in \Lambda \cap[1,R)} x^{-\alpha}$ denotes the untrimmed sum.
Thus, $c_R$ is, up to an asymptotically constant additive correction that can be absorbed into the limiting random variable, equal to the expectation of the untrimmed sum $\Sigma_R^*$. Choosing this normalization, we may take
$$
c_R=\mathbb E\Sigma_R^*
=\int_1^R x^{-\alpha}\d x
=
\begin{cases}
\displaystyle\frac{R^{1-\alpha}-1}{1-\alpha},
& \frac12<\alpha<1,\\[1.2ex]
\log R,
& \alpha=1.
\end{cases}$$

(v)
To determine the moments of the limiting random variable, recall from the
preceding construction that
\begin{equation*}
Y=\tilde Y+\sum_{x\in\Lambda^K\cap[0,1)}x^{-\alpha}.
\end{equation*}
Since $\tilde Y$ possesses finite moments of all orders, adding $\tilde Y$ does not affect the finiteness of the $p$-th moment. Hence the finiteness or divergence of $\E(|Y|^p)$ is determined by the near-origin term
\begin{equation*}
\sum_{x\in\Lambda^K\cap[0,1)}x^{-\alpha}.
\end{equation*}
Set $p_0 = \frac{K+1}{\alpha}$. It remains to show that this sum lies in $L^p$ for $p<p_0$ but not for $p=p_0$.

Condition on $\xi_{K+1}\in[\frac1j,\frac1{j-1})$. By the law of total expectation,
\begin{equation}\label{poissumfin}
\begin{aligned}
&\mathbb{E}\Bigl[\Bigl(\sum_{x\in\Lambda^K\cap[0,1)} x^{-\alpha}\Bigr)^p\Bigr]\\
= &\sum_{j=2}^{\infty}\mathbb{E}\Bigl[\Bigl(\sum_{x\in\Lambda^K\cap[0,1)} x^{-\alpha}\Bigr)^p
\Big|\xi_{K+1}\in\bigl[\tfrac1j,\tfrac1{j-1}\bigr)\Bigr]
\cdot \mathbb{P}\bigl(\xi_{K+1}\in\bigl[\tfrac1j,\tfrac1{j-1}\bigr)\bigr).
\end{aligned}
\end{equation}
On this event, $x\ge 1/j$ for all $x\in\Lambda^K\cap[0,1)$, so $\sum x^{-\alpha}\le N j^\alpha$ with $N=\#(\Lambda^K\cap[0,1))$. Conditioned on $\xi_{K+1}$, the points in $(\xi_{K+1},1]$ form a Poisson process of mean $1-\xi_{K+1}$, so $N=1+R$ where $R$ is Poisson with mean $1-\xi_{K+1}$. Its $p$-th moment is bounded uniformly over $\xi_{K+1}\in[\tfrac1j,\tfrac1{j-1})$, yielding
\begin{equation}\label{CExi}
\mathbb{E}\Bigl[\Bigl(\sum_{x\in\Lambda^K\cap[0,1)}x^{-\alpha}\Bigr)^p \Big|\,\xi_{K+1}\in[\tfrac1j,\tfrac1{j-1})\Bigr]
\ll j^{\alpha p}.
\end{equation}
For the probability factor, $\xi_{K+1}\in[\frac1j,\frac1{j-1})$ means that there are at least $K+1$ points in $\bigl[0,\frac{1}{j-1}\bigr)$ but at most $K$ points in $\bigl[0,\frac{1}{j}\bigr)$. Given $\#(\Lambda\cap[0,\frac{1}{j}))=i$, one needs at least $K-i+1$ points in $[\frac{1}{j},\frac{1}{j-1})$, with corresponding probabilities
\begin{align*}
\P\left(\#\left(\Lambda\cap\left[0,\frac{1}{j}\right)\right)=i\right)&\asymp j^{-i},\\
\P\left(\#\left(\Lambda\cap\left[\frac{1}{j},\frac{1}{j-1}\right)\right)\geq K-i+1\right) &\asymp j^{-2(K-i+1)}.
\end{align*}
Summing over $i=0,\dots,K$ gives
\begin{equation}\label{PRxi}
\mathbb{P}\biggl(\xi_{K+1}\in\Bigl[\frac{1}{j},\frac{1}{j-1}\Bigr)\biggr)\asymp j^{-K-2}.
\end{equation}
Combined with \eqref{CExi}, the $j$-th term of the sum in \eqref{poissumfin} is $O(j^{\alpha p-K-2})$, so the series converges precisely when $p<p_0$.

For the lower bound of the $p_0$-th moment: if $\xi_{K+1}\in[\frac1j,\frac1{j-1})$, then we trivially estimate $\sum x^{-\alpha} \ge \xi_{K+1}^{-\alpha} \gg j^{\alpha}$, hence
\begin{equation*}
\mathbb{E}\Bigl[\Bigl(\sum_{x\in\Lambda^K\cap[0,1)}x^{-\alpha}\Bigr)^{p_0} \Big|\,\xi_{K+1}\in[\tfrac1j,\tfrac1{j-1})\Bigr]
\gg j^{\alpha p_0}.
\end{equation*}
Combining this estimate with \eqref{poissumfin} and \eqref{PRxi} yields
$$
\mathbb{E}\Bigl[\Bigl(\sum_{x\in\Lambda^K\cap[0,1)}x^{-\alpha}\Bigr)^{p_0}\Bigr]
\gg\sum_{j=2}^\infty j^{\alpha p_0-K-2}=\infty,
$$
which shows the sum fails to be in $L^{p_0}$.

\end{proof}

\section{Lightly trimmed distributional limits}\label{LgTrDL}

In this section, we assume that the system $(M, \mathcal{B}(M), \mu, T)$ is exponentially mixing of all orders, and that $\Ord_{x^*}(f) = \beta >\frac{d}{2}$. We will show a lightly trimmed distributional law, which establishes Theorem \ref{lighttrimthm}. As noted in the previous sections, this is again simpler for $\hat{S}_N^K$, where we simply remove the $K$ closest points to $x^*$. However, unlike our proof for the SLLN, the argument in Lemma \ref{strongsklem} is not quite enough to yield a limit theorem for $S_N^K$, and a quantitative estimate is required. Recall that by definition $\hat{S}_N^K(f)\geq S_N^K(f)$.

\begin{lemma}\label{nearlightlem}
For any fixed integer $K\geq 0$,
\begin{equation*}
\lim_{N\rightarrow \infty} \mu\left( \hat{S}_N^K(f)-S_N^K(f)>\tilde\epsilon N^{\alpha}\right)=0 \quad \forall \tilde\epsilon>0.
\end{equation*}
\end{lemma}
\begin{proof}
For $K=0$ we have $\hat{S}_N^0(f)=S_N^0(f)=S_N(f)$, and the claim is trivially satisfied. For the rest of the proof, we focus on the case $K\geq 1$.

With the notation from Section~\ref{nearsec}, each $y_i$ satisfies $d(x^*,y_i)>d(x^*,\hat{x}_K)$, since otherwise $y_i$ would be included among the $K$ closest points. Consequently,
\begin{align*}
f(y_i)&\leq \biggl(1+\epsilon\Bigl(d(x^*,y_i)\Bigr)\biggr)\Res_{x^*}(f) d(x^*,y_i)^{-\beta}\\
&<\biggl(1+\epsilon\Bigl(d(x^*,x_K)\Bigr)\biggr) \Res_{x^*}(f) d(x^*,\hat{x}_K)^{-\beta}.
\end{align*}
Similarly, for each $\hat{y}_i$, we have $d(x^*, \hat{y}_i) \leq d(x^*, \hat{x}_K)$, and thus
\begin{equation*}
f(\hat{y}_i) \geq \biggl(1-\epsilon\Bigl(d(x^*, \hat{x}_K)\Bigr)\biggr)\Res_{x^*}(f) d(x^*, \hat{x}_K)^{-\beta}.
\end{equation*}
The difference between the trimmed sums is
\begin{equation*}
\hat{S}_N^K(f) - S_N^K(f) = \sum_{i=1}^m \Bigl(f(y_i)-f(\hat{y}_i)\Bigr).
\end{equation*}
Substituting the bounds above,
\begin{equation}\label{SNDf}
\hat{S}_N^K(f)-S_N^K(f)\leq K \Res_{x^*}(f)\biggl(\epsilon\Bigl(d(x^*, x_K)\Bigr)+\epsilon\Bigl(d(x^*, \hat{x}_K)\Bigr)\biggr)d(x^*, \hat{x}_K)^{-\beta}.
\end{equation}
Fix $\eta > 0$ and define the set
$$G_N(\eta) = \left\{x : d(x^*,\hat{x}_K) \geq \eta N^{-1/d},\ \max\biggl(\epsilon\Bigl(d(x^*,x_K)\Bigr), \epsilon\Bigl(d(x^*,\hat{x}_K)\Bigr)\biggr) \leq \eta^{2\beta}\right\}.$$
On $G_N(\eta)$, \eqref{SNDf} implies
$$\left(\hat{S}_N^K(f)-S_N^K(f)\right)N^{-\alpha}\leq 2K\Res_{x^*}(f)\eta^{\beta}<\tilde\epsilon,$$
for $\eta$ sufficiently small. It remains to show that $G_N(\eta)$ has large probability. By Lemmas \ref{hatxklem} and \ref{xklem}, for fixed $K$,
$$
d(x^*,\hat{x}_K)\to 0,\quad d(x^*,x_K)\to 0\quad \text{for almost every } x.
$$
Since $\epsilon(r)\to0$ as $r\to0$, we get
$$
\mu\left(\max\biggl(\epsilon\Bigl(d(x^*,x_K)\Bigr), \epsilon\Bigl(d(x^*,\hat{x}_K)\Bigr)\biggr) \leq \eta^{2\beta}\right)\to 1.
$$
For the remaining condition, observe that
$$
\left\{d(x^*,\hat{x}_K)<\eta N^{-1/d}\right\}=\left\{S_N\Bigl(\one_{B_{\eta N^{-1/d}}(x^*)}\Bigr)\ge K\right\}.
$$
Since $S_N\Bigl(\one_{B_{\eta N^{-1/d}}(x^*)}\Bigr)\Rightarrow P(B_d\rho(x^*)\eta^d)$ by Proposition \ref{PLTprop}, where $P(t)$ denotes a Poisson random variable with mean $t$,
$$\lim_{N\to\infty}
\mu\left\{d(x^*,\hat{x}_K)<\eta N^{-1/d}\right\}=
\P\left\{P(B_d\rho(x^*)\eta^d)\ge K \right\} \ll \eta^{dK}.
$$
Combining the preceding estimates, for every $\tilde\epsilon>0$,
$$
\limsup_{N\to\infty}
\mu\left(
\hat{S}_N^K(f)-S_N^K(f)>\tilde\epsilon N^\alpha
\right)\ll\eta^{dK}.
$$
Letting $\eta\to0$ proves
$$
\lim_{N\to\infty}
\mu\left(
\hat{S}_N^K(f)-S_N^K(f)>\tilde\epsilon N^\alpha
\right)=0.
$$
\end{proof}

Before proving Theorem \ref{lighttrimthm}, we recall the metric that will be used to study weak convergence. Let $(\M,d_\M)$ be a compact metric space and let $\s{\tilde{\vartheta}}{n}$ be a sequence of Lipschitz functions on $\M$ dense in $C(\M)$. Set $\vartheta_n=\tilde{\vartheta}_n/\|\tilde{\vartheta}_n\|_{\mathrm{Lip}}$. For probability measures $\lambda$ and $\lambda'$ on $\M$, the metric
\begin{equation}\label{DK}
D_{\M}(\lambda,\lambda')=\sum_{n\geq 1} 2^{-n} \left|\int_{\M} \vartheta_n \d\lambda- \int_{\M} \vartheta_n \d\lambda'\right|,
\end{equation}
metrizes weak convergence. In the proof below, we take $\M=\overline{\mathbb R}=[-\infty,\infty]$ with the metric $d_{\overline{\mathbb R}}(r,s)=|\arctan r-\arctan s|,$ with the usual conventions at $\pm\infty$, and write $D_{\overline{\mathbb R}}$ for the corresponding metric on probability measures.

\begin{proof}[Proof of Theorem \ref{lighttrimthm}]
(i) We show the claims of the theorem for $\hat{S}_N^K(f)$ by proving that
\begin{equation}\label{hatSlightcon}
\frac{\hat{S}_N^K(f)-a_N}{N^{\alpha}} \Rightarrow Y \;\;\;\as{N}.
\end{equation}
Given Lemma \ref{nearlightlem}, the corresponding result for $S_N^K(f)$ follows with the same normalizing constants $a_N$.

(ii)
Set $r_N=(B_d \rho(x^*) N)^{-\frac{1}{d}}$ and let $R>2K$, which we regard as fixed for now. For $N \geq 1$ define
$$
\rho_N=\rho_{N,R}=R^{\frac{1}{d}} r_N,
$$
and let
$$
f_N=f_{N,R}= f \cdot \one_{B_{\rho_N}(x^*)}, \quad
\bar{f}_N=\bar{f}_{N,R} = f \cdot \left( 1 - \one_{B_{\rho_N}(x^*)} \right).
$$
Let $\chi_N = \chi_{N,R} = \one_{B_{\rho_N}(x^*)}$. Then
$$
\lim_{N\rightarrow \infty} \E\Bigl(S_N(\chi_N)\Bigr) = R,
$$
and the proof of Proposition \ref{PLTprop}, specifically \eqref{PoiFDmC}, shows that also
$$
\lim_{N\rightarrow \infty} \Var\Bigl(S_N(\chi_N)\Bigr)= R.
$$
It follows that
\begin{equation}\label{snksplit}
\begin{aligned}
&\limsup_{N\rightarrow \infty} \mu\left(\hat{S}_N^K(f) \neq \hat{S}_N^K(f_N) + S_N(\bar{f}_N) \right)
\leq \limsup_{N\rightarrow \infty} \mu\Bigl( S_N(\chi_N) \leq K-1 \Bigr)\\
&\leq \limsup_{N\rightarrow \infty}\mu\biggl(\left|S_N(\chi_N)-\E \Bigl(S_N(\chi_N)\Bigr)\right|\geq\E\Bigl(S_N(\chi_N)\Bigl)-(K-1)\biggr)\\
&\leq \lim_{N\rightarrow \infty}\frac{\Var\Bigl(S_N(\chi_N)\Bigr)}{\left(\E S_N(\chi_N)-(K-1)\right)^2} \leq \frac{4}{R}.
\end{aligned}
\end{equation}
In the remainder of the proof, we will analyze $\hat{S}_N^K(f_N)$ and $S_N(\bar{f}_N)$ separately.

Later on, in step (v), especially \eqref{LmVN}, we will see that the contribution of $S_N(\bar{f}_N)/N^\alpha$ is asymptotically deterministic, in the sense that its fluctuations around its expectation vanish. Therefore, for the main part of the proof, we focus on $\hat{S}_N^K(f_N)$.

(iii) Let $m\geq 1$ be a positive integer, which we regard as fixed for now, and partition the ball $B_{\rho_N}(x^*)$ into annular regions
\begin{equation*}
A_N^j = \left\{ x : \left( \frac{ R j}{m} \right)^{\frac{1}{d}} r_N \leq d(x, x^*) < \left( \frac{R (j+1)}{m} \right)^{\frac{1}{d}} r_N \right\}, \quad j=0,1,\ldots,m-1,
\end{equation*}
and introduce the trimmed counting process
\begin{equation*}
\mathcal{N}_{N,j}^K = \hat{S}_N^K \left(\one_{A_N^j}\right).
\end{equation*}
By Proposition \ref{trimPLTpropF}, the joint convergence
\begin{equation}\label{stablethmPois}
\left( \mathcal{N}_{N,j}^K \right)_{j=0,\dots,m-1} \Rightarrow \left( \# \left\{\Lambda^K \cap \left[\frac{Rj}{m},\frac{R(j+1)}{m}\right) \right\} \right)_{j=0,\dots,m-1} \;\;\;\as{N}
\end{equation}
holds. Set $c=\Res_{x^*}(f) B_d^{\alpha} \rho(x^*)^{\alpha}$ and define the bounds
\begin{equation*}
\theta_{N,m,R,-}^K = c \left(\frac{mN}{R}\right)^{\alpha} \sum_{j=0}^{m-1} \frac{1}{(j+1)^{\alpha}} \mathcal{N}_{N,j}^K,
\end{equation*}
and
\begin{equation*}
\theta_{N,m,R,+}^K = c \left(\frac{mN}{R}\right)^{\alpha} \sum_{j=1}^{m-1} \frac{1}{j^{\alpha}} \mathcal{N}_{N,j}^K.
\end{equation*}
From the joint convergence \eqref{stablethmPois}, we derive
\begin{equation}\label{thetaweakcon}
N^{-\alpha} \theta_{N,m,R,\pm}^K \Rightarrow \theta_{m,R,\pm}^K \quad \as{N},
\end{equation}
where
\begin{align*}
&\theta_{m,R,-}^K = c \left(\frac{m}{R}\right)^{\alpha} \sum_{j=0}^{m-1} \frac{1}{(j+1)^{\alpha}} \# \left\{\Lambda^K \cap \left[\frac{Rj}{m},\frac{R(j+1)}{m}\right) \right\} \; \text{ and}\\
&\theta_{m,R,+}^K = c \left(\frac{m}{R}\right)^{\alpha} \sum_{j=1}^{m-1} \frac{1}{j^{\alpha}} \# \left\{\Lambda^K \cap \left[\frac{Rj}{m},\frac{R(j+1)}{m}\right) \right\}.
\end{align*}
Since
\begin{equation*}
\Bigl(1-\epsilon(\rho_N)\Bigr)c\left(\frac{mN}{R(j+1)}\right)^\alpha<f(x)
<\Bigl(1+\epsilon(\rho_N)\Bigr)c\left(\frac{mN}{Rj}\right)^\alpha
\end{equation*}
for $x\in A_{N}^{j}$ and $1\leq j\leq m-1$, where $\epsilon$ comes from the local singularity estimate \eqref{Estf}, and
$$f(x)>\Bigl(1-\epsilon(\rho_N)\Bigr)c\left(\frac{mN}{R}\right)^\alpha\quad\text{for}\quad x\in A_{N}^{0},$$
we establish the lower bound
\begin{equation}\label{SNfNlob}
\hat{S}_N^K(f_N) \geq \Bigl(1-\epsilon(\rho_N)\Bigr) \theta_{N,m,R,-}^K,
\end{equation}
and the probabilistic upper bound
\begin{equation}\label{SNfNupb}
\begin{aligned}
&\limsup_{N\rightarrow\infty}\mu\left(\hat{S}_N^K(f_N) >\Bigl(1+\epsilon(\rho_N)\Bigr)\theta_{N,m,R,+}^K\right)\\
&\quad\leq\limsup_{N\rightarrow\infty}\mu\left(\left\{
\exists\,0\le k\le N-1: T^k(x) \in A_N^0\right\}\right)\\
& \quad\leq \lim_{N\rightarrow\infty} N \mu(A_N^0)\leq \frac{R}{m}.
\end{aligned}
\end{equation}

(iv)
Using Proposition \ref{PPPconprop}, there is a non-degenerate distribution $Y$ and constants $c_R$,  such that
\begin{equation*}
\sum_{x\in \Lambda^K \cap [0,R)} x^{-\alpha}- c_R \Rightarrow Y \;\;\; \as{R}.
\end{equation*}
Moreover, for later use, we explain why the discrete lower and upper sums approximate the Poisson sum. Fix a realization in the full-probability event on which $\Lambda^K\cap[0,R)$ is finite and does not contain $0$. If this set is empty, the desired convergence is immediate. Otherwise, write
\begin{equation*}
\Lambda^K\cap[0,R)=\{x_1,\dots,x_\ell\},\quad 0<x_1<\cdots<x_\ell<R.
\end{equation*}
For all sufficiently large $m$, the interval $[0,R/m)$ contains no points of $\Lambda^K$. Hence, for every $i=1,\dots,\ell$, there exists a unique $j_i\in\{1,\dots,m-1\}$ such that
\begin{equation*}
\frac{Rj_i}{m}\leq x_i<\frac{R(j_i+1)}{m}.
\end{equation*}
Therefore, for all sufficiently large $m$,
\begin{equation*}\frac{1}{c}\theta^K_{m,R,-}=
\sum_{i=1}^{\ell}\left(\frac{R(j_i+1)}{m}\right)^{-\alpha},
\quad\frac{1}{c}\theta^K_{m,R,+}=
\sum_{i=1}^{\ell}\left(\frac{Rj_i}{m}\right)^{-\alpha}.
\end{equation*}
Since the mesh size $\frac{R}{m}$ tends to zero, we have
$\frac{Rj_i}{m},\frac{R(j_i+1)}{m}\rightarrow x_i$
for every $i=1,\dots,\ell$. Consequently,
\begin{equation*}
\frac{1}{c}\theta^K_{m,R,-}
\rightarrow
\sum_{i=1}^{\ell}x_i^{-\alpha},
\qquad
\frac{1}{c}\theta^K_{m,R,+}
\rightarrow
\sum_{i=1}^{\ell}x_i^{-\alpha}.
\end{equation*}
Since the exceptional set of realizations has probability zero, we obtain
\begin{equation}\label{poisriesumcon}
\frac{1}{c}\theta^K_{m,R,\pm}
\rightarrow
\sum_{x\in\Lambda^K\cap[0,R)}x^{-\alpha}
\quad\text{almost surely as }m\to\infty .
\end{equation}
Consequently this convergence also holds in distribution.

(v)
The conclusion now follows from a diagonal argument. For $L\ge1$, first choose $\tilde R_L\ge 8L$ sufficiently large such that
\begin{align*}
D_{\overline{\mathbb R}}\left(\sum_{x\in \Lambda^K \cap [0,\tilde{R}_L)} x^{-\alpha}-c_{\tilde{R}_L}, Y\right) < \frac{1}{L},
\end{align*}
where $D_{\overline{\mathbb R}}$ is the metric defined in \eqref{DK}. By \eqref{poisriesumcon}, choose
$\tilde m_L\ge 2L\tilde R_L$ sufficiently large such that
\begin{equation*}
D_{\overline{\mathbb R}} \left(\frac{1}{c} \theta_{\tilde{m}_L,\tilde{R}_L,\pm}^K-c_{\tilde R_L},\sum_{x\in \Lambda^K \cap [0,\tilde{R}_L)}x^{-\alpha}-c_{\tilde R_L} \right) < \frac{1}{L}.
\end{equation*}
Using \eqref{snksplit}, \eqref{thetaweakcon}, and \eqref{SNfNupb}, we choose $\tilde{N}_L$ sufficiently large so that for all $N\ge \tilde N_L$, we have $\epsilon\left(\tilde{R}_L^{\frac{1}{d}} r_{N}\right) c_{\tilde{R}_L}<\frac{1}{L},$ where $\epsilon$ denotes the local singularity estimate \eqref{Estf}, and
\begin{equation}\label{tNLcond}
\begin{aligned}
D_{\overline{\mathbb R}}\left(\frac1cN^{-\alpha}\theta_{N,\tilde{m}_L,\tilde{R}_L,\pm}^K-c_{\tilde R_L}, \frac1c\theta_{\tilde{m}_L,\tilde{R}_L,\pm}^K-c_{\tilde R_L}\right) &< \frac{1}{L},\\
\mu\left(\hat{S}_N^K(f) \neq \hat{S}_N^K(f_{N,\tilde{R}_L}) + S_N(\bar{f}_{N,\tilde{R}_L})\right)&<\frac{1}{L},\\
\mu\left(\hat S_N^K(f_{N,\tilde R_L}) >\bigl(1+\epsilon(\tilde R_L^{1/d}r_N)\bigr)\theta_{N,\tilde m_L,\tilde R_L,+}^K\right) &<\frac1L.
\end{aligned}
\end{equation}
For $N\in [\tilde{N}_L, \tilde{N}_{L+1})$, set $R_N=\tilde{R}_L$ and $m_N=\tilde{m}_L.$ By the choice of $\tilde m_L,\tilde R_L$ and the first line of \eqref{tNLcond}, the triangle inequality gives
\begin{equation*}
D_{\overline{\mathbb R}}\left(\frac{1}{c} N^{-\alpha} \theta_{N,m_N,R_N,\pm}^K - c_{R_N}, Y\right)<\frac{3}{L}.
\end{equation*}
Hence
\begin{equation*}
\frac{1}{c} N^{-\alpha}\theta_{N,m_N,R_N,\pm}^K - c_{R_N}\Rightarrow Y \quad \as{N}.
\end{equation*}

Using Lemma \ref{sandwichlem} with
\begin{align*}
X_N&=\left(1-\epsilon(\rho_N)\right) \frac{1}{c} N^{-\alpha} \theta_{N,m_N,R_N,-}^K - c_{R_N},\\
Y_N&=\frac{1}{c} N^{-\alpha} \hat{S}_N^K(f_{N,R_N}) - c_{R_N},\; \text{ and}\\
Z_N&=\left(1+\epsilon(\rho_N)\right) \frac{1}{c} N^{-\alpha} \theta_{N,m_N,R_N,+}^K - c_{R_N},
\end{align*}
it follows from \eqref{SNfNlob} and \eqref{SNfNupb} that
\begin{equation}\label{mainSNcon}
\frac{1}{c} N^{-\alpha} \hat{S}_N^K(f_{N,R_N}) - c_{R_N} \Rightarrow Y \quad \as{N}.
\end{equation}
It remains to estimate the contribution of $\bar f_{N,R_N}$. Since now we have $N \rho_N^d\rightarrow \infty$, we can apply Proposition~\ref{momentprop} with the regularly adapted sequence $A_N = B_{\rho_N}(x^*)$ and $m_1=0,m_2=2$,
\begin{equation*}
\Var\Bigl(S_N(\bar{f}_{N,R_N})\Bigr) \ll  N \rho_N^{-2\beta + d} \ll N^{2\alpha} R_N^{1 - 2\alpha},
\end{equation*}
as $N\rightarrow \infty$. By Chebyshev's inequality, for every $\delta>0$ and $N\in[\tilde N_L,\tilde N_{L+1})$,
\begin{align*}
\mu\left(\left|\frac{S_N(\bar f_{N,R_N})-N\E(\bar f_{N,R_N})}{N^\alpha}\right|>\delta\right)
&\le\delta^{-2}N^{-2\alpha}\Var\left(S_N(\bar f_{N,R_N})\right)\\
&\ll R_N^{1-2\alpha}.
\end{align*}
Since $\alpha>\frac{1}{2}$ and $R_N\to\infty$, the right-hand side tends to $0$. Therefore
\begin{equation}\label{LmVN}
\frac{S_N(\bar f_{N,R_N})-N\E(\bar f_{N,R_N})}{N^\alpha}
\to0
\end{equation}
in probability.

Finally, combining \eqref{mainSNcon}, \eqref{LmVN}, and the second line of \eqref{tNLcond}, we obtain
\begin{equation*}
\frac{1}{c}N^{-\alpha}\left(\hat S_N^K(f)-cN^\alpha c_{R_N}-N\E(\bar f_{N,R_N})\right)
\Rightarrow Y\quad \as{N}.
\end{equation*}
Consequently,
$N^{-\alpha}\left(\hat S_N^K(f)-a_N\right)\Rightarrow cY.$
Since $c>0$, the random variable $cY$ is non-degenerate, and
relabelling it as $Y$ gives \eqref{hatSlightcon}.

For $\frac{1}{2}<\alpha\leq 1$, using the expression for $c_R$ obtained in Proposition \ref{PPPconprop}, we have
$$
a_N=cN^\alpha\int_1^{R_N}x^{-\alpha}\,\d x+N\E(\bar f_{N,R_N}).
$$
When $\alpha>1$, we may take $a_N=0$, since $c_{R_N}=0$ and
$N^{-\alpha}N\E(\bar f_{N,R_N})\ll R_N^{1-\alpha}\to0.$

Since the diagonal construction above can be carried out simultaneously for all $0\le K\le L$ at each stage $L$, the resulting sequence $R_N$, and hence $a_N$, does not depend on $K$. Its stated asymptotics follow by direct computation.

If $K=0$, then $\hat S_N^0(f)=S_N^0(f)=S_N(f).$ If $K\ge1$, passing from $\hat{S}_N^K(f)$ to $S_N^K(f)$ follows from Lemma~\ref{nearlightlem}.

\end{proof}

\section{Facts about distributional convergence}\label{metricsec}
We collect here two lemmas concerning the metric $D_{\overline{\mathbb R}}$ and weak convergence, as defined in Section \ref{LgTrDL}.

\begin{lemma}\label{sandwichlem}
Let $X_N,Y_N,Z_N,Y$ be real-valued random variables with
\begin{align*}
X_N \Rightarrow Y \; \text{ and } \; Z_N\Rightarrow Y \quad \as{N}.
\end{align*}
If $X_N,Y_N,Z_N$ are defined on a common probability space and
\begin{equation}\label{squeezeasm}
\P( Y_N < X_N) \rightarrow 0 \; \text{ and }\; \P(Y_N > Z_N) \rightarrow 0 \quad \as{N},
\end{equation}
then
\begin{equation*}
Y_N\Rightarrow Y \quad \as{N}.
\end{equation*}
\end{lemma}
\begin{proof}
Let $\xi$ be a continuity point of the distribution function of $Y.$ Then
\begin{equation*}
\liminf_{N\to \infty} \P(Y_N \leq \xi) \geq \liminf_{N\to \infty}\P(Z_N \leq \xi) -\lim_{N\to \infty} \P(Y_N > Z_N) =\P(Y\leq \xi),
\end{equation*}
and
\begin{equation*}
\limsup_{N\to \infty} \P(Y_N \leq \xi) \leq \limsup_{N\to \infty} \P(X_N \leq \xi) + \lim_{N\to \infty}\P( Y_N < X_N)=\P(Y\leq \xi).
\end{equation*}
\end{proof}

\begin{lemma}\label{sandwichhighlem}
Let $K\geq 1$ be fixed, and suppose that for each $N$ we have $K$-dimensional random vectors
$\mathbf{X}_N = \bigl( X_N^{(1)}, \dots, X_N^{(K)} \bigr),$
$\mathbf{Y}_N = \bigl( Y_N^{(1)}, \dots, Y_N^{(K)} \bigr),$ and
$\mathbf{Z}_N = \bigl( Z_N^{(1)}, \dots, Z_N^{(K)} \bigr),$ such that as $N \to \infty$,
$$
\mathbf{X}_N \Rightarrow \mathbf{Y} \quad \text{and} \quad \mathbf{Z}_N \Rightarrow \mathbf{Y},
$$
where $\mathbf{Y} = (Y^{(1)},\dots,Y^{(K)})$ is a $K$-dimensional random vector.
Assume further that for each $k=1,\dots,K$ and every $\epsilon\in (0,1)$, the following conditions hold as $N\rightarrow \infty$:
\begin{equation}\label{MDfYNXN}
\begin{aligned}
\mathbb{P}\left(Y_N^{(k)} \geq 0, \, Y_N^{(k)} < (1 - \epsilon) X_N^{(k)}\right) &\to 0, \\
\mathbb{P}\left(Y_N^{(k)} \leq 0, \, Y_N^{(k)} < (1 + \epsilon) X_N^{(k)}\right) &\to 0, \\
\mathbb{P}\left(Y_N^{(k)} \geq 0, \, Y_N^{(k)} > (1 + \epsilon) Z_N^{(k)}\right) &\to 0, \\
\mathbb{P}\left(Y_N^{(k)} \leq 0, \, Y_N^{(k)} > (1 - \epsilon) Z_N^{(k)}\right) &\to 0.
\end{aligned}
\end{equation}
Then
\begin{equation*}
\mathbf{Y}_N \Rightarrow \mathbf{Y} \quad \as{N}.
\end{equation*}
\end{lemma}
\begin{proof}
Choose a sequence $\epsilon_N \to 0$ such that all limits in \eqref{MDfYNXN} still hold with $\epsilon$ replaced by $\epsilon_N$. For each $k$ define
\begin{equation*}
\widetilde{X}_N^{(k)}= \begin{cases}
(1-\epsilon_N) X_N^{(k)} & \text{if } Y_N^{(k)} \geq 0,\\
(1+\epsilon_N) X_N^{(k)} & \text{if } Y_N^{(k)} < 0,
\end{cases}
\end{equation*}
and
\begin{equation*}
\widetilde{Z}_N^{(k)}=\begin{cases}
(1+\epsilon_N) Z_N^{(k)} & \text{if } Y_N^{(k)} \geq 0,\\
(1-\epsilon_N) Z_N^{(k)} & \text{if } Y_N^{(k)} < 0.
\end{cases}
\end{equation*}
By splitting according to the sign of $Y_N^{(k)}$ and using \eqref{MDfYNXN}, we obtain, for every $k$,
\begin{equation}\label{Cppb}
\mathbb{P}\bigl(Y_N^{(k)} < \widetilde{X}_N^{(k)}\bigr) \to 0,\quad
\mathbb{P}\bigl(Y_N^{(k)} > \widetilde{Z}_N^{(k)}\bigr) \to 0.
\end{equation}
Since $\mathbf{X}_N \Rightarrow \mathbf{Y}$, $\mathbf{Z}_N \Rightarrow \mathbf{Y}$ and $\epsilon_N \to 0$, Slutsky's theorem yields
\begin{equation}\label{StC}
\widetilde{\mathbf{X}}_N \Rightarrow \mathbf{Y},\quad \widetilde{\mathbf{Z}}_N \Rightarrow \mathbf{Y}.
\end{equation}

Let
$$E_N=\bigcap_{k=1}^{K}\left\{\widetilde X_N^{(k)}\le Y_N^{(k)}\le\widetilde Z_N^{(k)}\right\}.$$
By \eqref{Cppb},
$$\mathbb P(E_N^c)\le\sum_{k=1}^{K}\left[\mathbb P\left(Y_N^{(k)}<\widetilde X_N^{(k)}\right)
+\mathbb P\left(Y_N^{(k)}>\widetilde Z_N^{(k)}\right)\right]\to0.
$$
For any $x=(x_1,\ldots,x_K)\in\mathbb R^K$, on the event $E_N,$ we have $\left\{\widetilde{\mathbf Z}_N\le x\right\}\subseteq\left\{\mathbf Y_N\le x\right\}\subseteq\left\{\widetilde{\mathbf X}_N\le x\right\}.$
Consequently,
$$\mathbb P\left(\widetilde{\mathbf Z}_N\le x\right)-\mathbb P(E_N^c)\le\mathbb P\left(\mathbf Y_N\le x\right)\le\mathbb P\left(\widetilde{\mathbf X}_N\le x\right)+\mathbb P(E_N^c).
$$
Now let $x$ be a continuity point of the distribution function of $\mathbf Y$. Using \eqref{StC}, both the left and the right sides  of the above inequality converge to $\mathbb P(\mathbf Y\le x).$
It follows that $\mathbb P(\mathbf Y_N\le x)\longrightarrow\mathbb P(\mathbf Y\le x),$ and therefore $\mathbf Y_N\Rightarrow\mathbf Y.$
\end{proof}

\section{Intermediately trimmed distributional limits}\label{intersec}
In this section, we assume that the system $(M, \mathcal{B}(M), \mu, T)$ is exponentially mixing of all orders, and that $\Ord_{x^*}(f) = \beta >\frac{d}{2}$. We will prove an intermediately trimmed distributional law, which establishes Theorem \ref{intertrimthm}. Without loss of generality, we assume that $\Res_{x^*}(f) = 1$ in this section.

Recall that $x_i$ is the $i$th closest point to $x^*$ among the $k(N)$ points selected to maximize their $f$ values, satisfying
\begin{align*}
&d(x_1,x^*) \leq d(x_2,x^*)\leq\ldots\leq d(x_{k(N)}, x^*)\\
& f(x_i) \geq f(T^j(x)) \quad \text{for } 1\leq i\leq k(N),\, 0\leq j\leq N-1,\, T^j(x)\not\in \{x_1,\ldots,x_{k(N)}\}.
\end{align*}
The point $\hat{x}_i$ represents the $i$th closest point to $x^*$ among all of $\{x, T(x), \ldots, T^{N-1}(x)\}$. With these definitions, the trimmed sum is given by
\begin{equation*}
S_N^{k(N)}(f)(x) = S_N(f)(x) - \sum_{i=1}^{k(N)} f(x_i),
\end{equation*}
and the version defined in Section~\ref{nearsec} is given by
\begin{equation*}
\hat{S}_N^{k(N)}(f)(x) = S_N(f)(x) - \sum_{i=1}^{k(N)} f(\hat{x}_i).
\end{equation*}

To prevent the proof of Theorem \ref{intertrimthm} from being obscured by technical details, we first explain the idea for the geometrically trimmed sum $\hat{S}_N^{k(N)}(f)$, a simpler case, before proceeding with the full technical arguments.

Let $r_N>0$ be chosen by
$$
\mu(B_{r_N}(x^*))=\frac{k(N)}{N}.
$$
Since $\mu(B_r(x^*)) \sim B_d \rho(x^*) r^d$, it follows that $r_N^d\sim \frac{k(N)}{B_d \rho(x^*) N}$.
Define $\hat{f}_N = f\cdot \left(1-\one_{B_{r_N}(x^*)} \right).$ Then
\begin{equation}\label{DchS}
\hat{S}_N^{k(N)}(f)= S_N(\hat{f}_N)+\hat{W}_N,
\end{equation}
where
$$\hat{W}_N=S_N\left(f \cdot \one_{B_{r_N}(x^*)}\right)-\sum_{i=1}^{k(N)} f(\hat{x}_i).$$
Writing $\xi_N=d(x^*,\hat{x}_{k(N)})$, heuristically we expect $\xi_N\approx r_N$. We analyze $\hat{W}$ by considering the value of $\xi_N$:
\begin{itemize}
\item If $\xi_N\leq r_N$, then all of the points $\hat{x}_i$ lie in $B_{r_N}(x^*)$; in this case $\hat{W}_N$ is the sum of $f$ over all the points in $B_{r_N}(x^*)$ that are not one of the $\hat{x}_i$,
\item Otherwise, if $\xi_N>r_N$, then every point in $B_{r_N}(x^*)$ is one of the $\hat{x}_i$, in this case $-\hat{W}_N$ is the sum of $f$ over the remaining points.
\end{itemize}
In both cases, all of the points that are being summed over are between $r_N$ and $\xi_N$, keeping in mind that $\xi_N \approx r_N$ it follows that
\begin{equation}\label{hatWeq}
\hat{W}_N \approx \frac{S_N\bigl(\one_{B_{r_N}(x^*)}\bigr)-k(N)}{r_N^{\beta}}.
\end{equation}
From Corollary \ref{CLTgenrem} with the regularly adapted sequence $A_N=B_{r_N}(x^*)$, it follows that
\begin{equation}\label{CLTbal}
\left(\frac{S_N\left(\one_{B_{r_N}(x^*)}\right)-k(N)}{\sqrt{k(N)}},\, \frac{S_N(\hat{f}_N)-N\E\hat{f}_N}{\sqrt{k(N)}\,r_N^{-\beta}}\right) \Rightarrow (\mathcal{N}(0,1), \mathcal{N}(0,\sigma^2)),
\end{equation}
as $N\to\infty,$ where $(\mathcal{N}(0,1), \mathcal{N}(0,\sigma^2))$ is a vector of independent normal random variables, and
$\sigma^2=\frac{1}{2\alpha-1}.$ Combining \eqref{DchS}, \eqref{hatWeq} and \eqref{CLTbal} completes the proof for $\hat{S}_N^{k(N)}(f)$.

The main technical challenge in going from $\hat{S}_N^{k(N)}$ to $S_N^{k(N)}$ is that, if we simply try to split $S_N^{k(N)}$ in the same way, setting
\begin{equation*}
W_N = S_N\left(f \cdot \one_{B_{r_N}(x^*)}\right) - \sum_{i=1}^{k(N)} f(x_i),
\end{equation*}
there is no analogue of \eqref{hatWeq}.

Motivated by the proof of Lemma~\ref{xklem}, instead of balls we will use superlevel sets of $f$. Therefore, let
\begin{equation*}
\Sigma_R = \{f\geq R\} \quad \textit{for } R>0,
\end{equation*}
and let $\lambda_N>0$ be the unique value such that
\begin{equation*}
\mu(\Sigma_{\lambda_N})=\frac{k(N)}{N}.
\end{equation*}
From Lemma \ref{inclanlem} it follows that
\begin{equation*}
\lambda_N \sim B_d^{\alpha} \rho(x^*)^{\alpha} N^{\alpha} k(N)^{-\alpha} \sim r_N^{-\beta}.
\end{equation*}

Now, defining
\begin{equation}\label{Wdef}
W_N = S_N\left(f \cdot \one_{\Sigma_{\lambda_N}}\right) - \sum_{i=1}^{k(N)} f(x_i),
\end{equation}
it \emph{does} hold that
\begin{equation}\label{Weq}
W_N \approx \lambda_N \left(S_N(\one_{\Sigma_{\lambda_N}}) - k(N) \right),
\end{equation}
as will be rigorously justified below, in direct analogy with \eqref{hatWeq}.

The preceding heuristic argument can be made rigorous once the assumptions of Corollary~\ref{CLTgenrem} have been verified for the superlevel sets $\Sigma_{\lambda_N}$. We therefore show that the sequence $\Sigma_{\lambda_N}$ is regularly adapted in the sense of Section~\ref{Tr}.

\begin{lemma}\label{inclanlem}
Let $R_N\rightarrow \infty$. Then there are $\epsilon_N \rightarrow 0$ such that
\begin{equation}\label{inclanlemeq}
B_{(1-\epsilon_N) R_N^{-1/\beta}}(x^*) \subset \Sigma_{R_N} \subset B_{(1+\epsilon_N) R_N^{-1/\beta}}(x^*).
\end{equation}
\end{lemma}

\begin{proof}
Fix a constant $C_1>\max\{1,C_0^{1/\beta}\}$. The result follows from \eqref{Estf} with
\begin{equation*}
\epsilon_N=\max\left(1-\left(1-\epsilon(R_N^{-1/\beta})\right)^{1/\beta},\,
\left(1+\epsilon(C_1R_N^{-1/\beta})\right)^{1/\beta}-1\right).
\end{equation*}
Clearly, $\epsilon_N\to0$. Indeed, if $d(x,x^*) < (1-\epsilon_N)R_N^{-1/\beta}$, then
\begin{align*}
f(x)&> \left(1-\epsilon(R_N^{-1/\beta})\right)\left((1-\epsilon_N)R_N^{-1/\beta}\right)^{-\beta}\\
&\geq \left(1-\epsilon(R_N^{-1/\beta})\right)\left(1-\epsilon(R_N^{-1/\beta})\right)^{-1}R_N = R_N,
\end{align*}
so $x \in \Sigma_{R_N}$. Conversely, suppose $x\in\Sigma_{R_N}$. By the global estimate
$f(x)\leq C_0d(x,x^*)^{-\beta}$, we have $d(x,x^*)\leq C_0^{1/\beta}R_N^{-1/\beta}<C_1R_N^{-1/\beta}$.
Applying \eqref{Estf} with radius $C_1R_N^{-1/\beta}$ then yields
$$R_N\leq f(x)<\left(1+\epsilon(C_1R_N^{-1/\beta})\right)d(x,x^*)^{-\beta},$$
and therefore
$$d(x,x^*)<\left(1+\epsilon(C_1R_N^{-1/\beta})\right)^{1/\beta}
R_N^{-1/\beta}\leq(1+\epsilon_N)R_N^{-1/\beta}.$$
Thus $x\in B_{(1+\epsilon_N)R_N^{-1/\beta}}(x^*)$.
\end{proof}

A direct corollary of Lemma \ref{inclanlem} is that for $R_N \to \infty$,
\begin{equation}\label{muanlem}
\mu(\Sigma_{R_N}) \sim B_d \rho(x^*) R_N^{-\frac{1}{\alpha}}.
\end{equation}

\begin{lemma}\label{approxanlem}
For $R_N\rightarrow \infty$ and $\epsilon_N \rightarrow 0$, there exist $C^{\kappa}$ functions $h_N:M\rightarrow[0,1]$ such that
$\one_{\Sigma_{R_N}}\leq h_N\leq\one_{\Sigma_{R_N-\epsilon_N}}$,
\begin{equation*}
\|\one_{\Sigma_{R_N}} - h_N\|_{L^1} \ll \epsilon_N R_N^{-\frac{1}{\alpha}-1} \quad \text{and} \quad \|h_N\|_{C^{\kappa}} \ll \epsilon_N^{-\kappa} R_N^{\kappa+\frac{\kappa}{\beta}}.
\end{equation*}
\end{lemma}

\begin{proof}
Let $\theta_N:[0,\infty) \rightarrow [0,1]$ be a smooth function with
\begin{itemize}
\item $\theta_N(r)=1$ if $r\geq R_N$,
\item $\theta_N(r)=0$ if $r\leq R_N-\epsilon_N$,
\item and $\|\theta_N\|_{C^{\kappa}} \ll \epsilon_N^{-\kappa}$.
\end{itemize}
Define $h_N(x)=\theta_N(f(x))$ for $x\neq x^*$ and $h_N(x)=1$ for $x=x^*.$
We verify the required properties for $h_N$. Using \eqref{inclanlemeq}, it follows that $h_N= 1$ on $\Sigma_{R_N} \supset B_{R_N^{-1/\beta}/2}(x^*)$, and $h_N= 0$ outside $\Sigma_{R_N-\epsilon_N} \subset B_{2R_N^{-1/\beta}}(x^*)$. Therefore, the standard $C^\kappa$ composition estimate implies that
\begin{equation*}
\|h_N\|_{C^{\kappa}} \ll \epsilon_N^{-\kappa} R_N^{\kappa+\frac{\kappa}{\beta}}.
\end{equation*}
On the other hand, to estimate $\|\one_{\Sigma_{R_N}}-h_N\|_{L^1} \leq \mu(R_N-\epsilon_N \leq f \leq R_N)$, write $f(x) = \phi(x) \cdot g(x)$ with $\phi(x) = d(x^*, x)^{-\beta}$ and $g\in C^{\kappa}$ as in Definition~\ref{PowSing}. Using the product rule $D_x f=D_x \phi \cdot g(x)+\phi(x) \cdot D_x g$,
and recalling that by assumption $g(x^*)=\Res_{x^*} (f)=1$, we have
\begin{equation}\label{DfEst}
\| D_x f \| \sim \beta d(x^*,x)^{-\beta-1} \quad \textit{as } x\rightarrow x^*.
\end{equation}
For $t \in [R_N-\epsilon_N, R_N]$, the level set $L_t=\{f=t\}$ is a $C^{\kappa}$-regular $(d-1)$-dimensional hypersurface and its $(d-1)$-dimensional Hausdorff measure scales as $H^{d-1}(L_t)=O\left(t^{-(d-1)/\beta}\right)=O\left(R_N^{-(d-1)/\beta}\right)$. Indeed, even though $L_t$ is not simply a hypersphere, its deviation from the hypersphere is determined by $g$, a $C^{\kappa}$ function whose derivative is bounded. By the co-area formula,
\begin{equation}\label{Coarea}
\mu(R_N - \epsilon_N \leq f \leq R_N )
= \int_{R_N - \epsilon_N}^{R_N} \left( \int_{L_t} \frac{\rho(x)}{\|D_x f\|} \d H^{d-1}(x) \right)\d t.
\end{equation}
On $L_t$, it holds that $\|D_x f\| \asymp R_N^{1 + 1/\beta}$ according to \eqref{DfEst}, so \eqref{Coarea} is bounded above by $O\left(\epsilon_N R_N^{-1/\alpha-1}\right)$. Hence,
\begin{equation*}
\| \one_{\Sigma_{R_N}}-h_N\|_{L^1} \ll \epsilon_N R_N^{-\frac{1}{\alpha}-1}.
\end{equation*}
\end{proof}

The two lemmas above show that $\Sigma_{\lambda_N}$ is regularly adapted, and Corollary~\ref{CLTgenrem} can be applied.

\begin{lemma}\label{CLTlem}
Let $\lambda_N>0$ be chosen such that
$$
\mu(\Sigma_{\lambda_N})=\frac{k(N)}{N},
$$ where $\Sigma_R=\{f\ge R\},$
and define
$$f_N=f \cdot \left(1-\one_{\Sigma_{\lambda_N}} \right).$$
Assume $k(N)\to\infty$ and $k(N)=o(N)$. Then
\begin{equation}\label{normalclaim}
\left( \frac{S_N(\one_{\Sigma_{\lambda_N}}) - k(N)}{\sqrt{k(N)}}, \frac{S_N(f_N) - N \E f_N}{\sqrt{k(N)} \lambda_N} \right) \Rightarrow (\mathcal{N}(0,1), \mathcal{N}(0,\sigma^2)) \quad \as{N},
\end{equation}
where $(\mathcal{N}(0,1), \mathcal{N}(0,\sigma^2))$ is a vector of independent normal random variables, and
\begin{equation*}
\sigma^2= \frac{1}{2\alpha-1}.
\end{equation*}
\end{lemma}
\begin{proof}
By Lemmas \ref{inclanlem} and \ref{approxanlem}, the superlevel sets $\Sigma_{\lambda_N}$ form a regularly adapted sequence. Moreover, their effective radii satisfy
$$r_N\sim \lambda_N^{-\frac{1}{\beta}}\sim\left(\frac{k(N)}{B_d\rho(x^*)N}\right)^{\frac{1}{d}}.$$
In particular, $Nr_N^d\sim \frac{k(N)}{B_d\rho(x^*)}\to\infty.$ Applying Corollary~\ref{CLTgenrem} with $A_N=\Sigma_{\lambda_N}$, the claim follows.

\end{proof}

\begin{lemma}\label{normallem}
Under the same hypotheses as Lemma~\ref{CLTlem}, for every $t>0$ we have
\begin{equation*}
\frac{S_N(\one_{\Sigma_{t \lambda_N}}) - N \mu(\Sigma_{t \lambda_N})}{\sqrt{k(N)}} \Rightarrow \mathcal{N}(0,t^{-\frac{1}{\alpha}}) \quad \as{N},
\end{equation*}
where $\mathcal{N}(0,v)$ is a normal distribution with mean $0$ and variance $v$.
\end{lemma}
\begin{proof}
Fix $t>0$, and set $\tilde\lambda_N=t\lambda_N,$ $\tilde k(N)=N\mu(\Sigma_{\tilde\lambda_N}).$
Since $\mu(\Sigma_R)\sim B_d\rho(x^*)R^{-\frac{1}{\alpha}}$, we have
$\tilde k(N)\sim B_d\rho(x^*)N (t\lambda_N)^{-\frac{1}{\alpha}}\sim t^{-\frac{1}{\alpha}}k(N).$
Applying Lemma~\ref{CLTlem} with $\tilde\lambda_N$ and $\tilde k(N)$ gives
$$
\frac{S_N(\one_{\Sigma_{\tilde\lambda_N}})-\tilde k(N)}{\sqrt{\tilde k(N)}}
\Rightarrow \mathcal N(0,1).
$$
Now rewrite
$$
\frac{S_N(\mathbf{1}_{\Sigma_{t\lambda_N}})-N\mu(\Sigma_{t\lambda_N})}{\sqrt{k(N)}}
=\sqrt{\frac{\tilde{k}(N)}{k(N)}}\;\cdot\;\frac{S_N(\mathbf{1}_{\Sigma_{\tilde{\lambda}_N}})-\tilde{k}(N)}{\sqrt{\tilde{k}(N)}}.
$$
Since $\tilde k(N)/k(N)\to t^{-1/\alpha}$, the desired result follows by Slutsky's theorem.
\end{proof}

Letting
\begin{equation*}
f_N=f \cdot \left(1-\one_{\Sigma_{\lambda_N}} \right),
\end{equation*}
and with $W_N$ defined as in \eqref{Wdef}, we have
\begin{equation*}
S_N^{k(N)}(f) = S_N(f_N)+W_N.
\end{equation*}
The conclusion of the theorem will follow once we show
\begin{itemize}
\item[(A)] \begin{equation*}
\frac{S_N(f_N)-N \E f_N}{B_d^{\alpha} \rho(x^*)^{\alpha} N^{\alpha} k(N)^{\frac{1}{2}-\alpha}} \Rightarrow \mathcal{N}(0,\sigma^2) \quad \as{N},
\end{equation*}
where
\begin{equation*}
\sigma^2 = \frac{1}{2\alpha - 1},
\end{equation*}
\item[(B)] \begin{equation*}
\frac{W_N}{B_d^{\alpha} \rho(x^*)^{\alpha} N^{\alpha} k(N)^{\frac{1}{2}-\alpha}} \Rightarrow \mathcal{N}(0,1) \quad \as{N},
\end{equation*}
\item[(C)] $(S_N(f_N)-N\E f_N) / (B_d^{\alpha} \rho(x^*)^{\alpha} N^{\alpha} k(N)^{\frac{1}{2}-\alpha})$ is asymptotically independent of $W_N /(B_d^{\alpha} \rho(x^*)^{\alpha} N^{\alpha} k(N)^{\frac{1}{2}-\alpha})$.
\end{itemize}

Recall that $W_N = S_N\left(f \cdot \one_{\Sigma_{\lambda_N}}\right) - \sum_{i=1}^{k(N)} f(x_i),$
where $x_1,\dots,x_{k(N)}$ are the orbit points where $f$ assumes the largest values. Now one of two alternatives can happen:
\begin{itemize}
\item[(i)] If $f(x_i) \geq \lambda_N$ for all $1 \leq i \leq k(N)$, then $W_N$ is the sum of $f(y)$ over all points $y \in \Sigma_{\lambda_N}$ that are not among the $x_i$'s; hence $W_N \geq 0$.
\item[(ii)] Otherwise, there exists some $i$ with $1 \leq i \leq k(N)$ such that $f(x_i) < \lambda_N$; in this case, all points in $\Sigma_{\lambda_N}$ are already included among the $x_i$'s, and $-W_N$ equals the sum of $f$ over those $x_i$ lying outside $\Sigma_{\lambda_N}$, so $W_N< 0$.
\end{itemize}

In both situations, $W_N$ is (up to sign) a sum of $f(y)$ over certain orbit points. We prove upper and lower bounds for $W_N$ in both cases.

Notice first a simple lower bound. In case (i) each term $f(y)$ satisfies $f(y)\geq \lambda_N$ because $y\in\Sigma_{\lambda_N}$; in case (ii) each term satisfies $f(y)<\lambda_N$ because $y\notin\Sigma_{\lambda_N}$.
The number of terms is $S_N(\one_{\Sigma_{\lambda_N}})-k(N)$, which is negative in case (ii).
Thus, in both cases,
\begin{equation}\label{Wntrivbound}
W_N\geq \lambda_N\left(S_N(\one_{\Sigma_{\lambda_N}})-k(N)\right).
\end{equation}

We now turn to obtaining a probabilistic upper bound for $W_N$ in case (i). Let $J=S_N\bigl(\one_{\Sigma_{\lambda_N}}\bigr)-k(N)$ and denote
\begin{equation*}
\{y_1,\dots , y_J\}= \{T^n(x) \;|\;0\leq n\leq N-1, T^n(x) \in \Sigma_{\lambda_N} \setminus \{x_1,\dots,x_{k(N)}\}\},
\end{equation*}
then $W_N=\sum_{j=1}^J f(y_j)$. If
$$
W_N>(1+\epsilon)\lambda_N\left(S_N\bigl(\one_{\Sigma_{\lambda_N}}\bigr)-k(N)\right),
$$
then $J\geq1,$ and at least one of the $y_j$'s satisfies $f(y_j)>(1+\epsilon)\lambda_N$. Since the $x_i$'s are defined as the points with the largest $f$-values, it follows that $f(x_i)\geq f(y_j)>(1+\epsilon)\lambda_N$ for all $i=1,\dots, k(N)$. Consequently,
$$S_N\left(\one_{\Sigma_{(1+\epsilon)\lambda_N}}\right)\ge k(N)+1\geq k(N).$$
Hence
\begin{equation}\label{WNpi}
\left\{W_N > 0,\, W_N > (1 + \epsilon) \lambda_N \left(S_N(\one_{\Sigma_{\lambda_N}})-k(N)\right)\right\}
\subset\left\{S_N\left(\one_{\Sigma_{(1+\epsilon)\lambda_N}}\right)\ge k(N)\right\}.
\end{equation}
For $t>0$, Lemma~\ref{normallem} gives
\begin{equation}\label{CLTsps}
\frac{S_N(\one_{\Sigma_{t \lambda_N}}) - N \mu(\Sigma_{t \lambda_N})}{\sqrt{k(N)}} \Rightarrow \mathcal{N}(0,t^{-\frac{1}{\alpha}}) \quad \as{N}.
\end{equation}
Applying \eqref{CLTsps} with $t=1+\epsilon$ we obtain
\begin{equation}\label{CLT1p}
\frac{S_N\bigl(\one_{\Sigma_{(1+\epsilon)\lambda_N}}\bigr)
-N\mu\bigl(\Sigma_{(1+\varepsilon)\lambda_N}\bigr)}{\sqrt{k(N)}}
\Rightarrow\mathcal{N}\bigl(0,(1+\epsilon)^{-1/\alpha}\bigr),
\end{equation}
Moreover, \eqref{muanlem} implies that
$$
N\mu\left(\Sigma_{(1+\epsilon)\lambda_N}\right)
\sim
(1+\varepsilon)^{-\frac{1}{\alpha}}k(N).
$$
Thus, for all sufficiently large $N$,
\begin{equation}\label{kN1p}
k(N)-N\mu(\Sigma_{(1+\epsilon)\lambda_N})
\ge c(\epsilon)k(N)
\end{equation}
with some $c(\epsilon)>0$. Using \eqref{CLT1p} and \eqref{kN1p}, we have
\begin{align*}
&\mu\left\{S_N\left(\one_{\Sigma_{(1+\epsilon)\lambda_N}}\right)\ge k(N)\right\}\\
&\le\mu\left\{
\frac{S_N\left(\one_{\Sigma_{(1+\epsilon)\lambda_N}}\right)-N\mu(\Sigma_{(1+\epsilon)\lambda_N})}{\sqrt{k(N)}}
\ge c(\epsilon)\sqrt{k(N)}\right\}\to0.
\end{align*}
Combining this with \eqref{WNpi} proves
\begin{equation}\label{Wnposboundl}
\lim_{N\rightarrow\infty} \mu\left(W_N>0,\, W_N>(1+\epsilon)\lambda_N \left(S_N(\one_{\Sigma_{\lambda_N}})-k(N)\right) \right) = 0.
\end{equation}
Moreover, using \eqref{CLTsps} with $t=1$ we have $\lim_{N\rightarrow\infty} \mu(W_N=0)=0.$
Therefore,
\begin{equation}\label{Wnposbound}
\lim_{N\rightarrow\infty} \mu\left(W_N\ge0,\, W_N>(1+\epsilon)\lambda_N \left(S_N(\one_{\Sigma_{\lambda_N}})-k(N)\right) \right)= 0,\quad \forall \epsilon>0.
\end{equation}

For case (ii) we argue similarly: if $f(x_i) < (1-\varepsilon)\lambda_N$ for some $1\leq i\leq k(N)$, then $\Sigma_{(1-\varepsilon)\lambda_N}$ contains at most $k(N)-1$ of the orbit points, so $S_N\bigl(\one_{\Sigma_{(1-\varepsilon)\lambda_N}}\bigr) \le k(N)-1.$ Then the same CLT argument using \eqref{CLTsps} gives
\begin{equation}\label{Wnnegbound}
\lim_{N\rightarrow\infty} \mu\left(W_N\le 0,\, W_N > (1-\epsilon)\lambda_N \left(S_N(\one_{\Sigma_{\lambda_N}})-k(N)\right) \right)=0,\quad \forall \epsilon>0.
\end{equation}

Now define
\begin{equation*}
X_N^{(1)}=Z_N^{(1)}=\frac{S_N(\one_{\Sigma_{\lambda_N}})-k(N)}{\sqrt{k(N)}},\qquad
Y_N^{(1)}=\frac{W_N}{\sqrt{k(N)}\lambda_N},
\end{equation*}
and
\begin{equation*}
X_N^{(2)}=Y_N^{(2)}=Z_N^{(2)}=\frac{S_N(f_N)-N\E f_N}{\sqrt{k(N)}\lambda_N}.
\end{equation*}
Equations \eqref{Wntrivbound}, \eqref{Wnposbound} and \eqref{Wnnegbound} show that $X_N^{(1)}$, $Y_N^{(1)}$, $Z_N^{(1)}$ satisfy the hypotheses of Lemma~\ref{sandwichhighlem}.
Thus, once we prove
\begin{equation}\label{NmCm}
\left(X_N^{(1)},\,X_N^{(2)}\right)
\Rightarrow\left(\mathcal{N}(0,1),\,\mathcal{N}(0,\sigma^2)\right)
\quad \as{N},
\end{equation}
where the two limiting normal variables are independent and $\sigma^2=\frac{1}{2\alpha-1}$, Lemma~\ref{sandwichhighlem} will imply the corresponding convergence
\begin{equation*}
\left(Y_N^{(1)},\,Y_N^{(2)}\right)
\Rightarrow\left(\mathcal{N}(0,1),\,\mathcal{N}(0,\sigma^2)\right).
\end{equation*}
This is exactly the content of (A)--(C). Together with Lemma~\ref{CLTlem}, we can conclude the proof.

\begin{proof}[Proof of Theorem \ref{intertrimthm}]
Lemma~\ref{CLTlem} shows that
\begin{equation*}
\left( \frac{S_N(\one_{\Sigma_{\lambda_N}}) - k(N)}{\sqrt{k(N)}}, \frac{S_N(f_N) - N \E f_N}{\sqrt{k(N)} \lambda_N} \right) \Rightarrow (\mathcal{N}(0,1), \mathcal{N}(0,\sigma^2))
\end{equation*}
as $N\to\infty,$ where $(\mathcal{N}(0,1), \mathcal{N}(0,\sigma^2))$ is a vector of independent random variables, and
\begin{equation*}
\sigma^2=\frac{1}{2\alpha-1}.
\end{equation*}

By \eqref{Wntrivbound}, \eqref{Wnposbound} and \eqref{Wnnegbound}, we can apply Lemma \ref{sandwichhighlem} with
\begin{equation*}
X_N^{(1)}=Z_N^{(1)}=\frac{S_N(\one_{\Sigma_{\lambda_N}}) - k(N)}{\sqrt{k(N)}} \; \text{ and } \; Y_N^{(1)}=\frac{W_N}{\sqrt{k(N)} \lambda_N},
\end{equation*}
where $W_N$ is defined in \eqref{Wdef}, and
\begin{equation*}
X_N^{(2)}=Y_N^{(2)}=Z_N^{(2)}=\frac{S_N(f_N) - N \E f_N}{\sqrt{k(N)} \lambda_N},
\end{equation*}
and it follows that
\begin{equation*}
\left( \frac{W_N}{\sqrt{k(N)} \lambda_N}, \frac{S_N(f_N) - N \E f_N}{\sqrt{k(N)} \lambda_N} \right) \Rightarrow (\mathcal{N}(0,1), \mathcal{N}(0,\sigma^2)) \quad \as{N}.
\end{equation*}
This proves (A)--(C), and thus finishes the proof of convergence.
Finally, the claimed asymptotic expression for $a_N$ can be readily derived from the relation $a_N=N\mathbb{E}f_N$.
\end{proof}

\providecommand{\bysame}{\leavevmode\hbox to3em{\hrulefill}\thinspace}
\providecommand{\MR}{\relax\ifhmode\unskip\space\fi MR }
\providecommand{\MRhref}[2]{%
  \href{http://www.ams.org/mathscinet-getitem?mr=#1}{#2}
}
\providecommand{\href}[2]{#2}

\end{document}